\documentclass[10 pt]{article}
\usepackage{amsfonts}
\usepackage{amsbsy}
\usepackage{amssymb}
\usepackage{amsmath}
\usepackage{amsthm}
\usepackage{indentfirst}
\usepackage[pdftex]{graphicx}
\usepackage{subfigure}
\usepackage{enumitem}
\usepackage{epstopdf}
\usepackage{listings}
\usepackage{color}
\usepackage{courier}
\usepackage{algorithm}
\usepackage{graphicx}
\usepackage{verbatim}
\usepackage[mathcal]{eucal}
\usepackage{multicol}
\usepackage{float}
\usepackage{hyperref}
\usepackage{mathtools}

\newtheorem{definition}{Definition}
\newtheorem{theorem}{Theorem}
\newtheorem{lemma}{Lemma}
\newtheorem{corollary}{Corollary}
\newtheorem{remark}{Remark}
\newtheorem{proposition}{Proposition}
\providecommand{\keywords}[1]
{
	\small	
	\textbf{\textit{Keywords---}} #1
}

\title{Equivalence of the Dirichlet and Neumann problems for the Laplace operator in planar doubly-connected regions}

\author{Claudiu DINICU}
\date{}

\begin{document}
	
	\maketitle 
	
	\pagestyle{myheadings}
	\markboth{Claudiu Dinicu}{Equivalence in planar doubly-connected regions}

	\begin{abstract}
		Motivated by recent results regarding the equivalence of the Dirichlet and Neumann problems for the Laplace operator in the case of simply connected regions, the present paper takes a step further and provides a similar equivalence between the above mentioned problems  in the case of planar doubly-connected regions. The equivalence means that solving any of these problems leads by an explicit formula to a solution of the other problem. In addition, sufficient conditions guaranteeing the uniform H{\"o}lder continuity of the higher order partial derivatives of the solutions to the Neumann problem are provided as a consequence of this equivalence. \\
		
		 \\
		
		\textit{\textbf{Mathematics Subject Classification (2010) -}} 31B05, 31B10, 35J05, 42B37, 30D05, 30C20.  \\
	\end{abstract}
	
	
	\section{Introduction}
	
	The Dirichlet and Neumann problems are fundamental in the theory of
	differential equations, while still capturing the interest of the mathematical community. Some representations of the solutions of the Dirichlet and Neumann problems which can be found in the literature are: by single/double layer potential and spherical harmonics (see for instance \cite[Chapters 2, 3]{Folland2}) and by probabilistic methods (see \cite{Bass1} for the Dirichlet problem and \cite{Bass2} for the Neumann problem). 
	
	Recently the connection between these problems was investigated and it was shown that in the case of the Laplace operator (and other differential operators satisfying certain homogeneity conditions) there is an equivalence between these problems, in the sense that solving one of them leads by an explicit formula to a solution of the other problem (for details see \cite{Beznea1}, \cite{Beznea2}). Moreover it was shown that this equivalence leads to a new probabilistic representation of the solution of the Neumann problem, for the case of the (unit) ball (see \cite{Beznea 3}). The domains taken into consideration in these papers were simply connected. In the present paper the author takes a step further and shows that a similar equivalence between the Dirichlet and the Neumann problems holds in the case of planar doubly-connected regions. Incidentally it is shown that this intimate connection can be used for proving the uniform H{\"o}lder continuity of the higher order partial derivatives of the solutions of the Neumann problem. This aspect is in concordance with a well-known theorem in potential theory which is attributed to O. D. Kellogg, where the smooth extension of the higher order partial derivatives of the solution of the Dirichlet problem, under some appropriate smoothness assumption, was investigated (see \cite{Kellogg}). The contributions of this paper can be summarized as follows:
	\begin{itemize}
		\item First, an equivalence between the solutions of the Dirichlet and Neumann problems for the Laplace operator in annular regions $\boldsymbol A_{r_1;r_2}$ satisfying $0<r_1<r_2<\infty$, formulated in polar coordinates, is provided (see Theorem \ref{main theorem}). \\
		\item Second, sufficient conditions for the uniform H{\"o}lder continuity of the higher order partial derivatives of the solutions of the Neumann problem \eqref{Neumann problem biss} are provided as a consequence of Theorem \ref{main theorem} (see Theorem \ref{Holder continuity polar theorem}). \\
		\item A consistent definition of the Neumann problem in the special case of the punctured unit disk is given in Definition \ref{Neumann_punctured} using a polar coordinates representation (see also Remark \ref{Remark Neumann punctured}).
		\item When the boundary data is symmetric, a simplified formula for a solution of the Neumann problem \eqref{Neumann problem biss} in terms of the solution of the Dirichlet problem \eqref{Dirichlet problem biss} is presented (see Theorem \ref{simplified formula Neumann}).
		\item An equivalence between the solutions of the Neumann and Dirichlet problems \eqref{Neumann problem} and \eqref{Dirichlet problem}, respectively, as well as sufficient conditions for the uniform H{\"o}lder continuity of the higher order partial derivatives of the solutions of the Neumann problem \eqref{Neumann problem} are provided in Theorem \ref{main result cartesian}, in the case of annular regions whose radii satisfy the above mentioned condition.
		\item As a consequence, two main results of \cite{Beznea1} (namely Theorems $1$ and $3$) are recaptured in the case of $\mathbb{R}^2$ as particular instances in Corollary \ref{main corollary}. In addition Corollary \ref{main corollary} also provides sufficient conditions for the uniform H{\"o}lder continuity of the higher order partial derivatives of the solutions of the Neumann problem \ref{Neumann problem} for the case of the (unit) disk.   
		\item For the general case of planar, doubly-connected regions, Theorem \ref{thm in the case of doubly connected regions} gives an equivalence of the solutions of the Dirichlet and the Neumann problems \eqref{Dirichlet problem} and \eqref{Neumann problem}, respectively, while Theorem \ref{smooth exension proposition doubly-connected} provides sufficient conditions for the uniform H{\"o}lder continuity of the higher order partial derivatives of the solutions of the Neumann problem \eqref{Neumann problem}, as a consequence of this equivalence.
	\end{itemize}
	
	The structure of the paper is the following. In Section \ref{Section Preliminaries} the notation is established and some preparations are made. In Section \ref{Section Main results} the author presents all the results of the paper which were announced above. Section \ref{Conclusions} draws some final conclusions and some potential research directions. 
	
	\section{Preliminaries \label{Section Preliminaries}}
	
	\subsection{Notations \label{Secion Notations}}
	
	Denote by $\mathbb{U}=\left\{ z\in \mathbb{C}:\left\vert
	z\right\vert <1\right\} $ the unit disk, by $\dot{\mathbb{U}}$ the punctured unit disk, by $C_r$ the circle of radius $r$ (centered in origin), and the annulus with radii $0\leq r_{1}<r_{2}$ by $\boldsymbol A_{r_{1},r_{2}}=\left\{ z\in \mathbb{C}:r_{1}<\left\vert z\right\vert
	<r_{2}\right\} $, respectively. In addition, for any region $\boldsymbol D$, $C^1(\overline{\boldsymbol D})$ will stand for the set of all functions $h \in C^1(\boldsymbol D)$ for which the gradient $\nabla h$ can be continuously extended to $\overline{\boldsymbol D}$, and if $\boldsymbol D$ is in addition smooth and bounded, $\sigma$ will be denoting the length measure on its boundary. If $\boldsymbol E$ is a subset of $\mathbb{C}$ or $\mathbb{R}$, then the real or complex valued function $f$ belongs to $C^{m, \alpha}(\boldsymbol E)$, for some non-negative integer $m$ and some $\alpha \in (0,1]$, if the $m$ order partial derivatives of $f$ exist and are locally $\alpha$ H{\" o}lder continuous on $\boldsymbol E$ (in the case $f$ is complex-valued, the derivative should be understood in the sense of Complex Analysis). Also if $w$ is any function defined on some set $X$, then we define $\|w\| = \sup\limits_{x \in X}|w(x)|$. Throughout the paper the author will switch between the complex and the $\mathbb{R}^2$ notations, depending on the context to discriminate between them. For example if $u$ is a harmonic function defined on some region containing the point $(\cos\theta,\sin\theta)$ then $ u(e^{i\theta})$ is a shorthand representing the value of $u$ in $(\cos\theta,\sin\theta)$. Next if $u$ is differentiable at $z = (r\cos\theta,r\sin\theta)$, $r > 0$, then the directional derivative of $u$ in the direction of the unit vector $(\cos\theta,\sin\theta)$ evaluated in $z$ will be denoted $\frac{\partial u}{\partial \boldsymbol a_{\theta}}(z)$. If $h_1$ and $h_2$ are any two functions then $h_1 \sim h_2$, $h_1 \lesssim h_2$ provided that there is some constant $C$ such that $h_1 = Ch_2$, and $h_1 \le Ch_2$, respectively. In the last case any constant $C$ with that property will be referred to as a \textit{proportionality constant}. Last but not least if $z$ is any complex number we will let $x$ denote its real part, $x = \Re(z)$, and $y$ denote its imaginary part, $y = \Im(z)$ (unless otherwise specified).
	
	\subsection{Definitions and preliminary aspects}
	
	Throughout the paper $\boldsymbol D \subset \mathbb{C}$ will denote a bounded, doubly-connected region of the complex plane.
	\begin{definition}\label{definition smooth region}
		We say that $\boldsymbol D \in C^{m,\alpha}$ for some non-negative integer $m$ and some real $\alpha \in (0,1]$ if for any point $a \in \partial \boldsymbol D$ there exists an open interval $\boldsymbol I$ and a real-valued function $\beta: \boldsymbol I \rightarrow (c,d)$ satisfying $\beta \in C^{m,\alpha}(\boldsymbol I)$, such that the set $\boldsymbol U := \boldsymbol I \times (c,d)$ is a neighborhood of $a$ and, eventually up to a rotation
		\begin{equation}
		\boldsymbol D \cap \boldsymbol U = \{(x,t): x \in \boldsymbol I, ~ t < \beta(x)\}.
		\end{equation}  
	\end{definition} 
	Notice that $\boldsymbol I$, $c$, $d$, as well as the function $\beta$ may depend on $a$.
	
	\begin{definition}\label{definition smooth boundary value}
		Let $f$ be a real-valued function defined on the boundary of $\boldsymbol D$ and assume $\boldsymbol D \in C^{m,\alpha}$. Then $f \in C^{m,\alpha}(\partial \boldsymbol D)$ if for any point $a \in \partial \boldsymbol D$ and any local parametrization $\beta$ as above the $m$ order derivative of the function $f \circ \beta$ is locally $\alpha$ H{\"o}lder continuous on $\boldsymbol I$. 
	\end{definition}
	
	If $\boldsymbol D$ is also smooth, consider the corresponding Dirichlet and Neumann problems for the Laplace operator
	\begin{equation}\label{Dirichlet problem}
	\begin{cases}
	\Delta u=0 &\text{ in } \boldsymbol D, \\
	u=g &\text{ on }\partial \boldsymbol D,
	\end{cases}
	\end{equation}
	and
	\begin{equation} \label{Neumann problem}
	\begin{cases}
	\Delta U=0 &\text{ in } \boldsymbol D, \\
	\frac{\partial U}{\partial \boldsymbol \nu }=f &\text{ on }\partial \boldsymbol D, 
	\end{cases}
	\end{equation}
	respectively, where $\boldsymbol \nu$ is the unitary outward normal to the boundary of $\boldsymbol D$. In the particular case when $\boldsymbol D = \boldsymbol A_{r_1;r_2}, ~r_1 > 0$, we have
	\begin{equation}
	\boldsymbol \nu \left( z\right) =
	\begin{cases}
	\frac{z}{r_2}, &\text{if } |z|=r_{2}, \\
	-\frac{z}{r_1}, &\text{if } |z|=r_{1}.
	\end{cases}
	\end{equation}
	
	By a solution of the Dirichlet/Neumann problems above, it is understood a function $u\in C^{2}(\boldsymbol D) \cap C^0(\overline{\boldsymbol D}) $, respectively $U \in C^{2}(\boldsymbol D) \cap C^1(\overline{\boldsymbol D})$, which satisfies (\ref{Dirichlet problem}), respectively (\ref{Neumann problem}).
	
	\begin{remark} Assume $\boldsymbol D$ is the annulus $\boldsymbol A_{r_1,r_2}, ~r_1 > 0$. Using the maximum principle for harmonic functions (see, e.g. \cite[Theorem
		2.2.4]{Evans}), it can be seen that for continuous boundary data $g$ the
		Dirichlet problem (\ref{Dirichlet problem}) has a unique solution. Also if $f$ is a continuous function satisfying $\int\limits_{\partial \boldsymbol D} f d\sigma = 0$ then it can be argued that the (classical) Neumann problem always has a solution, which is unique up to additive constants.
		
		The existence of solutions of the Dirichlet and the Neumann problems in the
		case of the punctured disk $\boldsymbol A_{0,r_{2}}$ requires special attention. As
		shown by Zaremba's example, for continuous boundary data $g$ and $r_{1}=0$, the Dirichlet problem (\ref{Dirichlet problem}) has a solution iff $g\left( 0\right) =\frac{1}{2\pi r_{2}}%
		\int_{0}^{2\pi }g\left( r_{2}e^{i\theta }\right) d\theta $. Also, for
		continuous boundary data $f$ and $r_{1}=0$, the boundary condition at the origin of the
		Neumann problem (\ref{Neumann problem}) should be ignored (the exterior
		normal to $\partial \boldsymbol A_{0,r_{2}}$ at the origin cannot be properly defined), and a solution of (\ref{Neumann problem}) satisfying the boundary condition just on $\partial \boldsymbol A_{0,r_{2}}\backslash \left\{ 0\right\} $ exists only if $%
		\int_{0}^{2\pi }f\left( r_2e^{i\theta }\right) d\theta =0$. 
	\end{remark}
	
	When $\boldsymbol D = \boldsymbol A_{r_1,r_2}$, due to the radial symmetry of the region, it is natural to consider polar coordinates $\left( r,\theta \right) $, defined by $r=\left\vert z\right\vert $ and $\theta \in Arg(z) := \{\arg(z) + 2k\pi : k \in \mathbb{Z}\}$ for $z\in \boldsymbol A_{r_{1},r_{2}}$. Here $arg(z)$ denotes the principal argument of $z \neq 0$, and this notation will be in force for the rest of the paper.
	
	The link between the cartesian and polar coordinates formulation of the
	Dirichlet and Neumann problems (\ref{Dirichlet problem}) -- (\ref{Neumann
		problem}), when $\boldsymbol D = \boldsymbol A_{r_1,r_2}$, is given by the following proposition.
	
	\begin{proposition} \label{equiv formulation in cartesian and polar coordinates} If $w\in
		C^{2}\left( \boldsymbol A_{r_{1}, r_2}\right) $ satisfies $\Delta w=0$ in $ \boldsymbol A_{r_{1},r_{2}}$, then the function $\hat{w}:(r_{1},r_{2})\times \mathbb{%
			R\rightarrow R}$ defined by $\hat{w}\left( r,\theta \right) =w\left(
		re^{i\theta }\right) $ is $2\pi $-periodic in the second variable, has
		continuous second order partial derivatives and satisfies%
		\begin{equation}\label{laplace eq in polar coordinates}
		\hat{w}_{\boldsymbol r \boldsymbol r}+\frac{1}{r}\hat{w}_{\boldsymbol r}+\frac{1}{r^{2}}\hat{w}_{\boldsymbol \theta \boldsymbol \theta
		} = 0\qquad \text{in }(r_{1},r_{2})\times \mathbb{R}.
		\end{equation}
		
		Conversely, if the function $\hat{w}:(r_{1},r_{2})\times \mathbb{%
			R\rightarrow R}$ is $2\pi $-periodic in the second variable, has continuous
		second order partial derivatives and satisfies (\ref{laplace eq in polar
			coordinates}), then the function $w: \boldsymbol A_{r_{1},r_{2}} \rightarrow \mathbb{R}$ defined by $w\left( z\right) = \hat{w}\left( \left\vert z\right\vert ,\arg(z) \right) $ belongs to $C^{2}\left( \boldsymbol A_{r_{1},r_{2}}\right) $ and satisfies $\Delta w=0$ in $\boldsymbol A_{r_{1},r_{2}}$.
		
		Moreover, $w$ has a continuous extension to $\overline{\boldsymbol A_{r_{1},r_{2}}}$
		if and only if $\hat{w}$ has a continuous extension to $\left[ r_{1},r_{2}\right]
		\times \mathbb{R}$ as well, and in this case
		\begin{equation*}
		w( re^{i\theta }) =\hat{w}\left( r,\theta \right) \quad \text{for
		}\left( r,\theta \right) \in \left[ r_{1},r_{2}\right] \times \mathbb{R}.
		\end{equation*}
		
		Also $w$ has (outer) normal derivative at a point $re^{i\theta}\in
		\partial \boldsymbol A_{r_{1},r_{2}}$ if and only if $\hat{w}$ has partial derivative with
		respect to the first variable at the point $\left( r,\theta \right)
		\in \left\{ r_{1},r_{2}\right\} \times \mathbb{R}$, and in this case%
		\begin{equation}\label{normal-derivatives}
		\frac{\partial w}{\partial \boldsymbol \nu }\left( re^{i\theta}\right) =\left\{
		\begin{array}{cc}
		\hat{w}_{\boldsymbol r}\left( r,\theta \right) , & \text{if }r=r_{2}, \\
		-\hat{w}_{\boldsymbol r}\left( r,\theta \right) , & \text{if }r=r_{1}.%
		\end{array}%
		\right. 
		\end{equation}
		
		Finally $w \in C^1\left(\overline{\boldsymbol A_{r_1;r_2}}\right)$ if and only if $\hat{w} \in C^1\left([r_1,r_2] \times \mathbb{R}\right)$.
	\end{proposition}
	
	\begin{proof}
		The direct implication is immediate. For the converse, by using the $2\pi$-periodicity of $\hat{w}$ in the second variable and the fact that it has continuous second order partial derivatives, direct computations show that $w \in C^{2}\left( \boldsymbol A_{r_{1},r_{2}}\right)$. Also, it is not difficult to check that $$\Delta w(z)= \hat{w}_{rr}(|z|,\arg(z)) + \frac{1}{|z|} \hat{w}_r(|z|,\arg(z))+\frac{1}{|z|^2}\hat{w}_{\theta\theta}(|z|,\arg(z)) = 0,~~ z \in \boldsymbol A_{r_{1},r_{2}},$$
		where the last equality follows by using hypothesis (\ref{laplace eq in polar coordinates}). 
		
		The fact that $w$ has a continuous extension to the boundary of the domain if and only if $\hat{w}$ has is immediate. 
		
		Next notice that for any $\theta \in \mathbb{R}$ the corresponding directional derivatives are given by: \\
		
		$\begin{cases} \frac{\partial w}{\partial \boldsymbol a_{\theta}}(re^{i\theta}) &= \lim\limits_{t \rightarrow 0} \frac{w(r(cos\theta,sin\theta)+t(cos\theta,sin\theta)) - w(r(cos\theta,sin\theta))}{t} = \hat{w}_{\boldsymbol r}(r,\theta), \\
		\frac{\partial w}{\partial \boldsymbol \nu}(r_2e^{i\theta}) &= \lim\limits_{t \nearrow 0} \frac{w(r_2(cos\theta,sin\theta)+t(cos\theta,sin\theta)) - w(r_2(cos\theta,sin\theta))}{t} = \hat{w}_{\boldsymbol r}(r_2,\theta), \\
		\frac{\partial w}{\partial \boldsymbol \nu}(r_1e^{i\theta}) &= -\lim\limits_{t \searrow 0} \frac{w(r_1(cos\theta,sin\theta)+t(cos\theta,sin\theta)) - w(r_1(cos\theta,sin\theta))}{t} = -\hat{w}_{\boldsymbol r}(r_1,\theta), \end{cases}$ \\
		where the first and second equalities above hold for any $r\in (r_1,r_2)$.
		
		For the last claim it can also be checked by direct computations that the following equalities hold:
		\begin{equation}\label{lengthy computations}
		\begin{cases}
		w_{\boldsymbol x}\left(re^{i\theta}\right) = -\frac{r \sin\theta}{r^2}\hat{w}_{\boldsymbol \theta}(r,\theta) + \frac{r \cos\theta}{r}\hat{w}_{\boldsymbol r}(r,\theta) &~\text{if } ~\cos\theta <0, ~\sin\theta \neq 0, \\
		w_{\boldsymbol y}\left(re^{i\theta}\right) = \frac{r \cos\theta}{r^2}\hat{w}_{\boldsymbol \theta}(r,\theta) + \frac{r \sin\theta}{r}\hat{w}_{\boldsymbol r}(r,\theta) &~\text{if } \cos\theta <0, ~\sin\theta \neq 0, \\
		w_{\boldsymbol x}\left(re^{-i\pi}\right) = -\hat{w}_{\boldsymbol r}(-r,-\pi), \\
		w_{\boldsymbol y}\left(re^{-i\pi}\right) = -\frac{\hat{w}_{\boldsymbol \theta}(-r,-\pi)}{r}.
		\end{cases}
		\end{equation}
		Combining equations \eqref{lengthy computations} with the fact that $\arg(\cdot)$ is harmonic in $\mathbb{C} \setminus \{z: \Re(z) < 0, ~ \Im(z) = 0\}$ (and hence $\arg(\cdot) \in C^2 \left(\mathbb{C} \setminus \{z: \Re(z) < 0, ~ \Im(z) = 0\} \right)$) and $\hat{w}_{\boldsymbol r}$, $\hat{w}_{\boldsymbol \theta}$ are $2\pi$-periodic in the second variable, the conclusion follows. This ends the proof.
	\end{proof}
	
	The above result shows that in the case of annuli one can reformulate the Dirichlet and the Neumann problems (\ref{Dirichlet problem}) -- (\ref{Neumann problem}) in polar coordinates as follows: find $u=u\left( r,\theta
	\right) \in C^{2}\left( \left( r_{1},r_{2}\right) \times \mathbb{R}\right)
	\cap C^0 \left( \left[ r_{1},r_{2}\right] \times \mathbb{R}\right) $ which is $2\pi$-periodic in the second variable and satisfies
	\begin{equation}\label{Dirichlet problem biss}
	\begin{cases}
	u_{\boldsymbol r \boldsymbol r}+\frac{1}{r}u_{\boldsymbol r}+\frac{1}{r^{2}}u_{\boldsymbol \theta \boldsymbol \theta }=0 &~\text{in }
	(r_{1},r_{2})\times \mathbb{R}, \\
	u=\varphi &~\text{on } \left\{ r_{1},r_{2}\right\} \times \mathbb{R},
	\end{cases} 
	\end{equation}
	respectively find $U=U\left( r,\theta \right) \in C^2\left((r_1,r_2)\times \mathbb{R}\right) \cap C^1([r_1,r_2]\times \mathbb{R}) $ which is $2\pi $-periodic in the second variable and satisfies
	\begin{equation}\label{Neumann problem biss}
	\begin{cases}
	U_{\boldsymbol r \boldsymbol r}+\frac{1}{r}U_{\boldsymbol r}+\frac{1}{r^{2}}U_{\boldsymbol \theta \boldsymbol \theta }=0 &~\text{in }
	(r_{1},r_{2})\times \mathbb{R}, \\
	U_{\boldsymbol r}=\phi &~\text{on }\left\{ r_{1},r_{2}\right\} \times \mathbb{R}.
	\end{cases}
	\end{equation}
	and the boundary data $\varphi ,\phi :\left\{ r_{1},r_{2}\right\} \times
	\mathbb{R}$ is related to the boundary data \\
	$f,g:\partial A_{r_{1},r_{2}}\rightarrow \mathbb{R}$ in (\ref{Dirichlet problem}) -- (\ref{Neumann problem}) by
	\begin{equation*}
	\varphi (r,\theta) =g( re^{i\theta }) \quad \text{%
		and\quad }\phi (r,\theta) =\left\{
	\begin{array}{cc}
	f( re^{i\theta })  & \text{if }r=r_{2}, \\
	-f( re^{i\theta })  & \text{if }r=r_{1}.
	\end{array}
	\right.
	\end{equation*}%
	Notice that, in particular, the functions $\varphi ,\phi $ are $2\pi $%
	-periodic in the second variable.
	
	\begin{remark}
		The compatibility condition $\int_{\partial \boldsymbol A_{r_11,r_2}}f~d\sigma=0$ for the existence of a
		solution of the Neumann problem (\ref{Neumann problem}) in cartesian
		coordinates becomes, in polar coordinates, the following:%
		\begin{equation}
		\int_{0}^{2\pi }r_1\phi \left( r_{1},\theta \right) d\theta =\int_{0}^{2\pi
		}r_2\phi \left( r_{2},\theta \right) d\theta .
		\end{equation}
	\end{remark}
	
	\section{Main results\label{Section Main results}}
	
	This section is divided into two parts: Subsection $3.1$ is devoted to the study of the equivalence between the solutions of the Dirichlet and Neumann problems in the case of annular regions, while Subsection $3.2$ is devoted to the study of the equivalence of these two problems for general doubly connected regions.
	\subsection{Annular regions}
	
	At this point we are prepared to state and prove the main result of this section.
	
	\begin{theorem}\label{main theorem}
		Let $0<r_{1}<r_{2}<\infty$ and assume $\phi :\left\{r_{1},r_{2}\right\} \times \mathbb{R}\rightarrow \mathbb{R}$ is continuous, $2\pi $-periodic in the second variable, and satisfies the compatibility condition $\int\limits_{0}^{2\pi
		}r_1\phi \left( r_{1},\theta \right) d\theta =\int\limits_{0}^{2\pi }r_2\phi
		\left( r_{2},\theta \right) d\theta $. If $U$ is the solution of the Neumann problem (\ref{Neumann problem biss}) with boundary data $\phi$, satisfying
		$U(\sqrt{r_{1}r_{2}},0) =0$, then for any $\left(r,\theta\right) \in \left[r_{1},r_{2}\right] \times \mathbb{R}$
		\begin{equation}\label{U definition Th1}
		U(r,\theta )=\int\limits_{\frac{\sqrt{r_{1}r_{2}}}{r}}^{1}\frac{u(r\rho,\theta )}{\rho }d\rho +\sqrt{r_{1}r_{2}}\int\limits_{0}^{\theta }\left( \mathcal{C} - \int\limits_{0}^{t}u_{\boldsymbol r}(\sqrt{r_{1}r_{2}},\tau )d\tau \right)dt,
		\end{equation}
		where $u$ is the solution of the Dirichlet problem (\ref{Dirichlet problem biss}%
		) with boundary values $\varphi(r,\theta) = r\phi(r,\theta)$ on $\left\{ r_{1},r_{2}\right\} \times \mathbb{R}$ and
		\begin{equation}\label{first appearance of C}
		\mathcal{C}=\frac{\sqrt{r_{1}r_{2}}}{2\pi }\int\limits_{0}^{2\pi
		}\int\limits_{0}^{t}u_{\boldsymbol r}(\sqrt{r_{1}r_{2}},\tau )d\tau dt.
		\end{equation}
		Conversely if $\varphi :\left\{ r_{1},r_{2}\right\} \times \mathbb{R}%
		\rightarrow \mathbb{R}$ is continuous, $2\pi $-periodic in the second
		variable, satisfies $\int\limits_{0}^{2\pi }\varphi \left( r_{1},\theta
		\right) d\theta =\int\limits_{0}^{2\pi }\varphi \left( r_{2},\theta \right)
		d\theta $, and if $U$ is a solution of the Neumann problem (\ref{Neumann
			problem biss}) with $\phi \left( r,\theta \right) =\frac{\varphi \left( r,\theta
			\right)}{r} $ for $\left( r,\theta \right) \in \left\{ r_{1},r_{2}\right\}
		\times \mathbb{R}$, then
		\begin{equation}\label{u Thm1}
		u\left( r,\theta \right) =rU_{\boldsymbol r}\left( r,\theta \right) ,\qquad \left(
		r,\theta \right) \in \left[ r_{1},r_{2}\right] \times \mathbb{R},
		\end{equation}
		is the solution of the Dirichlet problem (\ref{Dirichlet problem biss}).
	\end{theorem}
	
	\begin{proof}
		Denote by $\mathcal{U}$ the right-hand side of \eqref{U definition Th1}. Let us first consider that $r_2=\frac{1}{r_1}=a>1$, in which case the problem
		reduces to showing that the function
		\begin{equation}
		\mathcal{U}(r,\theta) = \int\limits_{\frac{1}{r}}^1 \frac{u(r\rho,\theta)}{\rho}d\rho
		+ \int\limits_0^{\theta}\left( \mathcal{C} - \int\limits_0^t u_{\boldsymbol r}(1,\tau)d\tau \right)dt,
		\end{equation}
		satisfying $\mathcal{U}(1,0)=0$ is the desired solution of the Neumann problem \eqref{Neumann problem biss} on $\boldsymbol A_{\frac{1}{a};a}$ with boundary data 
		\begin{equation*}
		\phi (r,\theta) =
		\begin{cases}
		f( re^{i\theta }) &~\text{if } r=a, ~\theta \in \mathbb{R}, \\
		-f( re^{i\theta }) &~\text{if } r=\frac{1}{a}, ~\theta \in \mathbb{R},
		\end{cases}
		\end{equation*}
		where $u$ is the solution of the Dirichlet problem \eqref{Dirichlet problem biss} with boundary values $\varphi(r,\theta)=r\phi(r,\theta)$ on $\left\{ \frac{1}{a},a\right\}
		\times \mathbb{R}$. To this end we will first show that 
		\begin{equation}\label{eq1 Th1}
		\int\limits_0^{2\pi}u_{\boldsymbol r}(1,\tau)d\tau = 0.
		\end{equation}
		Indeed using the definition of $\mathcal{U}$ we get $\mathcal{U}_{\boldsymbol r}(r,\theta) = \frac{\partial}{\partial r}\int\limits_{
			\frac{1}{r}}^1 \frac{u(r\rho,\theta)}{\rho}d\rho = $ \\
		$\frac{\partial}{\partial r
		}\int\limits_1^r \frac{u(\rho,\theta)}{\rho}d\rho = \frac{u(r,\theta)}{r}, ~ (r,\theta) \in \left(\frac{1}{a},a\right) \times \mathbb{R}$. Consequently, the partial derivative of $\mathcal{U}$ with respect to the first variable can be continuously extended to $\left[\frac{1}{a},a\right] \times \mathbb{R}$ and thus for any $\theta \in \mathbb{R}$ one obtains $%
		\mathcal{U}_{\boldsymbol r}(a,\theta)= \frac{u(a,\theta)}{a}= \frac{\varphi(a,\theta)}{a}=\phi(a,\theta)$ and also $\mathcal{U}_{\boldsymbol r}(\frac{1}{a},\theta) = \frac{u(\frac{1}{a},\theta)%
		}{\frac{1}{a}}= \frac{\varphi(\frac{1}{a},\theta)%
		}{\frac{1}{a}} = \phi(\frac{1}{a},\theta)$. Define $W: \boldsymbol A_{\frac{1}{a};a} \rightarrow \mathbb{R}, ~ W(z) := u(|z|,\arg(z))$. Since $u$ is the solution of the Dirichlet problem \eqref{Dirichlet problem biss} it follows by Proposition \ref{equiv formulation in cartesian and polar coordinates} that $W$ is harmonic in $\boldsymbol A_{\frac{1}{a};a}$ and using a continuity argument $W(re^{i\theta})=\varphi(r,\theta), ~\forall (r,\theta) \in \{\frac{1}{a},a\}$. Then there exist real constants $\alpha,\beta \in \mathbb{R}$ such that $\int\limits_0^{2\pi}W(re^{i\theta})d\theta = \alpha \log r + \beta, ~\forall r \in \left[\frac{1}{a},a\right]$ (see \cite[Chapter 4, Theorem 20]{Ahlfors}). But then $-\alpha \log a + \beta = \int\limits_{C_{\frac{1}{a}}}W\left(\frac{1}{a}e^{i\theta}\right)d\theta =
		\int\limits_0^{2\pi}u\left(\frac{1}{a},\theta\right)d\theta =
		\int\limits_0^{2\pi}\frac{1}{a}\phi\left(\frac{1}{a},\theta\right)d\theta = \int\limits_0^{2\pi}a\phi(a,\theta) d\theta = \int\limits_{C_a}W\left(ae^{i\theta}\right) = \alpha \log a + \beta$ which implies $\alpha = 0$. To sum up $\int\limits_{C_r}W\left(re^{i\theta}\right)d\theta = \int\limits_0^{2\pi}u(r,\theta)d\theta$ is a constant function of $r$. Taking the
		derivative it follows that $\frac{d}{dr}\int\limits_0^{2\pi}u(r,\theta)d\theta = 0$. Since $1 \in \left(\frac{1}{a},a\right)$, an application of the \textit{Dominant Convergence} theorem together with the above identity concludes the proof of \eqref{eq1 Th1}. \\
		
		The next step is to show that whenever $(r,\theta) \in \left(\frac{1}{a},a\right) \times \mathbb{R}$ 
		\begin{equation}\label{eq2 Th1}
		\mathcal{U}(r,\theta + 2\pi) = \mathcal{U}(r,\theta).
		\end{equation}
		To this end compute $\mathcal{U}(r,\theta + 2\pi) = \int\limits_{\frac{1}{r}}^1 \frac{u(r\rho,\theta + 2\pi)}{\rho}d\rho - \int\limits_0^{\theta +
			2\pi}\int\limits_0^t u_{\boldsymbol r}(1,\tau)d\tau dt + \mathcal{C}(\theta + 2\pi) = \int\limits_{\frac{1}{r}}^1 \frac{u(r\rho,\theta)}{\rho}d\rho -
		\int\limits_0^{\theta}\int\limits_0^t u_{\boldsymbol r}(1,\tau)d\tau + \mathcal{C} \theta - \int\limits_{\theta}^{\theta+2\pi}\int\limits_0^tu_{\boldsymbol r}(1,\theta)d\tau dt + 2\pi \mathcal{C}$. As $u(r,\theta)=u(r,\theta + 2\pi)$, $u_{\boldsymbol r}(r_0,\theta + 2\pi)= \lim\limits_{r \rightarrow r_0} \frac{u(r,\theta + 2\pi)-u(r_0,\theta+2\pi)}{r-r_0} = \lim\limits_{r \rightarrow r_0} \frac{u(r,\theta)-u(r_0,\theta)}{r-r_0} = u_{\boldsymbol r}(r_0,\theta) ~\forall~ (r_0,\theta)$.
		Thus, $u_{\boldsymbol r}(1,\cdot)$ is $2\pi$-periodic and so $0 = \int\limits_t^{t+2\pi}u_{\boldsymbol r}(1,%
		\tau)d\tau =$ \\
		$\int\limits_0^{2\pi}u_{\boldsymbol r}(1,\tau)d\tau$. Consequently it follows that the function $t \rightarrow \int\limits_0^t
		u_{\boldsymbol r}(1,\tau)d\tau$ is $2\pi$-periodic, showing in turn that
		\begin{equation*}
		\int\limits_{\theta}^{\theta+2\pi}\int\limits_0^tu_{\boldsymbol r}(1,\tau)d\tau dt = \int\limits_0^{2\pi}\int\limits_0^tu_{\boldsymbol r}(1,\tau)d\tau dt = 2\pi \mathcal{C},
		\end{equation*}
		which proves relation \eqref{eq2 Th1}. \\
		
		Proceeding further we need to show that $\mathcal{U}$ satisfies \eqref{laplace eq in polar coordinates} for any pair $(r,\theta)$ in $\left(\frac{1}{a},a\right) \times \mathbb{R}$. But this follows easily using the Leibniz-Newton formula in the definition \eqref{U definition Th1} of $\mathcal{U}$, which gives $\mathcal{U}_{\boldsymbol r}(r,\theta) = \frac{u(1,\theta)}{r} + \int\limits_{\frac{1}{r}}^1 u_{\boldsymbol r}(r\rho,\theta)d\rho$, $\mathcal{U}_{\boldsymbol r \boldsymbol r}(r,\theta) = -\frac{u(1,\theta)}{r^2}+\frac{u_{\boldsymbol r}(1,\theta)}{r^2}+\int\limits_{\frac{1}{r}}^1 \rho u_{\boldsymbol r \boldsymbol r}(r\rho,\theta)d\rho$, and also $\mathcal{U}_{\boldsymbol \theta \boldsymbol \theta}(r,\theta) = -u_{\boldsymbol r}(1,\theta) + \int\limits_{\frac{1}{r}}^1 \frac{u_{\boldsymbol \theta \boldsymbol \theta}(r\rho,\theta)}{\rho}d\rho$, whenever $(r,\theta) \in \left(\frac{1}{a},a\right) \times \mathbb{R}$. Adding them up one obtains $\mathcal{U}_{\boldsymbol r \boldsymbol r}(r,\theta)+\frac{1}{r}\mathcal{U}_{\boldsymbol r}(r,\theta) + \frac{1}{r^2}\mathcal{U}_{\boldsymbol \theta \boldsymbol \theta}(r,\theta) = \int\limits_{\frac{1}{r}}^1\rho\left(u_{\boldsymbol r \boldsymbol r}(r\rho,\theta)+\frac{1}{r\rho}u_{\boldsymbol r}(r\rho,\theta)+\frac{1}{r^2\rho^2}u_{\boldsymbol \theta \boldsymbol \theta}(r\rho,\theta)\right)d\rho$, where the quantity on the right-hand side is identically $0$ since $u$ verifies relation \eqref{laplace eq in polar coordinates}. Let us now show that the derivative of $\mathcal{U}$ with respect to the first argument exists, is finite, and equals $\phi(r,\theta)$ at all points $(r,\theta) \in \{r_1,r_2\} \times \mathbb{R}$. Indeed $\lim\limits_{r \nearrow a}\frac{\mathcal{U}(r,\theta)-\mathcal{U}(a,\theta)}{r-a}= \phi(a,\theta)$, and likewise $\lim\limits_{r \searrow \frac{1}{a}}\frac{\mathcal{U}(r,\theta)-\mathcal{U}( \frac{1}{a},\theta)}{r-\frac{1}{a}}= \phi( \frac{1}{a},\theta)$. It only remains to be proved that the partial derivatives of $\mathcal{U}$ extend continuously to $[r_1,r_2]\times \mathbb{R}$. To see that this is indeed the case define $\tilde{\mathcal{U}}(re^{i\theta}) = \mathcal{U}(r,\theta), ~(r,\theta) \in (r_1,r_2) \times \mathbb{R}$, and also $f(re^{i\theta}) = \begin{cases} \phi(r,\theta) ~&\text{if } r=r_2, \\ -\phi(r,\theta) ~&\text{if } r=r_1. \end{cases}$ Thus, using Proposition \ref{equiv formulation in cartesian and polar coordinates}, it can be easily seen that $\tilde{\mathcal{U}}$ is harmonic on $\boldsymbol A_{r_1,r_2}$ and that the directional derivative of $\tilde{\mathcal{U}}$ along any ray is $f$. Let $V$ be any solution of the Neumann problem \eqref{Neumann problem} on $\boldsymbol A_{r_1,r_2}$ having boundary data $f$. It will be proved that $W := \tilde{\mathcal{U}}-V$ is constant on $\boldsymbol A_{r_1,r_2}$. Indeed let $r_{1;n}$ and $r_{2;n}$ be two sequences with positive terms such that $r_{n;1}$ is decreasing, $r_{n;1} \rightarrow r_1$, and $r_{n;2}$ is increasing, $r_{n;2} \rightarrow r_2$, respectively and denote $\boldsymbol A_n = \boldsymbol A_{r_{1;n},r_{2;n}}$. According to \textit{Green's first identity} applied to $W$ on $\boldsymbol A_n$ it follows that $\int\limits_{\boldsymbol A_n} \left( W\Delta W + \|\nabla W\|^2\right) dm = \int\limits_{\partial \boldsymbol A_n}\frac{\partial W}{\partial \boldsymbol \nu} d\sigma$ and since $\Delta W=0$ on $\boldsymbol A_n$ 
		\begin{equation}\label{Green's first identity Th1}
		\int\limits_{\boldsymbol A_n} \|\nabla W\|^2 dm = \int\limits_{\partial \boldsymbol A_n} \frac{\partial W}{\partial \boldsymbol \nu}d\sigma,
		\end{equation} 
		where $m$ is the Lebesgue measure. Since the sets $\boldsymbol A_n$ increase to $\boldsymbol A_{r_1,r_2}$, the sequence of non-negative real-valued functions $1_{\boldsymbol A_n}\|\nabla W\|^2$ increases to the function $1_{\boldsymbol A_{r_1,r_2}}\|\nabla W\|^2$ (where $1_{\boldsymbol E}$ is the indicator function of the set $\boldsymbol E \subset \mathbb{C}$) and hence an application of the \textit{Monotone Convergence} theorem to the left-hand side of \eqref{Green's first identity Th1} gives
		\begin{equation}\label{Green's first identity Th1 continuare}
		\int\limits_{\boldsymbol A_{r_1,r_2}} \|\nabla W\|^2 dm = \lim\limits_{n \rightarrow \infty}\int\limits_{\boldsymbol A_n} \|\nabla W\|^2 dm = \lim\limits_{n \rightarrow \infty}\int\limits_{\partial \boldsymbol A_n} \frac{\partial W}{\partial \boldsymbol \nu}d\sigma.
		\end{equation}
		On the other hand one can notice that on $\partial \boldsymbol A_n$ the normal derivative of $W$ is given by $\frac{\partial W}{\partial \boldsymbol \nu} = \frac{\partial \tilde{\mathcal{U}}}{\partial \boldsymbol \nu} - \frac{\partial V}{\partial \boldsymbol \nu} = \frac{\partial \tilde{\mathcal{U}}}{\partial \boldsymbol \nu} - \langle \nabla V;\boldsymbol \nu \rangle$. By definition $\nabla V$ extends continuously to $\overline{\boldsymbol A_{r_1,r_2}}$. Also, in view of Proposition \ref{equiv formulation in cartesian and polar coordinates} $\frac{\partial \tilde{\mathcal{U}}}{\partial \boldsymbol \nu}(re^{i\theta}) = \begin{cases}\mathcal{U}_{\boldsymbol r}(re^{i\theta}) ~&\text{if } r=r_{n;2} \\ -\mathcal{U}_{\boldsymbol r}(re^{i\theta}) ~&\text{if } r=r_{n;1}\end{cases}$ and since it has been already shown that $\mathcal{U}_{\boldsymbol r}(r,\theta) = \frac{u(r,\theta)}{r}$ it follows that $\frac{\partial \tilde{\mathcal{U}}}{\partial \boldsymbol \nu}$ extends continuously to $\boldsymbol A_{r_1,r_2}$ as well and one can conclude that $\frac{\partial W}{\partial \boldsymbol \nu}$ is bounded on $\boldsymbol A_{r_1,r_2}$. Consequently using the \textit{Dominant Convergence} theorem in \eqref{Green's first identity Th1 continuare} it follows that $\nabla W=0$ in $\boldsymbol A_{r_1,r_2}$ and hence $\tilde{\mathcal{U}} \in C^1(\overline{\boldsymbol A_{r_1,r_2}})$. Invoking again Proposition \eqref{equiv formulation in cartesian and polar coordinates} shows that $\mathcal{U} \in C^1([r_1,r_2]\times\mathbb{R})$, as desired. This completes the proof of the first part in the case $r_2=a > 1 > \frac{1}{a}=r_1$. \\
		
		For the general case $0 < r_1 < r_2$ define $\lambda = \frac{1}{\sqrt{r_1r_2}}$, $a = \sqrt{\frac{r_2}{r_1}}$ and let $\hat{u}$ be the solution of the Dirichlet problem \eqref{Dirichlet problem biss} on $\boldsymbol A_{\frac{1}{a};a}$ with boundary data $\hat{\varphi}(r,\theta)=\varphi(\frac{r}{\lambda},\theta)=\frac{r}{\lambda}\phi
		(\frac{r}{\lambda},\theta), ~(r,\theta) \in \{\frac{1}{a},a\} \times \mathbb{R}$. By the previous part the function $\hat{\mathcal{U}}(r,\theta) := \int\limits_{\frac{1}{r}}^1 \frac{\hat{u}(r\rho,\theta)}{\rho}d\rho
		+ \int\limits_0^{\theta}\left( \mathcal{C} - \int\limits_0^t \hat{u}_{\boldsymbol r}(1,\tau)d\tau \right)dt$ is the solution of the Neumann problem \eqref{Neumann problem biss} with boundary data $\hat{\phi}(r,\theta)=\frac{\hat{\varphi}(r,\theta)}{r}$ on $\{\frac{1}{a},a\} \times \mathbb{R}$, satisfying $\hat{\mathcal{U}}(1,0)=0$. Consequently defining $\mathcal{U}(R,\theta) = \hat{\mathcal{U}}(\lambda R,\theta), ~(R,\theta) \in (r_1,r_2) \times \mathbb{R}$, it follows that $\frac{\partial}{\partial R}\mathcal{U}(R,\theta) = \lambda~\hat{\mathcal{U}}_{\boldsymbol r}(\lambda R,\theta)$ from where $\frac{\partial \mathcal{U}}{\partial R}(r_2,\theta)=\lambda~\hat{\mathcal{U}}_{\boldsymbol r}(a,\theta) = \lambda~\hat{\phi}(a,\theta)=\phi(r_2,\theta)$ and also $\frac{\partial \mathcal{U}}{\partial R}(r_1,\theta)=\lambda~\hat{\mathcal{U}}_{\boldsymbol r}(\frac{1}{a},\theta) = \lambda~\hat{\phi}(\frac{1}{a},\theta)=\phi(r_1,\theta)$. In addition notice that equation \eqref{laplace eq in polar coordinates} is fulfilled for $\mathcal{U}$ on $(r_1,r_2) \times \mathbb{R}$, and  since $\mathcal{U}(\sqrt{r_1r_2},0)=\hat{\mathcal{U}}(1,0)=0$ one can conclude that $U=\mathcal{U}$. \\
		
		The proof of the second part is immediate and follows directly from equation \eqref{U definition Th1} by taking the derivative with respect to $r$.
		
	\end{proof}
	
	If an additional assumption on the smoothness of $\phi$ is added, the result in Theorem \ref{main theorem} can be strengthened. The main idea is that on $\boldsymbol A_{r_1,r_2}$ the solution $w$ of the Dirichlet problem \eqref{Dirichlet problem} with boundary data $g \in C^{m,\alpha}(\partial \boldsymbol A_{r_1,r_2})$, where $m \ge 2$ is integer and $\alpha \in (0,1]$, has the remarkable property that its $m$ order partial derivatives are locally $\alpha$ H{\"o}lder continuous in a sufficiently small neighborhood of each point $a \in \partial \boldsymbol A_{r_1,r_2}$. This result is often referred to as \textit{Kellogg's} theorem (for further details see \cite{Kellogg}). But we can link $w$ with $U$ given by \eqref{U definition Th1} and thus obtain important results on the continuous extensions of the higher order partial derivatives of $U$ to the closure of the domain where it is defined. Before stating and proving explicitly these results, we need to introduce a lemma which will be of crucial importance in the subsequent proofs.
	
	\begin{lemma}\label{Holder continuous extension for w}
		Let again $0<r_{1}<r_{2}<\infty$ and assume $\varphi :\left\{r_{1},r_{2}\right\} \times \mathbb{R}\rightarrow \mathbb{R}$ is $2\pi $-periodic in the second variable, satisfies the condition $\int\limits_{0}^{2\pi
		}\varphi \left( r_{1},\theta \right)~ d\theta = \int\limits_{0}^{2\pi }\varphi \left( r_{2},\theta \right)~ d\theta $, and in addition suppose there exists $\alpha \in (0,1]$ such that $\varphi(r,\cdot)$ belongs to $C^{m,\alpha}(\mathbb{R})$ for some positive integer $m \ge 2$, whenever $r \in \{r_1,r_2\}$. If the function $g: \partial \boldsymbol A_{r_1,r_2} \rightarrow \mathbb{R}$ satisfies $g(re^{i\theta}) = \varphi(r,\theta) ~\forall (r,\theta) \in  \{r_1,r_2\} \times \mathbb{R}$ and $w$ is the solution of the Dirichlet problem \eqref{Dirichlet problem} with boundary data $g$ then $w$,  together with all its partial derivatives up to order $m$, are uniformly H{\"o}lder continuous with exponent $\alpha$ on $\overline{\boldsymbol A_{r_1,r_2}}$. 	
	\end{lemma}
	\begin{proof}
		For simplicity the lemma will only be proved for the case $m=2$, since the case of a positive integer $m \ge 3$ follows in the same way, using backward induction. To begin with, choose any $z_1$, $z_2 \in \boldsymbol A_{r_1,r_2}$. By \textit{Kellogg's} theorem the second order partial derivatives of $w$ are locally $\alpha$ Holder continuous in some neighborhood of any point $a \in \partial \boldsymbol A_{r_1,r_2}$. Since $r_2 < \infty$ the closure of $\boldsymbol A_{r_1,r_2}$ is a compact subset of the complex-plane and consequently the second-order partial derivatives of $w$ turn out to be uniformly H{\"o}lder continuous with exponent $\alpha$. Indeed the second-order partial derivatives of $w$ are locally Lipschitz continuous in some neighborhood of any point contained in $\boldsymbol A_{r_1,r_2}$, and if we choose the neighborhood small enough it follows that they are also locally $\alpha$ H{\"o}lder continuous ($|z_2 - z_1| \le |z_2 - z_1|^{\alpha}$ when $z_1$ and $z_2$ are sufficiently close). Next assume $w_{\boldsymbol x \boldsymbol x}$ is not uniformly H{\"o}lder continuous with exponent $\alpha$. If so, there exist two sequences $z_n$ and $\xi_n$ in $\boldsymbol A_{r_1,r_2}$ such that $\frac{|w_{\boldsymbol x \boldsymbol x}(z_n) - w_{\boldsymbol x \boldsymbol x}(\xi_n)|}{|z_n - \xi_n|^{\alpha}} \rightarrow \infty$. Since $w_{\boldsymbol x \boldsymbol x}$ is in particular continuous on $\overline{\boldsymbol A_{r_1,r_2}}$ it is also bounded there and hence there are subsequences $z_{n_k}$ and $\xi_{n_k}$ converging to some $z$ in the closure of $\boldsymbol A_{r_1,r_2}$. But this contradicts the fact that $w_{\boldsymbol x \boldsymbol x}$ is locally $\alpha$ H{\"o}lder continuous at $z$. The exact same reasoning can also be applied to $w_{\boldsymbol y \boldsymbol y}$ and $w_{\boldsymbol x \boldsymbol y}$, respectively. It is also easy to prove that $w$, $w_{\boldsymbol x}$ and $w_{\boldsymbol y}$ are uniformly $\alpha$ H{\"o}lder continuous. This can be seen using an integral representation in terms of the higher-order partial derivatives in a sufficiently small convex neighborhood of each point $z \in \overline{\boldsymbol A_{r_1,r_2}}$, together with the compactness of the closed annulus. \\
	\end{proof}
	
	\begin{theorem}\label{Holder continuity polar theorem}
		Let $0<r_{1}<r_{2}<\infty$ and assume $\phi :\left\{r_{1},r_{2}\right\} \times \mathbb{R}\rightarrow \mathbb{R}$ is $2\pi $-periodic in the second variable, satisfies $\int\limits_{0}^{2\pi
		}r_1\phi \left( r_{1},\theta \right)~ d\theta = \int\limits_{0}^{2\pi }r_2\phi \left( r_{2},\theta \right)~ d\theta $, and in addition suppose there exists $\alpha \in (0,1]$ such that $\phi(r,\cdot)$ belongs to $C^{m,\alpha}(\mathbb{R})$ for some positive integer $m \ge 2$, whenever $r \in \{r_1,r_2\}$. If $U$ is any solution of the Neumamm problem \eqref{Neumann problem biss} on $[r_1,r_2] \times \mathbb{R}$ with boundary data $\phi$, then $U$ is uniformly H{\"o}lder continuous with exponent $\alpha$, and likewise are all its partial derivatives up to order $m+1$.
	\end{theorem}
	
	\begin{proof}
		Define $g: \partial \boldsymbol A_{r_1,r_2} \rightarrow \mathbb{R}$, $g(re^{i\theta}) = r\phi(r,\theta)$, $\theta \in \mathbb{R}$, and let $w$ be the solution of the Dirichlet problem \ref{Dirichlet problem} with boundary data $g$. According to Lemma \ref{Holder continuous extension for w} the harmonic function $w$ together with all its partial derivatives up to order $m$ are uniformly H{\"o}lder continuous with exponent $\alpha$ on $\overline{\boldsymbol A_{r_1,r_2}}$. The theorem will be proved for the case $m=2$, as the case of a general positive integer greater than or equal to two follows exactly in the same manner, using induction. With the same notations as those used in Theorem \ref{main theorem}, notice that Proposition \ref{equiv formulation in cartesian and polar coordinates} together with the uniqueness of the solution to Dirichlet problem implies that $u$ is just the representation in polar coordinates of $w$; more precisely we have $u(r,\theta) = w(re^{i\theta}) ~ \forall r \in [r_1,r_2], ~ \forall \theta \in \mathbb{R}$. Thus for all pairs $(r,\theta)$ in $(r_1,r_2) \times \mathbb{R}$ we obtain the following relations
		\begin{equation}\label{link cartesian-polar higher derivatives}
		\begin{cases}
		u_{\boldsymbol r}(r,\theta) = w_{\boldsymbol x}(z)\cos\theta + w_{\boldsymbol y}(z)\sin\theta, \\
		u_{\boldsymbol \theta}(r,\theta) = -rw_{\boldsymbol x}(z)\sin\theta + rw_{\boldsymbol y}(z)\cos\theta, \\ 
		u_{\boldsymbol r \boldsymbol r}(r,\theta) = w_{\boldsymbol x \boldsymbol x}(z)\cos^2\theta + w_{\boldsymbol x \boldsymbol y}(z)\sin 2\theta +  w_{\boldsymbol y \boldsymbol y}(z)\sin^2\theta, \\
		u_{\boldsymbol r \boldsymbol \theta}(r,\theta) = w_{\boldsymbol y}(z)\cos\theta - w_{\boldsymbol x}(z)\sin\theta + rw_{\boldsymbol x \boldsymbol y}(z)\cos 2\theta + \frac{r}{2}(w_{\boldsymbol y \boldsymbol y}(z) \\
		~~~~~~~~~~~~~~- w_{\boldsymbol x \boldsymbol x}(z))\sin 2\theta, \\
		u_{\boldsymbol \theta \boldsymbol \theta}(r,\theta) = -r^2w_{\boldsymbol x \boldsymbol y}(z)\sin 2\theta + r^2\left( w_{\boldsymbol y \boldsymbol y}(z)\cos^2\theta + w_{\boldsymbol x \boldsymbol x}(z)\sin^2\theta \right) - rw_{\boldsymbol y}(z)\cdot \\
		~~~~~~~~~~~~~~\sin\theta - rw_{\boldsymbol x}(z)\cos\theta, 
		\end{cases}
		\end{equation} 
		where $z = x+iy = r\cos\theta + ir\sin\theta$. According to Lemma \ref{Holder continuous extension for w} $w$, $w_{\boldsymbol x}$, $w_{\boldsymbol y}$, $w_{\boldsymbol x \boldsymbol x}$, $w_{\boldsymbol y \boldsymbol x}$, and $w_{\boldsymbol y \boldsymbol y}$ can be continuously extended to $\overline{\boldsymbol A_{r_1,r_2}}$, and so $u$ and all its partial derivatives up to order two can be continuously extended to $[r_1,r_2] \times \mathbb{R}$, due to the above relations. Further choose any $\alpha_1$, $\alpha_2 \in [r_1,r_2]$ and any $\theta_1$, $\theta_2 \in \mathbb{R}$ and denote $z_1 = \alpha_1e^{i\theta_1}$, $z_2 = \alpha_2e^{i\theta_2}$. Then one has $|u(\alpha_2,\theta_2) - u(\alpha_1,\theta_1)| = |w(z_2) - w(z_1)|$. Using again Lemma \ref{Holder continuous extension for w} there is a positive constant dubbed $H_w$ such that the H{\"o}lder constants corresponding to $w$ and all its partial derivatives up to order two, respectively, are upper bounded by it. Consequently this implies in particular that $|w(z_2) - w(z_1)| \le H_w |z_2 - z_1|^{\alpha}$. On the other hand notice that the geometry of the annulus $\boldsymbol A_{r_1,r_2}$ reveals that $z_1$ and $z_2$ must satisfy $|z_2 - z_1| \le |\alpha_2 - \alpha_1| + \max\{\alpha_1,\alpha_2\} |\theta_2 - \theta_1| \le |\alpha_2 - \alpha_1| + r_2 |\theta_2 - \theta_1|$ (see Figure \ref{fig:geometric_annulus}). Hence
		\begin{equation}\label{first geometric inequality}
		|z_2 - z_1| \le (r_2 + 1)\sqrt{|\alpha_2 - \alpha_1|^2 + |\theta_2 - \theta_1|^2}.
		\end{equation}
		To sum up we have just proved that 
		\begin{equation}\label{u Holder Th2}
		|u(\alpha_2,\theta_2) - u(\alpha_1,\theta_1)| \le H_w (r_2 + 1)^{\alpha} \left(|\alpha_2 - \alpha_1|^2 + |\theta_2 - \theta_1|^2\right)^{\frac{\alpha}{2}}.
		\end{equation}
		As for the first order partial derivatives of $u$, using the first two relations in \eqref{link cartesian-polar higher derivatives}, we notice that $|u_{\boldsymbol r}(\alpha_2,\theta_2) - u_{\boldsymbol r}(\alpha_1,\theta_1)| \le |w_{\boldsymbol x}(z_2)\cos\theta_2 - w_{\boldsymbol x}(z_1)\cos\theta_1| + |w_{\boldsymbol y}(z_2)\sin\theta_2 - w_{\boldsymbol y}(z_1)\sin\theta_1| \le |w_{\boldsymbol x}(z_2) - w_{\boldsymbol x}(z_1)| + |\cos\theta_2 - \cos\theta_1||w_{\boldsymbol x}(z_1)| + |w_{\boldsymbol y}(z_2) - w_{\boldsymbol y}(z_1)| + |\sin\theta_2 - \sin\theta_1||w_{\boldsymbol y}(z_1)| \le 2H_w|z_2-z_1| + \left(\|w_{\boldsymbol x}\| + \|w_{\boldsymbol y}\|\right) |\theta_2 - \theta_1|$. If the euclidean distance between the pairs of points $(\alpha_1,\theta_1)$, $(\alpha_2,\theta_2)$ is less than one then $\sqrt{|\alpha_2 - \alpha_1|^2 + |\theta_2 - \theta_1|^2} \le \left(|\alpha_2 - \alpha_1|^2 + |\theta_2 - \theta_1|^2\right)^{\frac{\alpha}{2}}$, and using \eqref{first geometric inequality}
		\begin{equation}\label{u_r Th2}
		|u_{\boldsymbol r}(\alpha_2,\theta_2) - u_{\boldsymbol r}(\alpha_1,\theta_1)| \lesssim \left(|\alpha_2 - \alpha_1|^2 + |\theta_2 - \theta_1|^2\right)^{\frac{\alpha}{2}}, 
		\end{equation}
		where a proportionality constant is given by
		\begin{equation*}
		2H_w(r_2+1)^{\alpha} + \|w_{\boldsymbol x}\| + \|w_{\boldsymbol y}\|.
		\end{equation*}
		In a similar way $|u_{\boldsymbol\theta}(\alpha_2,\theta_2) - u_{\boldsymbol\theta}(\alpha_1,\theta_1)| \le |\alpha_2 w_{\boldsymbol x}(z_2) \sin\theta_2 - \alpha_1 w_{\boldsymbol x}(z_1) \sin\theta_1| + |\alpha_2 w_{\boldsymbol y}(z_2) \cos\theta_2 - \alpha_1 w_{\boldsymbol y}(z_1) \cos\theta_1| \le |\alpha_2 w_{\boldsymbol x}(z_2) - \alpha_1 w_{\boldsymbol x}(z_1)| +\alpha_1 |\sin\theta_2 - \sin\theta_1|\cdot $ \\
		$|w_{\boldsymbol x}(z_1)| + |\alpha_2 w_{\boldsymbol y}(z_2) - \alpha_1 w_{\boldsymbol y}(z_1)| + \alpha_1 |\cos\theta_2 - \cos\theta_1||w_{\boldsymbol y}(z_1)| \le 2r_2H_w(r_2+1)^{\alpha}\left(|\alpha_2 - \alpha_1|^2 + |\theta_2 - \theta_1|^2\right)^{\frac{\alpha}{2}} + (r_2+1)\left(\|w_{\boldsymbol x}\| + \|w_{\boldsymbol y}|\right) \sqrt{|\alpha_2 - \alpha_1|^2 + |\theta_2 - \theta_1|^2}$, and if $(\alpha_1,\theta_1)$ and $(\alpha_2,\theta_2)$ are close enough in the euclidean distance (that is if the euclidean distance is less than one) then
		\begin{equation}\label{u_theta Th2}
		|u_{\boldsymbol\theta}(\alpha_2,\theta_2) - u_{\boldsymbol\theta}(\alpha_1,\theta_1)| \lesssim \left(|\alpha_2 - \alpha_1|^2 + |\theta_2 - \theta_1|^2\right)^{\frac{\alpha}{2}},
		\end{equation}
		where a proportionality constant is
		\begin{equation*}
		2r_2H_w(r_2+1)^{\alpha} + (r_2+1)\left(\|w_{\boldsymbol x}\| + \|w_{\boldsymbol y}\|\right).
		\end{equation*}
		Proceeding further the last three equations in \eqref{link cartesian-polar higher derivatives} will be used in order to derive similar conclusions on the second-order partial derivatives of $u$. To this end notice first that $|u_{\boldsymbol r \boldsymbol r}(\alpha_2,\theta_2) - u_{\boldsymbol r \boldsymbol r}(\alpha_1,\theta_1)| \le |w_{\boldsymbol x \boldsymbol x}(z_2) - w_{\boldsymbol x \boldsymbol x}(z_1)| + 2|\cos\theta_2 - \cos\theta_1||w_{\boldsymbol x \boldsymbol x}(z_1)| + |w_{\boldsymbol y \boldsymbol x}(z_2) - w_{\boldsymbol y \boldsymbol x}(z_1)| + |\sin 2\theta_2 - \sin 2\theta_1||w_{\boldsymbol y \boldsymbol x}(z_1)| + |w_{\boldsymbol y \boldsymbol y}(z_2) - w_{\boldsymbol y \boldsymbol y}(z_1)| + 2|\sin\theta_2 - \sin\theta_1||w_{\boldsymbol y \boldsymbol y}(z_1)| \le 3H_w|z_2 - z_1|^{\alpha} + \left(\|w_{\boldsymbol x \boldsymbol x}\| + \|w_{\boldsymbol y \boldsymbol x}\| + \|w_{\boldsymbol y \boldsymbol y}\|\right)$ \\
		$\cdot 2|\theta_2 - \theta_1|$, and if the pairs $(\alpha_2,\theta_2)$, $(\alpha_1,\theta_1)$ are again assumed to be close in the euclidean distance then 
		\begin{equation}
		|u_{\boldsymbol r \boldsymbol r}(\alpha_2,\theta_2) - u_{\boldsymbol r \boldsymbol r}(\alpha_1,\theta_1)| \lesssim \left(|\alpha_2 - \alpha_1|^2 + |\theta_2 - \theta_1|^2\right)^{\frac{\alpha}{2}},
		\end{equation} 
		where a proportionality constant is given by
		\begin{equation*}
		3H_w(r_2+1) + 2\left(\|w_{\boldsymbol x \boldsymbol x}\| + \|w_{\boldsymbol y \boldsymbol x}\| + \|w_{\boldsymbol y \boldsymbol y}\|\right).
		\end{equation*}
		Also $|u_{\boldsymbol r \boldsymbol \theta}(\alpha_2,\theta_2) - u_{\boldsymbol r \boldsymbol \theta}(\alpha_1,\theta_1)| \le |w_{\boldsymbol y}(z_2)\cos\theta_2 - w_{\boldsymbol y}(z_1)\cos\theta_1| + |w_{\boldsymbol x}(z_2)\sin\theta_2 - w_{\boldsymbol x}(z_1)\sin\theta_1| + |\alpha_2w_{\boldsymbol y \boldsymbol x}(z_2)\cos 2\theta_2 - \alpha_1w_{\boldsymbol y \boldsymbol x}(z_1)\cos 2\theta_1| + \frac{1}{2}|\alpha_2(w_{\boldsymbol y \boldsymbol y}(z_2) - w_{\boldsymbol x \boldsymbol x}(z_2))\cdot$ \\
		$\sin 2\theta_2 - \alpha_1(w_{\boldsymbol y \boldsymbol y}(z_1) - w_{\boldsymbol x \boldsymbol x}(z_1))\sin 2\theta_1|$. Again if $(\alpha_1,\theta_1)$ and $(\alpha_2,\theta_2)$ are sufficiently close in the euclidean distance then
		\begin{equation}
		|u_{\boldsymbol r \boldsymbol \theta}(\alpha_2,\theta_2) - u_{\boldsymbol r \boldsymbol \theta}(\alpha_1,\theta_1)| \lesssim \left(|\alpha_2 - \alpha_1|^2 + |\theta_2 - \theta_1|^2\right)^{\frac{\alpha}{2}},
		\end{equation}
		with a proportionality constant equal to
		\begin{equation*}
		\|w_{\boldsymbol x}\| + \|w_{\boldsymbol y}\| + (r_2 + 0.5)(\|w_{\boldsymbol x \boldsymbol x}\| + \|w_{\boldsymbol y \boldsymbol y}\|) + (2r_2 + 1)\|w_{\boldsymbol x \boldsymbol y}\| + 2H_w(r_2+1)^{\alpha + 1}.
		\end{equation*}
		Finally we compute $|u_{\boldsymbol \theta \boldsymbol \theta}(\alpha_2,\theta_2) - u_{\boldsymbol \theta \boldsymbol \theta}(\alpha_1,\theta_1)| \le |\alpha_2^2 w_{\boldsymbol x \boldsymbol y}(z_2)\sin 2\theta_2 -\alpha_1^2w_{\boldsymbol x \boldsymbol y}(z_1)\cdot$ \\
		$\sin 2\theta_1| + |\alpha_2^2(w_{\boldsymbol y \boldsymbol y}(z_2)\cos^2\theta_2 + w_{\boldsymbol x \boldsymbol x}(z_2)\sin^2\theta_2) - \alpha_1^2(w_{\boldsymbol y \boldsymbol y}(z_1)\cos^2\theta_1 + w_{\boldsymbol x \boldsymbol x}(z_1)\cdot$ \\
		$\sin^2\theta_1)| + |\alpha_2w_{\boldsymbol y}(z_2)\sin \theta_2 - \alpha_1w_{\boldsymbol y}(z_1)\sin\theta_1| + |\alpha_2w_{\boldsymbol x}(z_2)\cos\theta_2 -\alpha_1w_{\boldsymbol x}(z_1)\cos\theta_1|$, and if we assume once more that the pairs $(\alpha_1,\theta_1)$ and $(\alpha_2,\theta_2)$ are close in the euclidean distance then
		\begin{equation}\label{u_tt Holder Th2}
		|u_{\boldsymbol \theta \boldsymbol \theta}(\alpha_2,\theta_2) - u_{\boldsymbol \theta \boldsymbol \theta}(\alpha_1,\theta_1)| \lesssim \left(|\alpha_2 - \alpha_1|^2 + |\theta_2 - \theta_1|^2\right)^{\frac{\alpha}{2}},
		\end{equation}
		where a proportionality constant is
		\begin{align*}
		(r_2 + 1)\{ \|w_{\boldsymbol x}\| + \|w_{\boldsymbol y}\| + 2r_2 (\|w_{\boldsymbol x \boldsymbol x}\| + \|w_{\boldsymbol x \boldsymbol y}\| + \|w_{\boldsymbol y \boldsymbol y}\|) \} + 3H_wr_2(r_2+1)^{\alpha+1}.
		\end{align*}
		
		To sum up equations \eqref{u Holder Th2}-\eqref{u_tt Holder Th2} show that $u$ together with all its partial derivatives up to the second order are locally $\alpha$ H{\"o}lder continuous on $[r_1,r_2] \times \mathbb{R}$ with uniformly bounded constants. Hence $u$ and its partial derivatives up to order two are uniformly H{\"o}lder continuous with exponent $\alpha$ on any compact subset of $[r_1,r_2] \times \mathbb{R}$ and so on $[r_1,r_2] \times [0,3\pi]$ in particular. To show that this latter fact suffices to conclude that $u$ together with its partial derivatives up to order two are uniformly H{\"o}lder continuous with exponent $\alpha$ on $[r_1,r_2] \times \mathbb{R}$ choose any $\alpha_1, \ \alpha_2 \in [r_1,r_2]$ and any $\theta_1, \ \theta_2 \in \mathbb{R}$ and let $k_1, \ k_2$ be two integers such that defining $\theta_1' = \theta_1 - 2k_1\pi, \ \theta_2' = \theta_2 - 2k_2\pi$, $\theta_1'$ and $\theta_2'$ satisfy $|\theta_2' - \theta_1'| \in [0,2\pi]$. There are two possible cases.
		\begin{enumerate}
			\item[i.] $|\theta_2' - \theta_1'| \le \pi$, in which case we claim that $|\theta_2 - \theta_1| \ge |\theta_2' - \theta_1'|$. Indeed if $k_2 > k_1$ then $|\theta_2 - \theta_1| = |\theta_2' - \theta_1' + 2(k_2-k_1)\pi| = 2(k_2-k_1)\pi - (\theta_2' - \theta_1') \ge 2(k_2-k_1)\pi - |\theta_2' - \theta_1'| \ge 2(k_2-k_1)\pi - \pi \ge \pi \ge |\theta_2' - \theta_1'|$. If $k_1 > k_2$ switch the indexes $1$ and $2$, and if $k_1=k_2$ the inequality is trivial.
			\item[ii.] $|\theta_2' - \theta_1'| > \pi$. Assume first that $\theta_2' > \theta_1'$, in which case we find that $3\pi > \theta_1' + 2\pi =: \theta_1'' \ge \theta_2' =: \theta_2''$. Also $|\theta_2'' - \theta_1''| = \theta_1'' - \theta_2' = 2\pi - |\theta_2' - \theta_1'| < \pi$. But setting $k_1' := k_1-1$ and $k_2' := k_2$ it follows that $\theta_2 = \theta_2'' + 2k_2'\pi$ and $\theta_1 = \theta_1'' + 2k_1'\pi$, respectively. Since $|\theta_2'' - \theta_1''| < \pi$ the previous point shows $|\theta_2 - \theta_1| \ge |\theta_2'' - \theta_1''|$. If $\theta_1' > \theta_2'$ then $3\pi > \theta_2' + 2\pi =: \theta_2'' \ge \theta_1' =: \theta_1''$ and also $|\theta_2'' - \theta_1'| = 2\pi - |\theta_2' - \theta_1'| < \pi$. Proceeding similarly one obtains $|\theta_2 - \theta_1| \ge |\theta_2'' - \theta_1''|$. We conclude thus that in the case when $|\theta_2' - \theta_1'| > \pi$ there also exist two integers, dubbed $k_1'$ and $k_2'$, such that denoting $\theta_2'' = \theta_2 - 2k_2'\pi$ and $\theta_1'' = \theta_1 - 2k_1'\pi$, respectively, it follows that $\theta_1'', \ \theta_2'' \in [\pi,3\pi]$ and in addition $|\theta_2 - \theta_1| \ge |\theta_2'' - \theta_1''|$.
		\end{enumerate}
		Hence for any pairs $(\alpha_1,\theta_1), ~ (\alpha_2,\theta_2) \in [r_1,r_2] \times \mathbb{R}$ one can always find pairs $(\alpha_1,\theta_1')$ and $(\alpha_2,\theta_2')$, respectively, such that $\theta_1', \ \theta_2' \in [0,3\pi]$ and in addition $\sqrt{|\alpha_2-\alpha_1|^2 + |\theta_2'-\theta_1'|^2} \le \sqrt{|\alpha_2 - \alpha_1|^2 + |\theta_2 - \theta_1|^2}$.
		Consequently this shows, using the $2\pi$-periodicity in the second argument of $u$ and of all its partial derivatives up to order two, that there is a positive constant, call it $H_u$, such that 
		\begin{align*}\label{u and its derivatives uniformly Holder continuous}
		|u(\alpha_2,\theta_2) - u(\alpha_1,\theta_1)| &= |u(\alpha_2,\theta_2') - u(\alpha_1,\theta_1')| \le H_u \left(|\alpha_2 - \alpha_1|^2 + |\theta_2' - \theta_1'|^2\right)^{\frac{\alpha}{2}} \\
		& \le H_u \left(|\alpha_2 - \alpha_1|^2 + |\theta_2 - \theta_1|^2\right)^{\frac{\alpha}{2}},
		\end{align*}
		and the same holds true for $u_{\boldsymbol r}$, $u_{\boldsymbol \theta}$, $u_{\boldsymbol r \boldsymbol r}$, $u_{\boldsymbol \theta \boldsymbol r}$ and $u_{\boldsymbol \theta \boldsymbol \theta}$, respectively, on $[r_1,r_2] \times \mathbb{R}$.
		
		Finally the uniform H{\"o}lder continuity of $u$ and of its partial derivatives will be used in order to draw conclusions abut the uniform H{\"o}lder continuity of $U$ and of its partial derivatives up to order three. In this respect relations \eqref{U definition Th1} and \eqref{u Thm1} will be a key element. More precisely choose any $(\alpha_1,\theta_1), \ (\alpha_2,\theta_2) \in [r_1,r_2] \times \mathbb{R}$ and notice that $|U(\alpha_2,\theta_2) - U(\alpha_1,\theta_1)| \le \left|\int\limits_{\sqrt{r_1r_2}}^{\alpha_2} \frac{u(\rho,\theta_2)}{\rho}d\rho - \int\limits_{\sqrt{r_1r_2}}^{\alpha_1} \frac{u(\rho,\theta_1)}{\rho}d\rho \right| + \sqrt{r_1r_2} \left|\int\limits_{\theta_1}^{\theta_2}\left( \mathcal{C} - \int\limits_0^t u_{\boldsymbol r}(\sqrt{r_1r_2},\tau)d\tau\right) dt\right|$. According to the proof of Theorem \ref{main theorem} the real-valued function $h(t) := \mathcal{C} - \int\limits_0^t u_{\boldsymbol r}(\sqrt{r_1r_2},\tau)d\tau$ is $2\pi$-periodic and so one can readily see that $\|h\| = \sup\limits_{t \in [0,2\pi]}|h(t)| < \infty$. Consequently $|U(\alpha_2,\theta_2) - U(\alpha_1,\theta_1)| \le $ \\
		$\left|\int\limits_{\sqrt{r_1r_2}}^{\alpha_2} \frac{u(\rho,\theta_2)}{\rho}d\rho - \int\limits_{\sqrt{r_1r_2}}^{\alpha_1} \frac{u(\rho,\theta_1)}{\rho}d\rho \right| + \sqrt{r_1r_2}\int\limits_{\theta_1 \wedge \theta_2}^{\theta_1 \vee \theta_2}\left|\mathcal{C} - \int\limits_0^t u_{\boldsymbol r}(\sqrt{r_1r_2},\tau)d\tau\right|dt$, where the latter term is less than $\int\limits_{\alpha_1 \wedge \alpha_2}^{\alpha_1 \vee \alpha_2} \frac{|u(\rho,\theta_2)|}{\rho}d\rho + \int\limits_{\alpha_1 \wedge \sqrt{r_1r_2}}^{\alpha_2 \vee \sqrt{r_1,r_2}} \frac{|u(\rho,\theta_2) - u(\rho,\theta_1)|}{\rho}d\rho + \sqrt{r_1r_2}\cdot $ \\
		$\|h\||\theta_2 - \theta_1| \le \frac{1}{r_1}|\alpha_2 - \alpha_1|\|u\| + \frac{r_2 - r_1}{r_1}|\theta_2 - \theta_1|^{\alpha} + \sqrt{r_1r_2}|\theta_2 - \theta_1|\|h\|$. Under the assumption that $(\alpha_1,\theta_1)$ and $(\alpha_2,
		\theta_2)$ are close enough in the euclidean distance we observe that
		\begin{equation}
		|U(\alpha_2,\theta_2) - U(\alpha_1,\theta_1)| \lesssim \left(|\alpha_2 - \alpha_1|^2 + |\theta_2 - \theta_1|^2\right)^{\frac{\alpha}{2}},
		\end{equation}
		where a proportionality constant is given by
		\begin{equation*}
		\frac{\|u\|}{r_1} + \frac{r_2-r_1}{r_1} + \sqrt{r_1r_2}\|h\|.
		\end{equation*}
		Using relation \eqref{u Thm1} $|U_{\boldsymbol r}(\alpha_2,\theta_2) - U_{\boldsymbol r}(\alpha_1,\theta_1)| \le \frac{1}{r_1^2}|\alpha_1 u(\alpha_2,\theta_2) - \alpha_2 u (\alpha_1,\theta_1)| \le \frac{r_2}{r_1^2}|u(\alpha_2,\theta_2) - u(\alpha_1,\theta_1)| + \frac{\|u\|}{r_1^2}|\alpha_2 - \alpha_1|$ and assuming $(\alpha_1,\theta_1)$, $(\alpha_2,\theta_2)$ are sufficiently close then
		\begin{equation}
		|U_{\boldsymbol r}(\alpha_2,\theta_2) - U_{\boldsymbol r}(\alpha_1,\theta_1)| \lesssim \left(|\alpha_2 - \alpha_1|^2 + |\theta_2 - \theta_1|^2\right)^{\frac{\alpha}{2}}, 
		\end{equation} 
		where a proportionality constant is
		\begin{equation*}
		\frac{r_2+\|u\|}{r_1^2}.
		\end{equation*}
		Taking the derivative with respect to the second argument in \eqref{U definition Th1} and using the $2\pi$-periodicity in the second argument for $U_{\boldsymbol \theta}$ gives $|U_{\boldsymbol \theta}(\alpha_2,\theta_2) - U_{\boldsymbol \theta}(\alpha_1,\theta_1)| \le \int\limits_{\alpha_1 \wedge \alpha_2}^{\alpha_1 \vee \alpha_2}\frac{|u_{\boldsymbol \theta}(\rho,\theta_2)|}{\rho}d\rho + \int\limits_{\sqrt{r_1r_2}}^{\alpha_1}\frac{|u_{\boldsymbol \theta}(\rho,\theta_2) - u_{\boldsymbol \theta}(\rho,\theta_1)|}{\rho}d\rho 
		+ \sqrt{r_1r_2}\int\limits_{\theta_1 \wedge \theta_2}^{\theta_1 \vee \theta_2}|u_{\boldsymbol r}(\sqrt{r_1r_2},t)|dt \le |\alpha_2 - \alpha_1|\frac{\|u_{\boldsymbol \theta}\|}{r_1} + \frac{r_2 - r_1}{r_1}H_u|\theta_2 - \theta_1|^{\alpha} + r_2|\theta_2 - \theta_1| \|u_{\boldsymbol r}\|$, and if $(\alpha_1,\theta_1), \ (\alpha_2,\theta_2)$ are sufficiently close then
		\begin{equation}
		|U_{\boldsymbol \theta}(\alpha_2,\theta_2) - U_{\boldsymbol \theta}(\alpha_1,\theta_1)| \lesssim \left(|\alpha_2 - \alpha_1|^2 + |\theta_2 - \theta_1|^2\right)^{\frac{\alpha}{2}}, 
		\end{equation} 
		where a proportionality constant is found to be
		\begin{equation*}
		\frac{\|u_{\boldsymbol \theta}\|}{r_1} + \frac{r_2 - r_1}{r_1}H_u + r_2\|u_{\boldsymbol r}\|.
		\end{equation*}
		For the partial derivatives of order two of $U$ notice that, whenever $(r,\theta) \in (r_1,r_2) \times \mathbb{R}$, they are given by $U_{\boldsymbol r \boldsymbol r} = \frac{u_{\boldsymbol r}(r,\theta)}{r} - \frac{u(r,\theta)}{r^2}$, $U_{\boldsymbol r \boldsymbol \theta}(r,\theta) = \frac{u_{\boldsymbol \theta}(r,\theta)}{r}$, and using relations \eqref{laplace eq in polar coordinates} and \eqref{u Thm1} $U_{\boldsymbol \theta \boldsymbol \theta}(r,\theta) = -r^2U_{\boldsymbol r \boldsymbol r}(r,\theta) - rU_{\boldsymbol r}(r,\theta) = -ru_{\boldsymbol r}(r,\theta)$. These relations show that $U_{\boldsymbol r \boldsymbol \theta}$, $U_{\boldsymbol r \boldsymbol r}$ and $U_{\boldsymbol \theta \boldsymbol \theta}$ can be continuously extended to $[r_1,r_2] \times \mathbb{R}$ and the same notation will be kept for their continuous extensions. Choose now any pairs $(\alpha_1,\theta_1), \ (\alpha_2,\theta_2) \in [r_1,r_2] \times \mathbb{R}$ and compute $|U_{\boldsymbol r \boldsymbol r}(\alpha_2,\theta_2) - U_{\boldsymbol r \boldsymbol r}(\alpha_1,\theta_1)| \le \left| \frac{u_{\boldsymbol r}(\alpha_2,\theta_2)}{\alpha_2} - \frac{u_{\boldsymbol r}(\alpha_1,\theta_1)}{\alpha_1} \right| + \left|\frac{u(\alpha_2,\theta_2)}{\alpha_2^2} - \frac{u(\alpha_1,\theta_1)}{\alpha_1^2} \right| \le \frac{r_2}{r_1^2}\left|u_{\boldsymbol r}(\alpha_2,\theta_2) - u_{\boldsymbol r}(\alpha_1,\theta_1) \right| + \frac{1}{r_1^2}|\alpha_2 - \alpha_1||u_{\boldsymbol r}(\alpha_1,\theta_1)| + \frac{r_2^2}{r_1^4}|u(\alpha_2,\theta_2) - u(\alpha_1,\theta_1)| + \frac{r_1 + r_2}{r_1^4}|u(\alpha_1,\theta_1)||\alpha_2 - \alpha_1|$. Again if $(\alpha_1,\theta_1)$, $(\alpha_2,\theta_2)$ are close in the euclidean distance then
		\begin{equation}
		|U_{\boldsymbol r \boldsymbol r}(\alpha_2,\theta_2) - U_{\boldsymbol r \boldsymbol r}(\alpha_1,\theta_1)| \lesssim \left(|\alpha_2 - \alpha_1|^2 + |\theta_2 - \theta_1|^2\right)^{\frac{\alpha}{2}},
		\end{equation}
		where a proportionality constant is given by
		\begin{equation*}
		\frac{r_2^2 + \|u_{\boldsymbol r}\| + (r_1+r_2)\|u\|}{r_1^4}.
		\end{equation*} 
		Also $|U_{\boldsymbol r \boldsymbol \theta}(\alpha_2,\theta_2) - U_{\boldsymbol r \boldsymbol \theta}(\alpha_1,\theta_1)| \le \frac{r_2}{r_1^2}|u_{\boldsymbol \theta}(\alpha_2,\theta_2) - u_{\boldsymbol \theta}(\alpha_1,\theta_1)| +  \frac{1}{r_1^2}|u_{\boldsymbol \theta}(\alpha_1,\theta_1)||\alpha_2 - \alpha_1|$, and under the same closeness assumption on the pairs $(\alpha_1,\theta_1), \ (\alpha_2,\theta_2)$ one has
		\begin{equation}
		|U_{\boldsymbol r \boldsymbol \theta}(\alpha_2,\theta_2) - U_{\boldsymbol r \boldsymbol \theta}(\alpha_1,\theta_1)| \lesssim \left(|\alpha_2 - \alpha_1|^2 + |\theta_2 - \theta_1|^2\right)^{\frac{\alpha}{2}},
		\end{equation} 
		where a proportionality constant can be easily found to be
		\begin{equation*}
		\frac{r_2 + \|u_{\boldsymbol \theta}\|}{r_1^2}.
		\end{equation*}
		Finally $|U_{\boldsymbol \theta \boldsymbol \theta}(\alpha_2,\theta_2) - U_{\boldsymbol \theta \boldsymbol \theta}(\alpha_1,\theta_1)| \le r_2|u_{\boldsymbol r}(\alpha_2,\theta_2) - u_{\boldsymbol r}(\alpha_1,\theta_1)| + |u_{\boldsymbol r}(\alpha_1,\theta_1)|\cdot$ \\
		$|\alpha_2 - \alpha_1|$, and if $(\alpha_1,\theta_1), \ (\alpha_2,\theta_2)$ are close enough in the euclidean distance, then 
		\begin{equation}
		|U_{\boldsymbol \theta \boldsymbol \theta}(\alpha_2,\theta_2) - U_{\boldsymbol \theta \boldsymbol \theta}(\alpha_1,\theta_1)| \lesssim \left(|\alpha_2 - \alpha_1|^2 + |\theta_2 - \theta_1|^2\right)^{\frac{\alpha}{2}},
		\end{equation}
		where a proportionality constant is
		\begin{equation*}
		r_2H_u + \|u_{\boldsymbol r}\|.
		\end{equation*}
		Proceeding further $U_{\boldsymbol r \boldsymbol r \boldsymbol r}(r,\theta) = \frac{\partial}{\partial r}\left(\frac{ru_{\boldsymbol r}(r,\theta) - u(r,\theta)}{r^2}\right) = \frac{r^2u_{\boldsymbol r \boldsymbol r}(r,\theta) -2ru_{\boldsymbol r}(r,\theta) + 2u(r,\theta)}{r^3}$, $U_{\boldsymbol r \boldsymbol \theta \boldsymbol \theta}(r,\theta) = U_{\boldsymbol \theta \boldsymbol r \boldsymbol \theta}(r,\theta) = U_{\boldsymbol \theta \boldsymbol \theta \boldsymbol r}(r,\theta) = \frac{\partial}{\partial r}U_{\boldsymbol \theta \boldsymbol \theta}(r,\theta) = -u_{\boldsymbol r}(r,\theta) - ru_{\boldsymbol r \boldsymbol r}(r,\theta)$, $U_{\boldsymbol \theta \boldsymbol r \boldsymbol r}(r,\theta) = U_{\boldsymbol r \boldsymbol \theta \boldsymbol r}(r,\theta) = U_{\boldsymbol r \boldsymbol r \boldsymbol \theta}(r,\theta) = \frac{\partial}{\partial \theta}\left(\frac{ru_{\boldsymbol r}(r,\theta) - u(r,\theta)}{r^2} \right) = \frac{ru_{\boldsymbol \theta \boldsymbol r}(r,\theta) - u_{\boldsymbol \theta}}{r^2}$, and $U_{\boldsymbol \theta \boldsymbol \theta \boldsymbol \theta}(r,\theta) = -ru_{\boldsymbol \theta \boldsymbol r}(r,\theta)$. But then a similar reasoning as above, using the triangle inequality, shows that all the third order partial derivatives of $U$ can be continuously extended to $[r_1,r_2]\times \mathbb{R}$ and satisfy 
		\begin{equation}
		|U_{\boldsymbol a \boldsymbol b \boldsymbol c}(\alpha_2,\theta_2) - U_{\boldsymbol a \boldsymbol b \boldsymbol c}(\alpha_1,\theta_1)| \lesssim \left(|\alpha_2 - \alpha_1|^2 + |\theta_2 - \theta_1|^2\right)^{\frac{\alpha}{2}}, ~\boldsymbol a, \boldsymbol b, \boldsymbol c \in \{\boldsymbol r, \boldsymbol \theta\},
		\end{equation} 
		with proportionality constants depending only on $r_1$, $r_2$, $H_u$, $\|u\|$, $\|u_{\boldsymbol r}\|$, $\|u_{\boldsymbol \theta}\|$, $\|u_{\boldsymbol r \boldsymbol r}\|$, $\|u_{\boldsymbol r \boldsymbol \theta}\|$, and $\|u_{\boldsymbol \theta \boldsymbol \theta}\|$, for any $(\alpha_1,\theta_1)$, $(\alpha_2,\theta_2) \in [r_1,r_2] \times \mathbb{R}$ which are close enough in the euclidean distance.
		
		In conclusion it has been shown so far that $U$ together with its partial derivatives up to order three are locally $\alpha$ H{\"o}lder continuous on $[r_1,r_2] \times \mathbb{R}$, and using a compactness argument we can argue that they are uniformly H{\"o}lder continuous with exponent $\alpha$ on $[r_1,r_2] \times [0,3\pi]$. By considering the same argument as the one used earlier in the proof for $u$ and its partial derivatives up to order two, we can see that $U$ and its partial derivatives up to order three are actually uniformly H{\"o}lder continuous with exponent $\alpha$ on $[r_1,r_2] \times \mathbb{R}$. This concludes the whole proof.
	\end{proof}
	
	\begin{figure}[h!]
		\centering	
		\includegraphics[width=\linewidth]{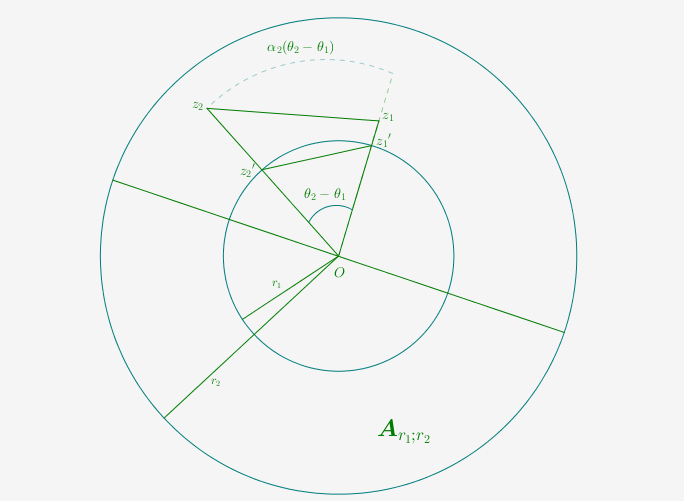}
		\caption{Some geometric properties of the annulus.}
		\label{fig:geometric_annulus}	
	\end{figure}
	
	Next some remarks will be provided. Notice first that for $r_{1}\searrow 0$ and $r_{2} = 1$, the region $%
	\boldsymbol A_{r_{1};r_{2}}$ becomes the punctured unit disk $\boldsymbol A_{0;1}=\left\{ z\in
	\mathbb{C}:0<\left\vert z\right\vert <1\right\} = \dot{\mathbb{U}} $. If $  w: \dot{\mathbb{U}} \rightarrow \mathbb{R}$ is a harmonic function having a finite limit at the origin (an isolated boundary point of the domain), then it is known that $w$ can be extended by continuity at the origin, and the resulting function is harmonic in $\mathbb{U}$. If $w$ has a continuous extension to $\overline{\mathbb{U}}$, with boundary values $w(0) = \varphi( 0,\cdot)$ (a constant function of $\theta \in \mathbb{R}$) and $w( e^{i\theta }) =\varphi( 1,\theta) $, $\theta \in \mathbb{R}$, then the condition $\int\limits_{0}^{2\pi }\varphi \left( 0,\theta \right)~ d\theta
	=\int\limits_{0}^{2\pi }\varphi \left( 1,\theta \right)~ d\theta $ in Theorem \ref{main theorem} is a necessary condition for the solvability of the
	Dirichlet problem in $\dot{\mathbb{U}}$ with continuous boundary data $\varphi $;
	\begin{equation*}
	w\left( 0\right) = \varphi \left( 0,\cdot \right) =\frac{1}{2\pi }%
	\int\limits_{0}^{2\pi }\varphi \left( 0,\theta \right)~ d\theta =\frac{1}{%
		2\pi }\int\limits_{0}^{2\pi }\varphi \left( 1,\theta \right)~ d\theta .
	\end{equation*}
	Subtracting a constant if necessary (i.e. considering $w-w\left( 0\right)$
	instead), without loss of generality it can be assumed that $w\left(
	0\right) =0$, or equivalently $0 = \frac{1}{2\pi }\int\limits_{0}^{2\pi }\varphi
	\left( 1,\theta \right)~ d\theta = \varphi \left( 0,\cdot \right)$. The above discussion shows that in the case of the punctured disk $\dot{\mathbb{U}} = \boldsymbol A_{0;1}$, the Dirichlet problem (\ref{Dirichlet problem biss}) has a unique solution for continuous boundary data
	$\varphi $ under the hypothesis $0 = \frac{1}{2\pi }\int\limits_{0}^{2\pi }\varphi
	\left( 1,\theta \right)~ d\theta = \varphi \left( 0,\cdot \right)$
	(implying $w\left( 0\right) =0$), which coincides with the solution of the Dirichlet problem in the whole unit disk $\mathbb{U}$, formulated in polar coordinates, with boundary data $\varphi
	\left( 1,\cdot \right)$, and thus under these
	hypotheses one can simply ignore the boundary condition at the origin
	(isolated boundary point of $\dot{\mathbb{U}}$).
	
	Similarly if $\phi(r,\cdot) \in C^0(\mathbb{R}), ~ r \in \{0,1\}$, satisfies $0 = \frac{1}{2\pi } \int\limits_{0}^{2\pi }\phi \left( 1,\theta \right)~ d\theta$, $\phi(0,\theta)=\frac{1}{\pi}\int\limits_0^{2\pi} cos~(t - \theta) \phi(1,t) ~dt, ~\theta \in \mathbb{R}$  and if $U_0$ is the solution of the Neumann problem \eqref{Neumann problem} on $\mathbb{U}$ which vanishes for $z=0$ and has boundary data $\Phi_0(z) := \phi \left(1,\arg(z)\right)$, then the Neumann problem (\ref{Neumann problem biss}) has a solution which actually coincides with the representation in polar coordinates of $U_0$. Indeed, by applying \cite[Theorem 1]{Beznea1} (or Corollary \ref{main corollary} below) it follows that $U_0(z)=\int\limits_0^1 \frac{u_0(\rho z)}{\rho}d\rho$, where $u_0$ is the solution of the Dirichlet problem in $\mathbb{U}$ with boundary data $\varphi_0 := \Phi_0$. But then defining $\hat{U}_0(r,\theta) = U_0(re^{i\theta}),~ (r,\theta) \in [0,1] \times \mathbb{R}$, it follows by Proposition \ref{equiv formulation in cartesian and polar coordinates} that $\hat{U}_0$ is $2\pi$-periodic in the second variable, has continuous second order partial derivatives, and satisfies equation \eqref{laplace eq in polar coordinates} in $(0,1) \times \mathbb{R}$, and in addition it has finite partial derivative with respect to the first variable at any point $(1,\theta_0), ~\theta_0 \in \mathbb{R}$. Moreover $\frac{\partial \hat{U}_0}{\partial r}(0,\theta) = \lim\limits_{r \searrow 0} \frac{\hat{U}_0(r,\theta) - \hat{U}_0(0,\theta)}{r} = \lim\limits_{r \searrow 0} \frac{\hat{U}_0(r,\theta)}{r} = \lim\limits_{r \searrow 0} \frac{U_0\left(re^{i\theta}\right)}{r} = \lim\limits_{r \searrow 0} \frac{1}{r}\int\limits_0^1\frac{u_0\left(\rho re^{i\theta}\right)d\rho}{\rho} = \frac{\partial u_0}{\partial x}(0)\cos\theta+ \frac{\partial u_0}{\partial y}(0)\sin\theta = \frac{\partial u_0}{\partial \boldsymbol a_{\theta}}(0) ~\forall \theta \in \mathbb{R}$. Also $\frac{\partial \hat{U}_0}{\partial r}(1,\theta)=\frac{\partial U_0}{\partial \boldsymbol \nu}\left(e^{i\theta}\right)=\Phi_0\left(e^{i\theta}\right)=\phi(1,\theta)$, whenever $\theta \in \mathbb{R}$. It is not difficult to show, using Poisson's formula as well as the \textit{Dominant Convergence} theorem (see also \cite[Theorem 2.27]{Real_Analysis}), that $\frac{\partial u_0}{\partial x}(0) = \frac{1}{\pi}\int\limits_0^{2\pi} \Phi_0(e^{it})\cos t ~dt = \frac{1}{\pi}\int\limits_0^{2\pi} \phi(1,t)\cos t ~dt$ and $\frac{\partial u_0}{\partial y}(0) = \frac{1}{\pi}\int\limits_0^{2\pi} \Phi_0(e^{it})\sin t ~dt = \frac{1}{\pi}\int\limits_0^{2\pi} \phi(1,t)\sin t ~dt$, which finally gives $\frac{\partial \hat{U}_0}{\partial r}(0,\theta) = \frac{1}{\pi}\int\limits_0^{2\pi} \cos (t - \theta) \phi(1,t) ~dt$. To sum up it can be concluded that
	\begin{equation*}
	\frac{\partial \hat{U}_0}{\partial r}(r,\theta) = \begin{cases}\frac{1}{\pi}\int\limits_0^{2\pi} \cos(t - \theta) \phi(1,t) ~dt, ~\text{ if } r=0, \\
	\phi(1,\theta), ~\text{ if } r=1. \end{cases}
	\end{equation*}
	The continuous extensions of $\frac{\partial \hat{U}_0}{\partial r}$ and $\frac{\partial \hat{U}_0}{\partial \theta}$ to $[0,1]\times\mathbb{R}$ are easily justified by Corollary \ref{main corollary} (below) and Proposition \ref{equiv formulation in cartesian and polar coordinates}. \\
	
	With this preamble the following two definitions will be introduced, with the convention that for both of them $\boldsymbol D = \boldsymbol A_{0;1}$.
	\begin{definition}\label{Dirichlet_punctured}
		If $\varphi : \mathbb{R} \rightarrow \mathbb{R}$ is continuous, $2\pi$-periodic, and satisfies $\int\limits_0^{2\pi}\varphi(\theta)~ d\theta = 0$, then the Dirichlet problem in polar coordinates for $\boldsymbol D$ consists in finding \\ $u=u(r,\theta) \in C^2((0,1) \times \mathbb{R}) \cap C^0([0,1] \times \mathbb{R})$ which is $2\pi$-periodic in the second variable and satisfies
		\begin{equation}
		\left\{
		\begin{array}{l}
		u_{rr}+\frac{1}{r}u_{r}+\frac{1}{r^{2}}u_{\theta \theta }=0 \qquad \text{in }%
		(0,1)\times \mathbb{R}, \\
		u(1,\cdot) = \varphi(\cdot), \\
		u(0,\cdot) = 0.
		\end{array}
		\right.
		\end{equation}
	\end{definition}
	
	\begin{definition}\label{Neumann_punctured}
		If $\phi: \{0;1\} \times \mathbb{R} \rightarrow \mathbb{R}$ is continuous, $2\pi$-periodic in the second argument, and satisfies $\int\limits_0^{2\pi} \phi(1,\theta)~ d\theta = 0$ as well as $\phi(0,\theta) = \frac{1}{\pi}\int\limits_0^{2\pi} \cos(t - \theta) \phi(1,t) ~dt$, then the Neumann problem in polar coordinates for $\boldsymbol D$ consists in finding \\
		$U\in C^2((0,1)\times \mathbb{R}) \cap C^1([0,1]\times\mathbb{R})$ which is $2\pi$-periodic in the second variable, and satisfies
		\begin{equation}\label{Neumann_punctured_solution}
		\begin{cases}
		U_{rr}+\frac{1}{r}U_{r}+\frac{1}{r^{2}}U_{\theta \theta }=0 &~\text{in } (0,1)\times \mathbb{R}, \\
		U_r = \phi &~\text{in } \{0;1\}\times \mathbb{R}, \\
		U(0,\cdot) = 0.
		\end{cases}
		\end{equation}
	\end{definition}
	
	\begin{remark} \label{Remark Neumann punctured}
		As we have already remarked, Definition \ref{Dirichlet_punctured} is nothing but the polar coordinates version of the Dirichlet problem \eqref{Dirichlet problem} for $\boldsymbol D=\mathbb{U}$ and boundary data $g(z)=\varphi(\arg(z))~ \text{on } \partial \mathbb{U}$. Definition \ref{Neumann_punctured}, instead, comes with a novelty which allows one to formulate the Neumann problem in a consistent way, for the punctured disk as well. This fact is in contrast with the classical Neumann problem where the (outward) normal derivative at $\{0\}$ can not be defined. In addition it reveals that if $\hat{U}$ is the solution of the Neumann problem \eqref{Neumann_punctured_solution} on $\boldsymbol A_{0;1}$, then $\hat{U}$ is just the representation in polar coordinates of the solution $U$ to the Neumann problem \eqref{Neumann problem} on $\mathbb{U}$, with boundary data $f(z) = \phi(1,\arg(z))$ and $U(0)=0$.
	\end{remark}
	
	In the particular case when the boundary data is symmetric, the result in
	Theorem \ref{main theorem} has the following simplified form.
	
	\begin{theorem}\label{simplified formula Neumann}
		Let $0<r_{1}<r_{2} < \infty$ and assume $\phi :\left\{ r_{1},r_{2}\right\} \times
		\mathbb{R}\rightarrow \mathbb{R}$ is continuous, $2\pi $-periodic in the second argument, verifies the Dirichlet conditions as a function of $\theta$, and satisfies $r_1\phi \left( r_{1},\theta \right) =r_2\phi
		\left( r_{2},\theta \right) $ for $\theta \in \mathbb{R}$. If $U$ is the solution of the Neumann problem (\ref{Neumann problem biss}) with boundary data $\phi$, satisfying $U(\sqrt{r_{1}r_{2}},0) =0$, then
		\begin{equation}\label{U definition Th3}
		U(r,\theta )=\int\limits_{\frac{\sqrt{r_{1}r_{2}}}{r}}^{1}\frac{u(r\rho,\theta )}{\rho }d\rho, ~\left(r,\theta\right) \in \left[r_{1},r_{2}\right] \times \mathbb{R},
		\end{equation}
		where $u$ is the solution of the Dirichlet problem (\ref{Dirichlet problem biss}%
		) with boundary values $\varphi(r,\theta) = r\phi(r,\theta)$ on $\left\{ r_{1},r_{2}\right\} \times \mathbb{R}$. Conversely, if $\varphi :\left\{ r_{1},r_{2}\right\} \times \mathbb{R}%
		\rightarrow \mathbb{R}$ is continuous, $2\pi $-periodic in the second
		variable, and satisfies $\int\limits_{0}^{2\pi }\varphi \left( r_{1},\theta
		\right) d\theta =\int\limits_{0}^{2\pi }\varphi \left( r_{2},\theta \right)
		d\theta $, and if $U$ is a solution
		of the Neumann problem (\ref{Neumann problem biss}) with boundary data $\phi(r,\theta)
		=\frac{\varphi(r,\theta)}{r}, ~ r \in \{r_{1},r_{2}\}$, then
		\begin{equation*}
		u\left( r,\theta \right) =rU_{\boldsymbol r}\left( r,\theta \right), \qquad \left(
		r,\theta \right) \in \left[ r_{1},r_{2}\right] \times \mathbb{R},
		\end{equation*}%
		is the solution of the Dirichlet problem (\ref{Dirichlet problem biss}).
	\end{theorem}
	
	\begin{proof}
		For the first part it will be shown that under the additional hypothesis $r_1\phi(r_1,\theta)=r_2\phi(r_2,\theta)$, $\theta\in \mathbb{R}$, one has $u(r,\theta)=u(\frac{r_1r_2}{r},\theta) ~\forall ~(r,\theta) \in [r_1,r_2]\times \mathbb{R}$ from where it follows by derivation with respect to the first argument that $u_{\boldsymbol r}(r,\theta) = -\frac{r_1r_2}{r^2}u_{\boldsymbol r}\left(\frac{r_1r_2}{r},\theta\right)$ and taking $r=\sqrt{r_1r_2}$ it follows that $u_{\boldsymbol r}(\sqrt{r_1r_2},\theta) = -u_{\boldsymbol r}(\sqrt{r_1r_2},\theta)$ which in turn implies $u_{\boldsymbol r}(\sqrt{r_1r_2},\theta)=0 ~\forall~ \theta \in \mathbb{R}$, and so $U$ will have the desired expression. Notice that it is enough to prove the result for the special case $r_2=a, ~r_1=1/a$, since the general case follows from this one by means of scalarization, in the same way it was done in the proof of Theorem \ref{main theorem}. Hence it can be assumed without loss of generality that $r_2=a, ~r_1=1/a, ~a > 1$. Writing again the Fourier expansions for $\varphi(r_2,\cdot) = \varphi(r_1,\cdot)$ it is obtained $\varphi(r_2,\theta)=a_0 + \sum\limits_{k=1}^{\infty}(a_k\cos k\theta + b_k\sin k\theta)=\varphi(r_1,\theta) ~\forall ~\theta\in \mathbb{R}$. But then the solution of the Dirichlet problem \eqref{Dirichlet problem biss} on $\left(\frac{1}{a},a\right) \times \mathbb{R}$ with boundary data $\varphi$ is given by \\
		$u(r,\theta)=A+B\log r +\sum\limits_{k=1}^{\infty}\left[\left(C_kr^k + D_Kr^{-k}\right)\cos k\theta + \left(E_kr^k+G_kr^{-k}\right)\sin k\theta\right]$, with $\begin{cases}A-B\log a=a_0, \\ A+B\log a=a_0, \\ C_ka^{-k} + D_ka^k = a_k, \\ C_ka^k + D_ka^{-k} = a_k, \\ E_ka^{-k} + G_ka^k = b_k, \\ E_ka^k + G_ka^{-k} = b_k, \end{cases} ~\forall k \in \mathbb{N}^*$. Consequently it follows that \\
		$A=a_0, ~B=0, ~C_k = D_k = \frac{a_k}{a^k+a^{-k}}, ~E_k = G_k = \frac{b_k}{a^k+a^{-k}}, ~ k \in \mathbb{N}^*$, and hence \\
		$u(r,\theta) = a_0 + \sum\limits_{k=1}^{\infty}\left[\frac{r^k+r^{-k}}{a^k+a^{-k}}\left(a_k \cos k\theta +  b_k\sin k\theta\right)\right], ~ (r,\theta) \in \left[\frac{1}{a},a\right] \times \mathbb{R}$, which finally gives $u(r,\theta) = u(\frac{1}{r},\theta)$ for any pair $(r,\theta) \in \left[\frac{1}{a},a\right]  \times \mathbb{R}$. \\
		
		The second part is just the second part of Theorem \ref{main theorem}. This concludes the proof.
	\end{proof}
	
	Combining Proposition \ref{equiv formulation in cartesian and polar coordinates}, Theorem \ref{main theorem}, and Theorem \ref{Holder continuity polar theorem}, an important result in cartesian coordinates is obtained. Before presenting and proving it, the following lemma must be provided.
	\begin{lemma}\label{theta-z lemma}
		There exists some positive constant $L > 0$ such that if $z_1 = \alpha_1e^{i\theta_1}$ and $z_2 = \alpha_2e^{i\theta_2}$ are any two points in $\boldsymbol A_{r_1;r_2}$, $r_1 > 0$, and $|\theta_2 - \theta_1| \le \pi$, then
		\begin{equation} 
		|\theta_2 - \theta_1| \le L|z_2 - z_1|.
		\end{equation}
	\end{lemma}
	\begin{proof}
		Denote $z_1' = r_1e^{i\theta_1}$ and $z_2' = r_2e^{i\theta_2}$ and notice that $|z_2' - z_1'| \le |z_2 - z_1|$ (see Figure \ref{fig:geometric_annulus}). Hence it suffices to show that $|\theta_2 - \theta_1| \le L|z_2' - z_1'|$. Using the \textit{Cosine Rule} we can readily notice that $|z_2' - z_1'| = |z_1'|^2 + |z_2'|^2 - 2|z_1'||z_2'|\cos\left(\theta_2 - \theta_1\right) = 2r_1^2\left(1 - \cos\left(\theta_2 - \theta_1\right)\right) = 4r_1^2\sin^2\left(\frac{\theta_2 - \theta_1}{2}\right)$. Since we have assumed $|\theta_2 - \theta_1| \le \pi$, it follows that
		\begin{equation}\label{eq1 Lemma1}
		\sin\left(\frac{|\theta_2 - \theta_1|}{2}\right) = \frac{|z_2' - z_1'|}{2r_1}.
		\end{equation}
		The equality $\lim\limits_{t \searrow 0} \frac{\sin t}{t} = 1$ implies the existence of some sufficiently small $\epsilon_0 > 0$ such that $\sin t \ge (1-\epsilon_0)t \ge \frac{t}{2}$ as soon as $t \in [0,\epsilon_0]$. On the other hand if $t \in \left[\epsilon_0,\frac{\pi}{2}\right]$ then $\sin t \ge \sin \epsilon_0$ and thus $\frac{2\sin \epsilon_0}{\pi}t \le \sin\epsilon_0 \le \sin t$. Consequently define $L_0 = \min \left(\frac{2\sin \epsilon_0}{\pi},\frac{1}{2}\right)$, which shows that $\sin\left(\frac{|\theta_2 - \theta_1|}{2}\right) \ge L_0\frac{|\theta_2 - \theta_1|}{2}$ from where, using relation \eqref{eq1 Lemma1}, one obtains $L_0 \frac{|\theta_2 - \theta_1|}{2} \le \sin\left(\frac{|\theta_2 - \theta_1|}{2}\right) = \frac{|z_2' - z_1'|}{2r_1}$. Finally, putting
		\begin{equation}\label{the L constant}
		L = \frac{1}{L_0 r_1} = \frac{1}{r_1}\frac{1}{\min \left(\frac{2\sin \epsilon_0}{\pi},\frac{1}{2}\right)}
		\end{equation} 
		the lemma is proved.
	\end{proof}
	
	\begin{theorem}\label{main result cartesian}
		Let $f: \partial \boldsymbol A_{r_1;r_2} \rightarrow \mathbb{R}$ be a continuous function satisfying $\int\limits_{\partial \boldsymbol A_{r_1;r_2}} f d\sigma = 0$. 
		If $U$ is the solution of the Neumann problem \eqref{Neumann problem} with boundary data $f$, satisfying $U(\sqrt{r_1r_2})=0$, then for any point $re^{i\theta} \in \overline{\boldsymbol A_{r_1;r_2}}$
		\begin{equation}\label{U cartesian}
		U(re^{i\theta})=\int\limits_{\frac{\sqrt{r_{1}r_{2}}}{r}}^{1}\frac{u(\rho re^{i\theta})}{\rho }d\rho +\sqrt{r_{1}r_{2}}\int\limits_{0}^{\theta }\left(
		\mathcal{C} - \int\limits_{0}^{t}\frac{\partial u}{\partial \boldsymbol a_{\tau} }(\sqrt{r_{1}r_{2}}e^{i\tau})d\tau \right)
		dt,
		\end{equation}
		where $u$ is the solution of the Dirichlet problem \eqref{Dirichlet problem} with boundary values $g(z) = \begin{cases}r_2f(z) &~\text{if } |z|=r_2, \\ -r_1f(z) &~\text{if } |z|=r_1,\end{cases}$ and where the constant $\mathcal{C}$ is given by
		\begin{equation}\label{C reprezentare}
		\frac{\sqrt{r_{1}r_{2}}}{2\pi }\int\limits_{0}^{2\pi
		}\int\limits_{0}^{t}\frac{\partial u}{\partial \boldsymbol a_{\tau}}(\sqrt{r_{1}r_{2}} e^{i\tau})~ d\tau dt.
		\end{equation}
		If in addition $f \in C^{m,\alpha}(\partial \boldsymbol A_{r_1;r_2})$ for some positive integer $m \ge 2$ an some $\alpha \in (0,1]$, then $U$ given in \eqref{U cartesian} together with all its partial derivatives up to order $m+1$ can be continuously extended to $\overline{\boldsymbol A_{r_1,r_2}}$ and their extensions are uniformly H{\"o}lder continuous with exponent $\alpha$. Conversely if $g: \partial \boldsymbol A_{r_1,r_2} \rightarrow \mathbb{R}$ is a continuous function satisfying $\int\limits_0^{2\pi}g\left(r_2e^{i\theta}\right)~ d\theta = \int\limits_0^{2\pi}g\left(r_1e^{i\theta}\right)~ d\theta$, and if $U$ is a solution of the Neumann problem \eqref{Neumann problem} with boundary data $f(z) = \begin{cases}\frac{g(z)}{r_2} &~\text{if } |z|=r_2, \\ -\frac{g(z)}{r_1} &~\text{if } |z|=r_1,\end{cases} $ then the function
		\begin{equation}
		u(re^{i\theta}) = r\frac{\partial U}{\partial \boldsymbol a_{\theta}}(re^{i\theta}), ~re^{i\theta} \in \overline{\boldsymbol A_{r_1;r_2}}
		\end{equation}
		is the solution of the Dirichlet problem \eqref{Dirichlet problem} with boundary data $g$.
	\end{theorem} 
	\begin{proof}
		The only thing left to be proved is that if $f \in C^{m,\alpha}(\partial \boldsymbol A_{r_1;r_2})$ for some positive integer $m \ge 2$ an some $\alpha \in (0,1]$, then $U$ given in \eqref{U cartesian} together with all its partial derivatives up to order $m+1$ can be continuously extended to $\overline{\boldsymbol A_{r_1;r_2}}$ and their extensions are uniformly H{\"o}lder continuous with exponent $\alpha$. As explained earlier in the proof of Theorem \ref{Holder continuity polar theorem}, this proof will be done for the case $m=2$ as the case of a positive integer $m \ge 3$ follows exactly in the same way, using induction. To begin with assume first that $r_2 \le 1/2$ and let $\hat{U}$ be the solution of the Neumann problem \eqref{Neumann problem biss} on $[r_1,r_2] \times \mathbb{R}$ with boundary data $\phi(r,\theta) = \begin{cases}f(re^{i\theta}) \ &\text{if } r = r_2, \\ -f(re^{i\theta}) \ &\text{if } r = r_1, \end{cases}$ satisfying $\hat{U}(\sqrt{r_1r_2},0) = 0$. Then, since $f \in C^{m,\alpha}(\partial \boldsymbol A_{r_1;r_2})$, it follows by Theorem \ref{Holder continuity polar theorem} that $\hat{U}$ together with all its partial derivatives up to order $m+1$ can be continuously extended to $[r_1,r_2] \times \mathbb{R}$ and their extensions are uniformly H{\"o}lder continuous with exponent $\alpha$ there. We will transfer this property to $U$ in \eqref{U cartesian} and its partial derivatives up to order $m+1$ as follows. Choose any $z_1, ~ z_2 \in \overline{\boldsymbol A_{r_1;r_2}}$ and let $\alpha_1,\alpha_2, \in [r_1,r_2]$ and $\theta_1, \theta_2 \in \mathbb{R}$ be such that $z_1 = \alpha_1e^{i\theta_1}, z_2 = \alpha_2e^{i\theta_2}$, $|\theta_2 - \theta_1| \le \pi$. Notice then that $\hat{U}$ and $U$ given by \eqref{U cartesian} are related through the equation
		\begin{equation}\label{U U_hat Th4}
		\hat{U}(r,\theta) = U(re^{i\theta}), \ (r,\theta) \in (r_1,r_2) \times \mathbb{R}.
		\end{equation}
		Next $|U(z_2) - U(z_1)| = |\hat{U}(\alpha_2,\theta_2) - \hat{U}(\alpha_1,\theta_1)| \le H_U\left(|\alpha_2 - \alpha_1|^2 + |\theta_2 - \theta_1|^2\right)^{\frac{\alpha}{2}} \le H_U \left(|z_2 - z_1|^2 + L^2|z_2-z_1|^2\right)^{\frac{\alpha}{2}}$, for some positive constant $H_U$, where the first inequality is due to Theorem \ref{Holder continuity polar theorem} whereas the second one is due to Lemma \ref{theta-z lemma}. Hence 
		\begin{equation}
		|U(z_2) - U(z_1)| \le H_U(1+L^2)^{\frac{\alpha}{2}}|z_2 - z_1|^{\alpha},	
		\end{equation}
		which proves that $U$ is indeed uniformly H{\"o}lder continuous with exponent $\alpha$ on $\overline{\boldsymbol A_{r_1;r_2}}$. Proceeding further take the derivatives with respect to $r$ and $\theta$ in \eqref{U U_hat Th4} and letting $z = (r\cos\theta,r\sin\theta)$ we obtain
		\begin{equation}\label{polar 2 cartesian first derivatives Th4}
		\begin{cases}
		U_{\boldsymbol x}(z) = \hat{U}_{\boldsymbol r}(r,\theta)\cos\theta - \frac{1}{r}\hat{U}_{\boldsymbol \theta}(r,\theta)\sin\theta, \\
		U_{\boldsymbol y}(z) = \hat{U}_{\boldsymbol r}(r,\theta)\sin\theta + \frac{1}{r}\hat{U}_{\boldsymbol \theta}(r,\theta)\cos\theta.
		\end{cases}
		\end{equation}
		To prove the locally $\alpha$ H{\"o}lder continuity property on $\overline{\boldsymbol A_{r_1;r_2}}$ for the partial derivatives of $U$ up to order three, assume in addition that $z_1$ and $z_2$ are close enough. Consequently compute $|U_{\boldsymbol x}(z_2) - U_{\boldsymbol x}(z_1)| \le |\hat{U}_{\boldsymbol r}(\alpha_2,\theta_2)\cos\theta_2 - \hat{U}_{\boldsymbol r}(\alpha_1,\theta_1)\cos\theta_1|+|\frac{1}{\alpha_2}\hat{U}_{\boldsymbol \theta}(\alpha_2,\theta_2)\sin\theta_2 - \frac{1}{\alpha_1}\hat{U}_{\boldsymbol \theta}(\alpha_1,\theta_1)\sin\theta_1| \le |\hat{U}_{\boldsymbol r}(\alpha_2,\theta_2) - \hat{U}_{\boldsymbol r}(\alpha_1,\theta_1)| + |\hat{U}_{\boldsymbol r}(\alpha_1,\theta_1)|\cdot$ \\
		$|\cos\theta_2 - \cos\theta_1| + \frac{1}{\alpha_2}|\hat{U}_{\boldsymbol \theta}(\alpha_2,\theta_2) - \hat{U}_{\boldsymbol \theta}(\alpha_1,\theta_1)| + |\hat{U}_{\boldsymbol \theta}(\alpha_1,\theta_1)|\left|\frac{\sin\theta_2}{\alpha_2} - \frac{\sin\theta_1}{\alpha_1}\right|$. Focusing on the last inequality the first term is upper-bounded by $H_U(1+L^2)^{\frac{\alpha}{2}}|z_2 - z_1|^{\alpha}$, the second term is upper-bounded by $L\|\hat{U}_{\boldsymbol r}\| |z_2 - z_1|$, the third term is upper-bounded by $(1+L^2)^{\frac{\alpha}{2}}\frac{H_U}{r_1}|z_2 - z_1|^{\alpha}$, and lastly the fourth term is upper-bounded by $\frac{1+Lr_2}{r_1^2}\|\hat{U}_{\boldsymbol \theta}\||z_2 - z_1|$. Since $r_2$ was assumed to be less than $0.5$ it follows that any $z_1, ~ z_2 \in \overline{\boldsymbol A_{r_1,r_2}}$ satisfy $|z_2 - z_1| \le |z_2 - z_1|^{\alpha}$ and thus 
		\begin{equation}
		|U_{\boldsymbol x}(z_2) - U_{\boldsymbol x}(z_1)| \lesssim |z_2 - z_1|^{\alpha},
		\end{equation}
		where a proportionality constant is readily found to be
		\begin{equation*}
		H_U(1+L^2)^{\frac{\alpha}{2}}\left(1+\frac{1}{r_1}\right) + L\|\hat{U}_{\boldsymbol r}\| + \frac{1+Lr_2}{r_1^2}\|\hat{U}_{\boldsymbol \theta}\|.
		\end{equation*}
		In a similar way compute $|U_{\boldsymbol y}(z_2) - U_{\boldsymbol y}(z_1)| \le |\hat{U}_{\boldsymbol r}(\alpha_2,\theta_2)\sin\theta_2 - \hat{U}_{\boldsymbol r}(\alpha_1,\theta_1)\sin\theta_1|+|\frac{1}{\alpha_2}\hat{U}_{\boldsymbol \theta}(\alpha_2,\theta_2)\cos\theta_2 - \frac{1}{\alpha_1}\hat{U}_{\boldsymbol \theta}(\alpha_1,\theta_1)\cos\theta_1| \le |\hat{U}_{\boldsymbol r}(\alpha_2,\theta_2) - \hat{U}_{\boldsymbol r}(\alpha_1,\theta_1)| + |\hat{U}_{\boldsymbol r}(\alpha_1,\theta_1)|\cdot$ \\
		$|\sin\theta_2 - \sin\theta_1| + \frac{1}{\alpha_2}|\hat{U}_{\boldsymbol \theta}(\alpha_2,\theta_2) - \hat{U}_{\boldsymbol \theta}(\alpha_1,\theta_1)| + |\hat{U}_{\boldsymbol \theta}(\alpha_1,\theta_1)|\left|\frac{\cos\theta_2}{\alpha_2} - \frac{\cos\theta_1}{\alpha_1}\right|$, and notice that all four terms have the same upper bounds as above. Hence one can conclude that 
		\begin{equation}
		|U_{\boldsymbol y}(z_2) - U_{\boldsymbol y}(z_1)| \lesssim |z_2 - z_1|^{\alpha},
		\end{equation}
		where a proportionality constant is thus
		\begin{equation*}
		H_U(1+L^2)^{\frac{\alpha}{2}}\left(1+\frac{1}{r_1}\right) + L\|\hat{U}_{\boldsymbol r}\| + \frac{1+Lr_2}{r_1^2}\|\hat{U}_{\boldsymbol \theta}\|.
		\end{equation*}
		
		To show the locally $\alpha$ H{\"o}lder continuity property on $\overline{\boldsymbol A_{r_1;r_2}}$ for the second order partial derivatives of $U$ plug $z = (r\cos\theta,r\sin\theta)$ in \eqref{polar 2 cartesian first derivatives Th4} and take again derivatives with respect to both $r$ and $\theta$ to obtain after some elementary algebraic manipulations
		\begin{align}\label{second order partial derivatives U cartesian Th4}
		U_{\boldsymbol x \boldsymbol x}(z) &= \hat{U}_{\boldsymbol r \boldsymbol r}(r,\theta)\cos^2\theta - \hat{U}_{\boldsymbol r \boldsymbol \theta}(r,\theta)\frac{\sin 2\theta}{r} + \hat{U}_{\boldsymbol \theta \boldsymbol \theta}(r,\theta)\frac{\sin^2 \theta}{r^2} \nonumber \\
		& + \hat{U}_{\boldsymbol \theta}(r,\theta)\frac{\cos\theta(1 + \sin\theta)}{r^2} + \hat{U}_{\boldsymbol r}(r,\theta)\frac{\sin^2 \theta}{r}, \nonumber \\
		U_{\boldsymbol x \boldsymbol y}(z) &= \hat{U}_{\boldsymbol r \boldsymbol r}(r,\theta)\frac{\sin 2\theta}{2} + \hat{U}_{\boldsymbol r \boldsymbol \theta}(r,\theta)\frac{\cos 2\theta}{r} - \hat{U}_{\boldsymbol \theta \boldsymbol \theta}(r,\theta)\frac{\sin 2\theta}{2r^2} \nonumber \\
		& - \hat{U}_{\boldsymbol \theta}(r,\theta)\left(\frac{\sin\theta}{r} + \frac{\cos^2 \theta}{r^2}\right) - \hat{U}_{\boldsymbol r}(r,\theta) \frac{\sin 2\theta}{2r}, \\
		U_{\boldsymbol y \boldsymbol y}(z) &= \hat{U}_{\boldsymbol r \boldsymbol r}(r,\theta)\sin^2\theta + \hat{U}_{\boldsymbol r \boldsymbol \theta}(r,\theta)\frac{\sin 2\theta}{r} + \hat{U}_{\boldsymbol \theta \boldsymbol \theta}(r,\theta)\frac{\cos^2 \theta}{r^2} \nonumber \\
		& - \hat{U}_{\boldsymbol \theta}(r,\theta)\frac{\sin 2\theta}{2}\left(1+\frac{1}{r^2}\right) + \hat{U}_{\boldsymbol r}(r,\theta)\frac{\cos^2 \theta}{r}. \nonumber
		\end{align} 
		So $|U_{\boldsymbol x \boldsymbol x}(z_2) - U_{\boldsymbol x \boldsymbol x}(z_1)| \le |\hat{U}_{\boldsymbol r \boldsymbol r}(\alpha_2,\theta_2) - \hat{U}_{\boldsymbol r \boldsymbol r}(\alpha_1,\theta_1)| + |\cos^2\theta_2 - \cos^2\theta_1|
		|\hat{U}_{\boldsymbol r \boldsymbol r}(\alpha_1,\theta_1)|$ \\
		$+ \frac{1}{\alpha_2}|\hat{U}_{\boldsymbol r \boldsymbol \theta}(\alpha_2,\theta_2) - \hat{U}_{\boldsymbol r \boldsymbol \theta}(\alpha_1,\theta_1)| + \left|\frac{\sin 2\theta_2}{\alpha_2} - \frac{\sin 2\theta_1}{\alpha_1}\right||\hat{U}_{\boldsymbol r \boldsymbol \theta}(\alpha_1,\theta_1)| + \frac{1}{\alpha_2^2}
		|\hat{U}_{\boldsymbol \theta \boldsymbol \theta}(\alpha_2,\theta_2) - \hat{U}_{\boldsymbol \theta \boldsymbol \theta}(\alpha_1,\theta_1)| + \left|\frac{\sin^2 \theta_2}{\alpha_2} - \frac{\sin^2 \theta_1}{\alpha_1}\right|\left|\hat{U}_{\boldsymbol \theta \boldsymbol \theta}(\alpha_1,\theta_1)\right| + \frac{2}{\alpha_2^2}\left|\hat{U}_{\boldsymbol \theta}(\alpha_2,\theta_2) - \hat{U}_{\boldsymbol \theta}(\alpha_1,\theta_1)\right|+$ \\ $|\hat{U}_{\boldsymbol \theta}(\alpha_1,\theta_1)| \left|\frac{\cos\theta_2(1+\sin\theta_2)}{\alpha_2^2} - \frac{\cos\theta_1(1+\sin\theta_1)}{\alpha_1^2}\right| + |\frac{\hat{U}_{\boldsymbol r}(\alpha_2,\theta_2)}{\alpha_2} - \frac{\hat{U}_{\boldsymbol r}(\alpha_1,\theta_1)}{\alpha_2}| + |\hat{U}_{\boldsymbol r}(\alpha_1,\theta_1)|\cdot$ \\
		$\left|\frac{\sin^2\theta_2}{\alpha_2} -  \frac{\sin^2\theta_1}{\alpha_1}\right|$. Let $\mathcal{H}> 0$ be such that $\mathcal{H}\left(|\alpha_2-\alpha_1|^2 + |\theta_2 - \theta_1|^2\right)^{\frac{\alpha}{2}}$ upper-bounds $|\hat{U}(\alpha_2,\theta_2)-\hat{U}(\alpha_1,\theta_1)|$, $|\hat{U}_{\boldsymbol r}(\alpha_2,\theta_2)-\hat{U}_{\boldsymbol r}(\alpha_1,\theta_1)|$, $|\hat{U}_{\boldsymbol \theta}(\alpha_2,\theta_2)-\hat{U}_{\boldsymbol \theta}(\alpha_1,\theta_1)|$, $|\hat{U}_{\boldsymbol r \boldsymbol r}(\alpha_2,\theta_2)-\hat{U}_{\boldsymbol r \boldsymbol r}(\alpha_1,\theta_1)|$, $|\hat{U}_{\boldsymbol r \boldsymbol \theta}(\alpha_2,\theta_2)-\hat{U}_{\boldsymbol r \boldsymbol \theta}(\alpha_1,\theta_1)|$, $|\hat{U}_{\boldsymbol \theta \boldsymbol \theta}(\alpha_2,\theta_2)-\hat{U}_{\boldsymbol \theta \boldsymbol \theta}(\alpha_1,\theta_1)|$, as well as $|\hat{U}_{\boldsymbol a \boldsymbol b \boldsymbol c}(\alpha_2,\theta_2) - \hat{U}_{\boldsymbol a \boldsymbol b \boldsymbol c}(\alpha_1,\theta_1)|$, $\boldsymbol a \boldsymbol b \boldsymbol c \in \{\boldsymbol r,\boldsymbol \theta\}$, for any $(\alpha_1,\theta_1), ~(\alpha_2,\theta_2) \in [r_1,r_2] \times \mathbb{R}$ (see Theorem \ref{Holder continuity polar theorem}). Consequently using simple algebraic manipulations it is easy to check that
		\begin{equation}
		|U_{\boldsymbol x \boldsymbol x}(z_2) - U_{\boldsymbol x \boldsymbol x}(z_1)| \lesssim |z_2 - z_1|^{\alpha},
		\end{equation}
		where a proportionality constant is
		\begin{align*}
		2L\|\hat{U}_{\boldsymbol r \boldsymbol r}\| + \frac{2Lr_2+1}{r_1^2}\|\hat{U}_{\boldsymbol r \boldsymbol \theta}\| &+ \frac{2Lr_2+1}{r_1^2}\|\hat{U}_{\boldsymbol \theta \boldsymbol \theta}\| + \frac{3Lr_2^2 + 4r_2}{r_1^4}\|\hat{U}_{\boldsymbol \theta}\| + \frac{2Lr_2+1}{r_1^2}\cdot \\
		&\|\hat{U}_{\boldsymbol r}\| +\mathcal{H}(L^2+1)^{\frac{\alpha}{2}}\left(1+\frac{2}{r_1}+\frac{3}{r_1^2}\right).
		\end{align*}
		Similarly $|U_{\boldsymbol x \boldsymbol y}(z_2) - U_{\boldsymbol x \boldsymbol y}(z_1)| \le \frac{1}{2}|\hat{U}_{\boldsymbol r \boldsymbol r}(\alpha_2,\theta_2) - \hat{U}_{\boldsymbol r \boldsymbol r}(\alpha_1,\theta_1)| + |\theta_2 - \theta_1||\hat{U}_{\boldsymbol r \boldsymbol r}(\alpha_1,\theta_1)|$ \\
		$+ \frac{1}{\alpha_2}|\hat{U}_{\boldsymbol r \boldsymbol \theta}(\alpha_2,\theta_2) - \hat{U}_{\boldsymbol r \boldsymbol \theta}(\alpha_1,\theta_1)| + \left|\frac{\cos 2\theta_2}{\alpha_2} - \frac{\cos 2\theta_1}{\alpha_1}\right||\hat{U}_{\boldsymbol r \boldsymbol \theta}(\alpha_1,\theta_1)| + \frac{1}{2\alpha_2^2}|\hat{U}_{\boldsymbol \theta \boldsymbol \theta}(\alpha_2,\theta_2) - \hat{U}_{\boldsymbol \theta \boldsymbol \theta}(\alpha_1,\theta_1)| + \left| \frac{\sin 2\theta_2}{2\alpha_2^2} - \frac{\sin 2\theta_1}{2\alpha_1^2}\right||\hat{U}_{\boldsymbol \theta \boldsymbol \theta}(\alpha_1,\theta_1)| + \left(\frac{1}{\alpha_2} + \frac{1}{\alpha_2^2} \right)|\hat{U}_{\boldsymbol \theta}(\alpha_2,\theta_2) - \hat{U}_{\boldsymbol \theta}(\alpha_1,\theta_1)| + \left(\left|\frac{\sin \theta_2}{\alpha_2} - \frac{\sin 2\theta_1}{\alpha_1}\right| + \left|\frac{\cos^2 \theta_2}{\alpha_2^2} - \frac{\cos^2 \theta_1}{\alpha_1^2}\right|\right)|\hat{U}_{\boldsymbol \theta}(\alpha_1,\theta_1)| + \frac{1}{2\alpha_2}|\hat{U}_{\boldsymbol r}(\alpha_2,\theta_2) - \hat{U}_{\boldsymbol r}(\alpha_1,\theta_1)| + \left|\frac{\sin 2\theta_2}{2\alpha_2} - \frac{\sin 2\theta_1}{2\alpha_1}\right||\hat{U}_{\boldsymbol r}(\alpha_1,\theta_1)|$, and it follows that
		\begin{equation}
		|U_{\boldsymbol x \boldsymbol y}(z_2) - U_{\boldsymbol x \boldsymbol y}(z_1)| \lesssim |z_2 - z_1|^{\alpha},
		\end{equation}
		where a proportionality constant is found to be
		\begin{align*}
		L\|\hat{U}_{\boldsymbol r \boldsymbol r}\| &+ \frac{2Lr_2+1}{r_1^2}\|\hat{U}_{\boldsymbol r \boldsymbol \theta}\| + \frac{Lr_2^2+r_2}{r_1^4}\|\hat{U}_{\boldsymbol \theta \boldsymbol \theta}\| + \|\hat{U}_{\boldsymbol r}\|\frac{2Lr_2+1}{r_1^2} \\ 
		&+ \frac{Lr_2(r_1^2+2r_2+1)+2r_2 +r_1^2}{r_1^4}\|\hat{U}_{\boldsymbol \theta}\| + \mathcal{H}\frac{r_1^2 + 3r_1 + 3}{2r_1^2}(L^2+1)^{\frac{\alpha}{2}}.
		\end{align*}
		Finally $|U_{\boldsymbol y \boldsymbol y}(z_2) - U_{\boldsymbol y \boldsymbol y}(z_1)| \le |\hat{U}_{\boldsymbol r \boldsymbol r}(\alpha_2,\theta_2) - \hat{U}_{\boldsymbol r \boldsymbol r}(\alpha_1,\theta_1)| + 2|\theta_2 - \theta_1||\hat{U}_{\boldsymbol r \boldsymbol r}(\alpha_1,\theta_1)| + \frac{1}{\alpha_2}|\hat{U}_{\boldsymbol r \boldsymbol \theta}(\alpha_2,\theta_2) - \hat{U}_{\boldsymbol r \boldsymbol \theta}(\alpha_1,\theta_1)| + \left|\frac{\sin 2\theta_2}{\alpha_2} - \frac{\sin 2\theta_1}{\alpha_1}\right||\hat{U}_{\boldsymbol r \boldsymbol \theta}(\alpha_1,\theta_1)| + \frac{1}{\alpha_2^2}|\hat{U}_{\boldsymbol \theta \boldsymbol \theta}(\alpha_2,\theta_2) - \hat{U}_{\boldsymbol \theta \boldsymbol \theta}(\alpha_1,\theta_1)| + \left| \frac{\cos^2 \theta_2}{\alpha_2^2} - \frac{\cos^2 \theta_1}{\alpha_1^2}\right||\hat{U}_{\boldsymbol \theta \boldsymbol \theta}(\alpha_1,\theta_1)| + \left(\frac{1}{2} + \frac{1}{2\alpha_2^2} \right)|\hat{U}_{\boldsymbol \theta}(\alpha_2,\theta_2) - \hat{U}_{\boldsymbol \theta}(\alpha_1,\theta_1)| + \frac{1}{2}\left|\sin 2\theta_2\left( 1+\frac{1}{\alpha_2^2} \right) - \sin 2\theta_1 \left( 1+\frac{1}{\alpha_1^2} \right) \right||\hat{U}_{\boldsymbol \theta}(\alpha_1,\theta_1)| + \frac{1}{\alpha_2}|\hat{U}_{\boldsymbol r}(\alpha_2,\theta_2) - \hat{U}_{\boldsymbol r}(\alpha_1,\theta_1)| + \left|\frac{\cos^2 \theta_2}{\alpha_2} - \frac{\cos^2 \theta_1}{\alpha_1}\right||\hat{U}_{\boldsymbol r}(\alpha_1,\theta_1)|$, and thus one obtains
		\begin{equation}
		|U_{\boldsymbol y \boldsymbol y}(z_2) - U_{\boldsymbol y \boldsymbol y}(z_1)| \lesssim |z_2 - z_1|^{\alpha},
		\end{equation}
		where a proportionality constant can be taken as
		\begin{align*}
		\frac{2Lr_2+1}{r_1^2}\|\hat{U}_{\boldsymbol r \boldsymbol \theta}\| &+ \frac{2Lr_2^2+r_2}{r_1^4}\|\hat{U}_{\boldsymbol \theta \boldsymbol \theta}\| +\|\hat{U}_{\boldsymbol r}\|\frac{2Lr_2+1}{r_1^2} + \frac{Lr_1^2(r_1^2+1)+r_2}{r_1^4}\|\hat{U}_{\boldsymbol \theta}\| \\
		&+ 2L\|\hat{U}_{\boldsymbol r \boldsymbol r}\| + \mathcal{H}\frac{2r_1^2+4r_1+3}{2r_1^2}(L^2+1)^{\frac{\alpha}{2}}.
		\end{align*}
		
		To show the locally $\alpha$ H{\"o}lder continuity property on $\overline{\boldsymbol A_{r_1;r_2}}$ for the third order partial derivatives of $U$ plug $z = (r\cos\theta,r\sin\theta)$ in the first and last equations of \eqref{second order partial derivatives U cartesian Th4} and take again derivatives with respect to both $r$ and $\theta$ to obtain after some algebraic manipulations 
		\begin{align*}
		U_{\boldsymbol x \boldsymbol x \boldsymbol x}&(z) = \hat{U}_{\boldsymbol r \boldsymbol r \boldsymbol r}(r,\theta)\cos^3\theta - \frac{3}{2}\hat{U}_{\boldsymbol \theta \boldsymbol r \boldsymbol r}(r,\theta)\frac{\sin 2\theta \cos\theta}{r} + \frac{\sin^2\theta(1+2\cos\theta)}{r^2}\cdot \\
		&\hat{U}_{\boldsymbol \theta \boldsymbol \theta \boldsymbol r}(r,\theta) - \hat{U}_{\boldsymbol \theta \boldsymbol \theta \boldsymbol \theta}(r,\theta)\frac{\sin^3\theta}{r^3} + \hat{U}_{\boldsymbol \theta \boldsymbol r}(r,\theta)\frac{3\cos^2\theta \sin\theta + \cos^2\theta + 2\cos 2\theta \sin\theta}{r^2} \\
		&+ \hat{U}_{\boldsymbol r \boldsymbol r}(r,\theta)\frac{3\sin^2\theta \cos\theta}{r} + \hat{U}_{\boldsymbol \theta \boldsymbol \theta}(r,\theta)\frac{\sin^2\theta (1+\sin\theta+3\cos\theta)}{r^3} - \frac{3\sin^2\theta \cos\theta}{r^2}\cdot \\
		&\hat{U}_{\boldsymbol r}(r,\theta) - \hat{U}_{\boldsymbol \theta}(r,\theta)\frac{2\cos^2\theta(1+\sin\theta) - \sin^2\theta + \sin\theta \cos 2\theta}{r^3},
		\end{align*}
		\begin{align}
		U_{\boldsymbol x \boldsymbol x \boldsymbol y}&(z) = \hat{U}_{\boldsymbol \theta \boldsymbol \theta \boldsymbol \theta}(r,\theta)\frac{\sin^2\theta \cos\theta}{r^3} + \hat{U}_{\boldsymbol \theta \boldsymbol r}(r,\theta)\frac{2\sin 2\theta \sin\theta + \cos\theta (\sin\theta - 2\cos 2\theta)}{r^2} \nonumber\\
		&+ \hat{U}_{\boldsymbol \theta \boldsymbol \theta}(r,\theta)\frac{\sin 2\theta \cos\theta - 2\sin^3\theta - \sin 2\theta (1+\sin\theta)}{r^3} + \hat{U}_{\boldsymbol r \boldsymbol r \boldsymbol r}(r,\theta)\cos^2\theta \sin\theta \nonumber \\	
		&+ \hat{U}_{\boldsymbol \theta \boldsymbol r \boldsymbol r}(r,\theta)\frac{\cos^3\theta - \sin 2\theta \sin\theta}{r} + \hat{U}_{\boldsymbol \theta}(r,\theta)\frac{\cos\theta (\cos 2\theta - 2\sin\theta - \sin^2\theta)}{r^3} \nonumber \\
		&+ \hat{U}_{\boldsymbol r \boldsymbol r}(r,\theta)\frac{\sin\theta(\sin^2\theta - 2\cos^2\theta)}{r} + \hat{U}_{\boldsymbol \theta \boldsymbol \theta \boldsymbol r}(r,\theta)\frac{\sin^3\theta - \sin 2\theta \cos\theta}{r^2} + \hat{U}_{\boldsymbol r}(r,\theta)\cdot \nonumber \\
		&\frac{\sin\theta (2\cos^2\theta - \sin^2\theta)}{r^2},
		\end{align}
		\begin{align*}
		U_{\boldsymbol y \boldsymbol y \boldsymbol x}&(z) = - \hat{U}_{\boldsymbol \theta \boldsymbol r}(r,\theta)\frac{\sin\theta \cos\theta + \sin\theta \cos^2\theta(r^2+1) + 2\cos 2\theta \sin\theta + \cos^2\theta \sin\theta}{r^2} \\
		&+ \hat{U}_{\boldsymbol \theta}(r,\theta)\frac{\sin 2\theta \cos\theta + \cos 2\theta \sin\theta(r^2+1)}{r^3} + \hat{U}_{\boldsymbol \theta \boldsymbol r \boldsymbol r}(r,\theta)\frac{\sin 2\theta \cos\theta - \sin^3 \theta}{r} \\
		&+ \hat{U}_{\boldsymbol \theta \boldsymbol \theta \boldsymbol r}(r,\theta)\frac{\cos^3\theta - \sin 2\theta \sin\theta}{r^2} - \hat{U}_{\boldsymbol \theta \boldsymbol \theta \boldsymbol \theta}(r,\theta)\frac{\cos^2\theta \sin\theta}{r^3} + \sin^2\theta \cos\theta\cdot \\
		&\hat{U}_{\boldsymbol r \boldsymbol r \boldsymbol r}(r,\theta) + \hat{U}_{\boldsymbol r \boldsymbol r}(r,\theta)\frac{\cos^3\theta - \sin 2\theta \sin\theta}{r} + \frac{\sin^2\theta \cos\theta (r^2+3) - 2\cos^3\theta}{r^3}\cdot \\
		&\hat{U}_{\boldsymbol \theta \boldsymbol \theta}(r,\theta) + \hat{U}_{\boldsymbol r}(r,\theta)\frac{\sin 2\theta \sin\theta - \cos^3\theta}{r^2},
		\end{align*}
		\begin{align*}
		U_{\boldsymbol y \boldsymbol y \boldsymbol y}&(z) = \hat{U}_{\boldsymbol \theta \boldsymbol r}(r,\theta)\frac{\cos^3\theta + 2\cos2\theta \cos\theta - \sin^2\theta - \cos\theta \sin^2\theta (r^2+1)}{r^2} + \sin^3\theta\cdot \\
		&\hat{U}_{\boldsymbol r \boldsymbol r \boldsymbol r}(r,\theta) + \hat{U}_{\boldsymbol \theta}(r,\theta)\frac{\sin 2\theta \sin\theta - \cos 2\theta \cos\theta (r^2+1)}{r^3} + \frac{3\cos^2\theta \sin\theta}{r}\cdot \\
		&\hat{U}_{\boldsymbol r \boldsymbol r}(r,\theta) - \hat{U}_{\boldsymbol \theta \boldsymbol \theta}(r,\theta)\frac{\cos^2\theta \sin\theta(r^2+5)}{r^3} + \hat{U}_{\boldsymbol \theta \boldsymbol r \boldsymbol r}(r,\theta)\frac{3\sin^2\theta \cos\theta}{r} + \frac{\cos^3\theta}{r^3}\cdot \\
		&\hat{U}_{\boldsymbol \theta \boldsymbol \theta \boldsymbol \theta}(r,\theta) + \hat{U}_{\boldsymbol \theta \boldsymbol \theta \boldsymbol r}(r,\theta)\frac{3\cos^2\theta \sin\theta}{r^2} + \hat{U}_{\boldsymbol r}(r,\theta)\frac{3\cos^2\theta \sin\theta}{r^2}.
		\end{align*}
		Using the above relations and the same ideas as for the first and second order partial derivatives of $U$, it can be checked that the third order partial derivatives of $U$ are locally $\alpha$ H{\"o}lder continuous on $\overline{\boldsymbol A_{r_1;r_2}}$ as well. 
		
		To sum up, it has been proved that the partial derivatives of $U$ up to order three are locally $\alpha$ H{\"o}lder continuous on $\overline{\boldsymbol A_{r_1;r_2}}$, and a compactness argument shows that this is enough to argue that they are in fact uniform H{\"o}lder continuous with exponent $\alpha$. In conclusion, the theorem is now proved for the case when $r_2 \le 0.5$. If $r_2 > 0.5$ performing a scaling of the annulus by $(2r_2)^{-1}$ and applying the previous conclusions to the function $U_s(w) := U(2r_2w)$ defined on the scaled annulus, it is immediately seen that $U$ together with all its partial derivatives up to order three are uniformly H{\"o}lder continuous with exponent $\alpha$ on $\overline{\boldsymbol A_{r_1,r_2}}$. The proof is now completed. 
	\end{proof}
	
	\begin{remark} The constant $\mathcal{C}$ which appears in both \eqref{first appearance of C} and \eqref{C reprezentare} has an interesting interpretation. Cutting the annulus $\boldsymbol A_{r_1,r_2}$ along the negative real axis for example, that is defining $\boldsymbol A_{r_1,r_2}^- = \{z \in \boldsymbol A_{r_1,r_2}| z \not\in \mathbb{R}_- \}$, the solution $u$ of the Dirichlet problem \eqref{Dirichlet problem} having boundary data $g(z) = \begin{cases}r_2f(z) &~\text{if } |z|=r_2, \\ -r_1f(z) &~\text{if } |z|=r_1,\end{cases}$ has an harmonic-conjugate function $v_0$ on $\boldsymbol A_{r_1,r_2}^-$ satisfying $v_0(\sqrt{r_1r_2}) = 0$. Then it is claimed that
		\begin{equation}
		\mathcal{C} = \frac{1}{2\pi}\int\limits_{-\pi}^{\pi}v_0(\sqrt{r_1r_2}e^{i\theta})~ d\theta.
		\end{equation} 
		Indeed assume first $r_1 = 1/a < a = r_2$ and denoting $\hat{u}(r,\theta) = u(re^{i\theta}), ~(r,\theta) \in \left[\frac{1}{a},a\right] \times \mathbb{R}$, we have $\mathcal{C} = \frac{1}{2\pi}\int\limits_{0}^{2\pi}\int\limits_0^t \frac{\partial u}{\partial a_{\boldsymbol \tau}}(e^{i\tau}) d\tau dt = \frac{1}{2\pi}\int\limits_{0}^{2\pi}\int\limits_0^t \hat{u}_{\boldsymbol r}(1,\tau) d\tau dt$. According to the proof of Theorem \ref{main theorem} the application $t \rightarrow \int\limits_0^t \hat{u}_{\boldsymbol r}(1,\tau)d\tau$ is $2\pi$-periodic and thus $\mathcal{C} = \frac{1}{2\pi}\int\limits_{-\pi}^0\int\limits_0^t \hat{u}_{\boldsymbol r}(1,\tau) d\tau dt + \frac{1}{2\pi}\int\limits_0^{\pi}\int\limits_0^t \hat{u}_{\boldsymbol r}(1,\tau) d\tau dt = -\frac{1}{2\pi}\int\limits_{-\pi}^0\int\limits_{\gamma_t^-}d^*u + \frac{1}{2\pi}\int\limits_0^{\pi}\int\limits_{\gamma_t^+}d^*u$, where $\gamma_t^-:[t,0] \rightarrow \mathbb{C}, \ \gamma_t^-(\tau) = e^{i\tau}, \ t< 0$, $\gamma_t^+:[0,t] \rightarrow \mathbb{C}, \ \gamma_t^+(\tau) = e^{i\tau}, \ t \ge 0$, and where $d^*u$ denotes the conjugate differential of $u$ (see \cite[Chapter 4.6.1]{Ahlfors}). But $-\int\limits_{\gamma_t^-}d^*u = v_0(e^{it}) - v_0(1)$ and also $\int\limits_{\gamma_t^+}d^*u = v_0(e^{it}) - v_0(1)$, and since $v_0(1) = 0$ the proof is complete for the case when $r_1 = 1/a < a = r_2$. The general case $0 < r_1 < r_2 < \infty$ follows by the previous one by performing a scaling with $\lambda = \frac{1}{\sqrt{r_1r_2}}$ and defining $a = \sqrt{\frac{r_2}{r_1}}$.
	\end{remark}
	
	The next corollary shows that Theorem \ref{main result cartesian} (or equivalently Theorem \ref{main theorem}) is a generalization of the main result in \cite{Beznea1} (actually its first part is exactly Theorem 1 in \cite{Beznea1} when the unit ball has dimension $2$). This will show, in particular, that the theory presented so far is a more powerful tool in $\mathbb{R}^2$ which embeds the main result in \cite{Beznea1} as a particular case. Moreover, if additional assumptions on the smoothness of $f$ are provided, the result in \cite{Beznea1} can be strengthened to uniform Holder continuity of $U$ and its partial derivatives.

	\begin{corollary}\label{main corollary}
		Assume $f: \partial \mathbb{U} \rightarrow \mathbb{R}$ is continuous and satisfies $\int\limits_0^{2\pi}f~d\theta = 0$. If $U$ is the solution of the Neumann problem \eqref{Neumann problem} on $\mathbb{U}$ with boundary data $f$, satisfying $U(0)=0$, then 
		\begin{equation}\label{U Corl2}
		U(z) = \int\limits_0^1 \frac{u(\rho z)}{\rho}d\rho, ~z \in \overline{\mathbb{U}},
		\end{equation}
		where $u$ is the solution of the Dirichlet problem \eqref{Dirichlet problem} on $\mathbb{U}$ with boundary data $g=f$. If $f \in C^{m,\alpha}(\partial \mathbb{U})$ for some positive integer $m \ge 2$ and some $\alpha \in (0,1]$ then $U$ and all its partial derivatives up to order $m+1$ are uniformly H{\"o}lder continuous with exponent $\alpha$ on $\mathbb{U}$. Conversely if $g: \partial \mathbb{U} \rightarrow \mathbb{R}$ is continuous, satisfies $\int\limits_0^{2\pi} g ~d\theta = 0$, and if $U$ is a solution of the Neumann problem \eqref{Neumann problem} on $\mathbb{U}$ with boundary data $f = g$, then the solution $u$ of the Dirichlet problem \eqref{Dirichlet problem} on $\mathbb{U}$ with boundary data $g$ is given by
		\begin{equation}
		u(re^{i\theta}) = r\frac{\partial U}{\partial \boldsymbol a_{\theta}}(re^{i\theta}), ~re^{i\theta} \in \overline{\mathbb{U}}.
		\end{equation}
	\end{corollary}
	\begin{proof}
		Define $r_n = r_1(n) = \frac{1}{n^2}, \ \boldsymbol A_n = \boldsymbol A_{r_n;1}, ~ n\in \mathbb{N} \setminus \{0,1\}$. On $\boldsymbol A_n$ let $u_n(\cdot)$ be the solution of the Dirichlet problem \eqref{Dirichlet problem} with boundary data $g_n = u_{| \partial \boldsymbol A_n}$. By the uniqueness of the solution of the Dirichlet problem it follows that $u_n = u$ on $\boldsymbol A_n$. Next define $U_n(re^{i\theta}) = \int\limits_{\frac{1}{nr}}^1\frac{u_n(\rho re^{i\theta})}{\rho}d\rho + \frac{1}{n}\int\limits_0^{\theta}\left(\mathcal{C}_n - \int\limits_0^t \frac{\partial u_n}{\partial \boldsymbol a_{\tau}}(\frac{e^{i\tau}}{n})d\tau \right)dt, ~re^{i\theta} \in \overline{\boldsymbol A_n}$, where $\mathcal{C}_n = \frac{1}{2\pi n}\int\limits_0^{2\pi}\int\limits_0^t \frac{\partial u_n}{\partial \boldsymbol a_{\tau}}(\frac{e^{i\tau}}{n})d\tau dt$. It follows by Theorem \ref{main result cartesian} that $U_n$ is harmonic in $\boldsymbol A_n$, has normal derivative $f_n(z) = \begin{cases}f(z) \ &\text{if } |z| = 1, \\
		-n^2u(z) \ &\text{if } |z| = \frac{1}{n^2}, \end{cases}$ and satisfies $U_n(\frac{1}{n}) = 0$. Further let $K \subset \boldsymbol A_{0;1}$ be any compact set. Hence $\exists N \in \mathbb{N} \setminus \{0;1\}$ (possibly depending on $K$) such that $ \forall n \ge N \Rightarrow K \subset \boldsymbol A_n$. So choose any $n \ge N$ and any $p \in \mathbb{N}^*$, and consequently $|U_{n+p}(re^{i\theta})-U_n(re^{i\theta})| \le \int\limits_{\frac{1}{r(n+p)}}^{\frac{1}{rn}} \left|\frac{u(\rho re^{i\theta})}{\rho}\right|d\rho + \left|\frac{1}{n+p}\int\limits_0^{\theta}\left( C_{n+p}-\int\limits_0^t\frac{\partial u}{\partial \boldsymbol a_{\tau}}(\frac{e^{i\tau}}{n+p}) d\tau\right)dt - \frac{1}{n}\int\limits_0^{\theta}\left( C_{n}-\int\limits_0^t\frac{\partial u}{\partial \boldsymbol a_{\tau}}(\frac{e^{i\tau}}{n}) d\tau\right)dt\right|$. To evaluate \\
		the first term, notice first that since $u \in C^0(\overline{\mathbb{U}}) \cap C^1(\mathbb{U})$ we have $\lim\limits_{\rho \searrow 0} \frac{u(\rho re^{i\theta})}{\rho} = u_{\boldsymbol x}(0)r\cos\theta + u_{\boldsymbol y}(0)r\sin\theta = r\frac{\partial u}{\partial \boldsymbol a_{\theta}}(0), ~\forall ~re^{i\theta} \in \mathbb{U}$. Hence $\exists M_1 > 0$ such that $\left|\frac{u(\rho re^{i\theta})}{\rho}\right| \le M_1, ~\forall ~\rho \in [0,1], ~\forall ~re^{i\theta} \in \mathbb{U}$. On the other hand, since $0 \not\in K$, there is a $\delta > 0$ (which may depend on $K$ as well) such that $d(0,K) = \delta$.The last two observations in turn imply $\int\limits_{\frac{1}{r(n+p)}}^{\frac{1}{rn}} \left|\frac{u(\rho re^{i\theta})}{\rho} \right|d\rho \le \frac{pM_1}{\delta n(n+p)}, ~ \forall re^{i\theta} \in K, ~\forall ~n\ge N, ~\forall ~p\in \mathbb{N}^*$. Next since $u \in C^1(\mathbb{U})$ it can be concluded that $\nabla u$ is bounded on, say, $|z| \le 2/3$ which in turn shows that one can choose $M_2 > 0$ for which $\left|\nabla u(\frac{e^{i\tau}}{n})\right| \le M_2, ~\forall ~\tau \in \mathbb{R}$. So $|\mathcal{C}_n| \le \frac{1}{2n\pi}\int\limits_0^{2\pi}\int\limits_0^t M_2~ d\tau dt = \frac{\pi M_2}{n}, ~\forall n \in \mathbb{N}^* \setminus \{1\}$. Finally, we obtain $\frac{pM_1}{\delta n(n+p)}+\theta\pi M_2\left(\frac{1}{n^2}-\frac{1}{(n+p)^2} \right) + \theta^2\frac{M_2}{n} \ge \int\limits_{\frac{1}{r(n+p)}}^{rn} \left|\frac{u\left(\rho re^{i\theta}\right)}{\rho}\right|d\rho +$ \\
		$\left|\frac{1}{n+p}\int\limits_0^{\theta}\left( \mathcal{C}_{n+p}-\int\limits_0^t\frac{\partial u}{\partial \boldsymbol a_{\tau}}(\frac{e^{i\tau}}{n+p}) d\tau\right)dt - \frac{1}{n}\int\limits_0^{\theta}\left( \mathcal{C}_{n}-\int\limits_0^t\frac{\partial u}{\partial \boldsymbol a_{\tau}}(\frac{e^{i\tau}}{n}) d\tau\right)dt\right|$, where the \\
		last term is greater than $|U_{n+p}(re^{i\theta})-U_n(re^{i\theta})|$. Since we can consider without loss of generality $\theta \in (-\pi,\pi]$, it follows that the sequence of harmonic functions $\{U_n\}_{n=2}^{\infty}$ is uniformly Cauchy on $K$, and hence on any compact subset of $\boldsymbol A_{0;1}$. Setting $\mathcal{U}(z) = \int\limits_0^1 \frac{u(\rho z)}{\rho}d\rho, ~z \in \mathbb{U}$, it is easy to see that $\lim\limits_{n \rightarrow \infty} U_n(z) = \int\limits_0^{1}\frac{u\left(\rho z\right)}{\rho}d\rho = \mathcal{U}(z)$ on $\boldsymbol A_{0;1}$. Hence $\mathcal{U}$ is harmonic in $\boldsymbol A_{0;1}$. In addition, using the \textit{Dominant Convergence} theorem, it follows that $\lim\limits_{z \rightarrow 0} \mathcal{U}(z) = 0$. This shows that $\mathcal{U}$ can be (uniquely) extended to a harmonic function in the whole unit disk, which shall also be denoted for brevity $\mathcal{U}$. It is not difficult to check that $\mathcal{U}$ can actually be extended by continuity to $\overline{\mathbb{U}}$. Finally $\frac{\partial \mathcal{U}}{\partial \boldsymbol \nu}(e^{i\theta}) = \lim\limits_{\epsilon \nearrow 0} \frac{\mathcal{U}(e^{i\theta} + \epsilon e^{i\theta}) - \mathcal{U}(e^{i\theta})}{\epsilon} = \lim\limits_{\epsilon \nearrow 0} \frac{1}{\epsilon} \int\limits_1^{1+\epsilon} \frac{u(\rho e^{i\theta})}{\rho}~ d\rho = u(e^{i\theta}) = g(e^{i\theta}) = f(e^{i\theta}), ~\forall ~\theta\in \mathbb{R}$. The continuous extension of $\nabla \mathcal{U}$ to $\overline{\mathbb{U}}$ follows by exactly the same arguments as those invoked in the proof of Theorem \ref{main theorem}; that is choosing any solution $V$ of the Neumann problem \eqref{Neumann problem} on $\mathbb{U}$ with boundary data $f$, and approximating the unit disk by an increasing sequence of disks of radii $r_n$, $r_n \nearrow 1$, the function $\mathcal{U}-V$ turns out to be constant on $\mathbb{U}$. In conclusion we have proved so far that $\mathcal{U}=U$, and so the solution of the Neumann problem \eqref{Neumann problem} on $\mathbb{U}$ with boundary data $f$ has the desired expresion provided by relation \eqref{U Corl2}. To complete the proof of the first part, it only remains to show that if $f \in C^{m,\alpha}(\partial \mathbb{U})$ for some positive integer $m \ge 2$ and some $\alpha \in (0,1]$, then $U$ and all its partial derivatives up to order $m+1$ are uniformly H{\"o}lder continuous with exponent $\alpha$ on $\mathbb{U}$. To this end we write
		\begin{equation}
		\overline{\mathbb{U}} = \overline{\boldsymbol A_{\frac{1}{3};1}} \cup \frac{1}{2}\mathbb{U},
		\end{equation}
		and denote by $U_1$ and $U_2$ the restrictions of $U$ to $\frac{1}{2}\mathbb{U}$ and $\overline{\boldsymbol A_{\frac{1}{3};1}}$, respectively. Further let $f_2: \partial \boldsymbol A_{\frac{1}{3};1} \rightarrow \mathbb{R}$, $f_2(re^{i\theta})= \begin{cases} f(e^{i\theta}), ~\text{if } r=1, \\ -U_{\boldsymbol x}(re^{i\theta})\cos\theta - U_{\boldsymbol y}(re^{i\theta})\sin\theta, ~\text{if } r=\frac{1}{3}, \end{cases}$ and observe that $f_2 \in C^{m,\alpha}(\partial \boldsymbol A_{\frac{1}{3};1})$; in addition $f_2$ satisfies the compatibility condition $\int\limits_{\partial \boldsymbol A_{\frac{1}{3};1}} f_2 d\sigma = 0$. According to Theorem \ref{main result cartesian} then, all the partial derivatives of $U_2$ up to order $m+1$ can be continuously extended to $\overline{\boldsymbol A_{\frac{1}{3};1}}$ and their extensions are uniformly H{\"o}lder continuous with exponent $\alpha$ there. Also, all the partial derivatives of $U_1$ are locally Lipschitz continuous on $\frac{1}{2}\mathbb{U}$ (due to the harmonicity of $U$), and consequently they are locally $\alpha$ H{\"o}lder continuous there. Appealing to the definitions of $U_1$ and $U_2$ it follows that $U$ together with all its partial derivatives up to order $m+1$ are locally $\alpha$ H{\"o}lder continuous on $\overline{\mathbb{U}}$, and using again a compactness argument concludes the first part of the proof.    
		
		For the second part denote $\hat{U}(r,\theta) = U(re^{i\theta}), ~\hat{u}(r,\theta) = u(re^{i\theta}), ~re^{i\theta}\in \overline{\mathbb{U}}$, where one can choose $U(0)=0$. Using the first part $\hat{U}(r,\theta)=\int\limits_0^1 \frac{\hat{u}(\rho r,\theta)}{\rho}~ d\rho = \int\limits_0^r \frac{\hat{u}(\rho,\theta)}{\rho}~ d\rho, ~re^{i\theta}\in \mathbb{U}$. Taking the derivative with respect to the first argument one obtains $\hat{U}_{\boldsymbol r}(r,\theta) = \frac{\hat{u}(r,\theta)}{r}$ or equivalently $\hat{u}(r,\theta) = r\hat{U}_{\boldsymbol r}(r,\theta)$, for any $r \in (0;1)$. Since $\hat{U}_{\boldsymbol r}(r,\theta) = \frac{\partial U}{\partial \boldsymbol a_{\theta}}(re^{i\theta}) \ \forall \theta \in \mathbb{R}$ the conclusion follows.  
	\end{proof}
	
	\subsection{General smooth, bounded, doubly-connected regions}
	
	Using the conformal invariance of harmonic functions and Theorem \ref{main result cartesian}, an important general result is obtained. Before stating it, some preparations are needed. First let $\boldsymbol D \subset \mathbb{C}$ be some smooth, doubly connected region whose boundary consists of two bounded Jordan curves which are the images of $\Gamma_i, ~i \in \{1,2\}$. It will be assumed that $\Gamma_1$ corresponds to the inner contour. Following the approach in \cite[Chapter 6]{Ahlfors} let $\omega_1$ be the harmonic measure of $\{\Gamma_1\}$ with respect to the region $\boldsymbol D$, and define $\alpha_1 = \int\limits_{\Gamma_1} \frac{\partial \omega_1}{\partial \boldsymbol n} ds$. Consequently define $\omega = \lambda_1 \omega_1$, where $\lambda_1 = \frac{2\pi}{\alpha_1}$, and letting $w=\xi+i\eta$ be the variable on $\boldsymbol D$ we also define $p = \frac{\partial \omega}{\partial \xi} - i\frac{\partial \omega}{\partial \eta}$, $q = \int p$ (where the integral is considered over any rectifiable curve having an extremity in $w_0$) and finally 
	\begin{equation*}
	G = e^q;
	\end{equation*}
	$w_0$ is an arbitrary point in $\boldsymbol D$ which is assumed to be fixed. Notice that $q$ is not single-valued, in general. However we will see in the lemma below that $G$ is actually a single-valued analytic function in $\boldsymbol D$. 
	\begin{lemma}\label{Auxiliary lemma for Theorems 5}
		Assume $\boldsymbol D \in C^{2;\alpha}$ for some $\alpha \in (0,1)$. Then $G$ defined above has the following properties.
		\begin{enumerate}
			\item[i.] $G$ is well defined on $\overline{\boldsymbol D}$;
			\item[ii.] $G(\overline{\boldsymbol D}) = \overline{\boldsymbol A_{1;e^{\lambda_1}}}$ and the mapping is one-to-one. In addition $G(\{\Gamma_2\}) = C_1$ and $G(\{\Gamma_1\}) = C_{e^{\lambda_1}}$, respectively;
			\item[iii.] $G$ is a conformal representation of $\boldsymbol D$ on $\boldsymbol A_{1;e^{\lambda_1}}$;
			\item[iv.] If $F= G^{-1}$ then the limit $\lim\limits_{z \rightarrow z^*,~ z \in \overline{\boldsymbol A_{1;e^{\lambda_1}}}} \frac{F(z^*) - F(z)}{z^* - z} =: F'(z^*)$ exists at all points $z^* \in \overline{\boldsymbol A_{1;e^{\lambda_1}}}$, and $F'$ can be extended by continuity to $\overline{\boldsymbol A_{1;e^{\lambda_1}}}$. 
			\item[v.] The limit $\lim\limits_{z \rightarrow z^*,~ z \in \overline{\boldsymbol A_{1;e^{\lambda_1}}}} \frac{F'(z^*) - F'(z)}{z^* - z} =: F''(z^*)$ exists at all points \\
			$z^* \in \overline{\boldsymbol A_{1;e^{\lambda_1}}}$, and $F''$ can be extended by continuity to $\overline{\boldsymbol A_{1;e^{\lambda_1}}}$.
		\end{enumerate} 
	\end{lemma}
	\begin{proof}
		For the proof of $i. - iii.$ see \cite[Chapter 6, Theorem 10]{Ahlfors}. For point $iv.$ notice first that the assumption $\partial \boldsymbol D \in C^{2;\alpha}$ implies (using \textit{Kellogg}'s theorem) that $\nabla \omega$ can be continuously extended to $\overline{\boldsymbol D}$. Consequently $G$ extends continuously to $\overline{\boldsymbol D}$. Using this aspect, the compactness of $\overline{\boldsymbol D}$, as well as points $i.$ and $ii.$ it is easy to notice that $F$ can be continuously extended to $\overline{\boldsymbol A_{1;e^{\lambda_1}}}$. The next step is to evaluate the limit $\lim\limits_{w \rightarrow w^*} \frac{G(w^*) - G(w)}{w^* - w}$ when $w^* \in \partial \boldsymbol D$ and $w \in \overline{\boldsymbol D}$. To this end it is helpful to notice that one may assume without loss of generality that the points $w_0, w, w^*$ belong to a single rectifiable curve as $w$ approaches $w^*$. With this observation in mind it is quite easy to see that $\lim\limits_{w \rightarrow w^*} \frac{G(w^*) - G(w)}{w^* - w} = p(w^*)G(w^*),~ \forall w^* \in \boldsymbol D$. Hence one can continuously extend the derivative of $G$ to $\overline{\boldsymbol D}$ by setting $G'(w) = p(w)G(w)$ if $w \in \partial \boldsymbol D$. Then $\lim\limits_{z \rightarrow z^*} \frac{F(z^*) - F(z)}{z^* - z} = \lim\limits_{w \rightarrow F(z^*)} \frac{1}{\frac{G(F(z^*)) - G(w)}{F(z^*) - w}} = \frac{1}{G'(F(z^*))}, ~z^* \in \partial \boldsymbol A_{1;e^{\lambda_1}}$. In order to conclude, it only remains to prove that $G'_{|\partial \boldsymbol D}$ does not vanish at any point. Suppose by contradiction that there is a point $w^* \in \partial \boldsymbol D$ such that $G'(w^*) = 0$, and one may assume without loss of generality that this point belongs to the exterior contour (for the case when $w^*$ belongs to the inner contour, the reasoning is similar, with the only difference that the conformal mapping $T$ defined right below will be considered from the exterior of $\boldsymbol \Omega_i$ to the interior of the unit disk). Define $\boldsymbol \Omega_i$ to be the region bounded by the image of $\Gamma_1$ and let $\boldsymbol \Omega$ be the region bounded by the image of $\Gamma_2$. Since $\boldsymbol \Omega \neq \mathbb{C}$ is a simply connected region, by \textit{Riemann Mapping} theorem there exists a (unique) conformal transformation $T$ of $\mathbb{U}$ onto $\boldsymbol \Omega$ such that $T(0) = w_i, \ T'(0) > 0$ for some $w_i \in \boldsymbol \Omega_i$ (see Figure \ref{fig: Lemma3_iv}). Letting $\boldsymbol V_i = T^{-1}(\boldsymbol \Omega_i)$ it follows by the \textit{Reflection Principle} that the map $J: \mathbb{U}\setminus \overline{\boldsymbol V_i} \rightarrow \boldsymbol A_{1,e^{\lambda_1}}, \ J = G \circ T$ can be analytically extended to $(\mathbb{U}\setminus \overline{\boldsymbol V_i}) \cup B(\lambda,2\epsilon)$ at any point $\lambda \in \partial \mathbb{U}$ (where $\epsilon$ may depend on $\lambda$) and in addition when restricting $J$ to $B(\lambda,2\epsilon)$, the disk of radius $2\epsilon$ centered in $\lambda$, the only points in $B(\lambda,2\epsilon)$ which are mapped on $\partial \boldsymbol A_{1,e^{\lambda_1}}$ are those which lie on $\partial \mathbb{U}$ and the correspondence is one-to-one. But defining $\lambda^* = T^{-1}(w^*)$, the assumption $G'(w^*)=0$ implies $J'(\lambda^*)=0$ (this is true since $T'$ can be continuously extended to $\partial \mathbb{U}$ according to \cite[Theorem 3.6]{Pommerenke} and is thus bounded on $\overline{\mathbb{U}}$). Restricting $J$ to $B(\lambda^*,2\epsilon)$ an application of the \textit{Argument Principle} shows that for any $z \in J(B(\lambda^*,\epsilon))$ the equation $z-J(\lambda)=0$ has either at least two solutions in $B(\lambda^*,\epsilon)$, or $\lambda$ is a solution of order at least two in which case $J'(\lambda)=0$. But if $z$ also lies on $\partial \boldsymbol A_{1,e^{\lambda_1}}$ then the earlier discussion shows that there is a unique $\lambda \in B(\lambda^*,\epsilon)$ for which $J(\lambda) = z$ and furthermore $\lambda$ also belongs to $\partial \mathbb{U}$. Consequently it must be the case that $J'(\lambda)=0$, and since this is true for any $\lambda \in \partial \mathbb{U} \cap B(\lambda^*,\epsilon)$ it follows that $J$ must be a constant function. This is obviously a contradiction and so $G(w^*) \neq 0$. Finally putting $F'(z) = \frac{1}{G'(F(z))}$ whenever $z \in \boldsymbol A_{1,e^{\lambda_1}}$ and noticing that $G'(w) = p(w)G(w) ~\forall w \in \boldsymbol D$ the proof of $iii.$ is complete.
		
		\begin{figure}	
			\includegraphics[width=\linewidth]{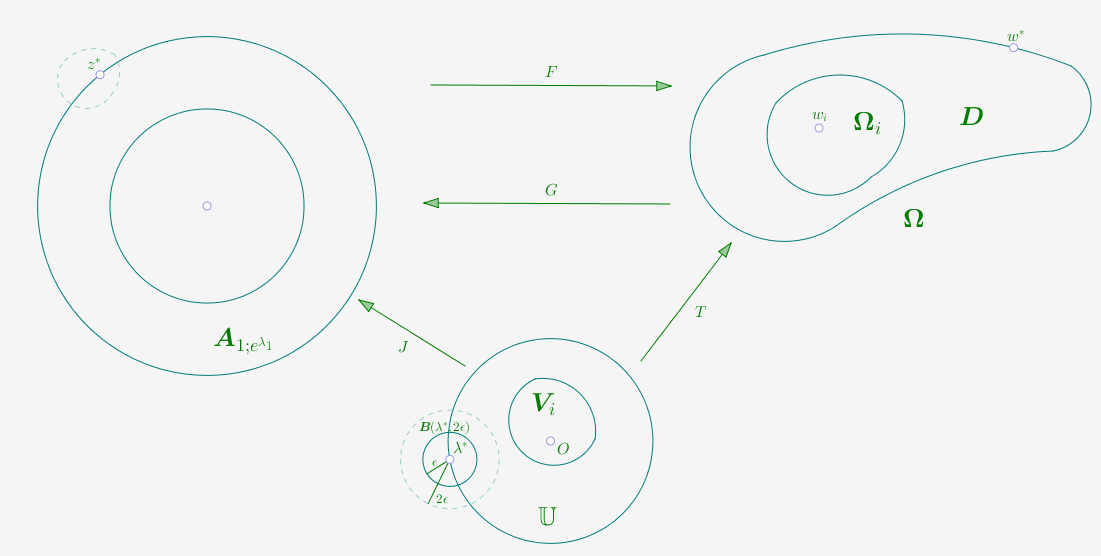}
			\caption{The conformal maps of point $iv.$ of Lemma \ref{Auxiliary lemma for Theorems 5}.}
			\label{fig: Lemma3_iv}	
		\end{figure}
		
		For the last point fix any arbitrary $z^* \in \partial \boldsymbol A_{1;e^{\lambda_1}}$ and denoting $w = F(z)$, $w^*=F(z^*)$, $\lim\limits_{z \rightarrow z^*,~ z \in \overline{\boldsymbol A_{1;e^{\lambda_1}}}} \frac{F'(z) - F'(z^*)}{z - z^*} = \lim\limits_{w \rightarrow w^*,~ w \in \overline{\boldsymbol D}} \left[ \frac{1}{G'(w)} \frac{G'(w^*) - G'(w)}{G(w)-G(w^*)} \right]\frac{1}{G'(w^*)}$ \\
		$= -\frac{1}{G'(w^*)}\lim\limits_{w \rightarrow w^*,~ w \in \overline{\boldsymbol D}} \left[ \frac{1}{G'(w)} \frac{G'(w^*) - G'(w)}{w^* - w}\frac{1}{\frac{G(w)-G(w^*)}{w - w^*}} \right]$. Also we notice that \\ $\lim\limits_{w \rightarrow w^*,~ w \in \overline{\boldsymbol D}} \frac{G'(w^*) - G'(w)}{w^* - w} = \lim\limits_{w \rightarrow w^*,~ w \in \overline{\boldsymbol D}} \left[ p(w^*)\frac{G(w^*) - G(w)}{w^* - w} + G(w)\frac{p(w^*) - p(w)}{w^* - w} \right]$. \\ 
		Let $B_{w^*}$ be a simply connected, relatively open (with respect to $\overline{\boldsymbol D}$) neighborhood of $w^*$, and let $w_0^*$ be some point in $B_{w^*}$ which will be chosen later on. It is easy to see that the function $\int\limits_{w_0^*}^w p'(\lambda) d\lambda + p(w_0^*)$ is well-defined and coincides with $p$, on $B_{w^*} \setminus \partial \boldsymbol  D$. In addition, since $p'$ extends continuously to $\partial \overline{\boldsymbol D}$ (use the \textit{Cauchy-Riemann} equations as well as \textit{Kellogg}'s theorem for $\omega_1$), it follows that $\int\limits_{w_0^*}^w p'(\lambda) d\lambda + p(w_0^*)$ can actually be extended by continuity to $B_{w^*}$. It is claimed that one can always choose the point $w_0^* \neq w$ in such way that the line segments with edges $(w_0^*,w),~ (w_0^*,w^*)$ are in $B_{w^*}$ and furthermore the ratio $\left|\frac{w_0^* - w}{w^*-w}\right|$ stays bounded as $w \rightarrow w^*$. Indeed since $\boldsymbol D \in C^{m,\alpha}$, $m \ge 2$, letting $t^* = \Re(w^*)$ there is a real-valued function $h$ defined on some open interval $\textbf{I}$ containing $t^*$ such that $h \in C^{m,\alpha}(\textbf{I})$ and, eventually after performing an appropriate rotation, the boundary of $\boldsymbol D$ around $w^*$ coincides with the graphic of $h$. Assume without loss of generality that $B_{w^*}\cap \boldsymbol D$ is below the graphic of $h$. If $h''(t^*) < 0$ then $h$ is locally (strictly) concave around $t^*$  and so one can choose $w_0^*$ to be the middle of the segment with edges $w$ and $w^*$. If $h''(t^*) > 0$ then $h$ is locally strictly convex around $t^*$ and one can choose $w_0^*$ to be the projection of $w$ on the tangent to the graphic of $h$ in $t^*$. Last but not least if $h''(t^*) = 0$ then $h$ is convex at the left-hand side of $t^*$ and concave at the right-hand side of $t^*$, or vice-versa, and we choose $w_0^*$ as above according to whether it lies on the convex or the concave side of the graphic (see Figure \ref{fig: Lemma3_v}), with the amendment that whenever the line segment with edges $(w,w^*)$ is included in $\overline{\boldsymbol D}$ the point $w_0^*$ can be chosen the middle of the segment. Having proved the claim we can go back and thus obtain $\lim\limits_{w \rightarrow w^*, ~w \in \overline{\boldsymbol D}} \frac{p(w^*) - p(w)}{w^* - w} = \lim\limits_{w \rightarrow w^*, ~w \in \overline{\boldsymbol D}} \frac{\left(p(w^*) - p(w_0^*)\right) + \left(p(w_0^*) - p(w) \right)}{w^* - w} = $ \\ $\lim\limits_{w \rightarrow w^*, ~w \in \overline{\boldsymbol D}} \left[ \frac{p(w^*) - p(w_0^*)}{w^* - w_0^*} + \left(\frac{p(w_0^*) - p(w)}{w_0^* - w} - \frac{p(w^*) - p(w_0^*)}{w^* - w_0^*}\right) \frac{w_0^* - w}{w^*-w}  \right]$, where in the last expression, using the integral representation for $p$ as well as the \textit{Dominant Convergence} theorem, the first term approaches $\frac{\partial^2 \omega}{\partial \xi^2}(w^*) - i\frac{\partial^2 \omega}{\partial\xi \partial\eta}(w^*)$ and the second term approaches $0$. In conclusion $\lim\limits_{w \rightarrow w^*, ~w \in \overline{\boldsymbol D}} \frac{p(w^*) - p(w)}{w^* - w} = \frac{\partial^2 \omega}{\partial \xi^2}(w^*) - i\frac{\partial^2 \omega}{\partial\xi \partial\eta}(w^*) =: p'(w^*)$. Returning, we observe that
		$\lim\limits_{z \rightarrow z^*,~ z \in \overline{\boldsymbol A_{1;e^{\lambda_1}}}} \frac{F'(z) - F'(z^*)}{z - z^*} =$ \\
		$-\frac{1}{G'(w^*)}\lim\limits_{w \rightarrow w^*,~ w \in \overline{D}} \left[ \frac{1}{G'(w)} \frac{G'(w^*) - G'(w)}{w^* - w}\frac{1}{\frac{G(w)-G(w^*)}{w - w^*}} \right] = \left[p^2(F(z^*)) + p'(F(z^*))\right]\cdot$ \\
		$z^*\left(-F'(z^*)\right)^3 =: F''(z^*)$. This expression obviously holds for all points $z^* \in \boldsymbol A_{1;e^{\lambda_1}}$ as well, and it is thus seen that $F''$ can be continuously extended to $\overline{\boldsymbol A_{1;e^{\lambda_1}}}$. This ends the proof of the lemma. 
	\end{proof} 
	
	\begin{figure}[h!]	
		\includegraphics[width=\linewidth]{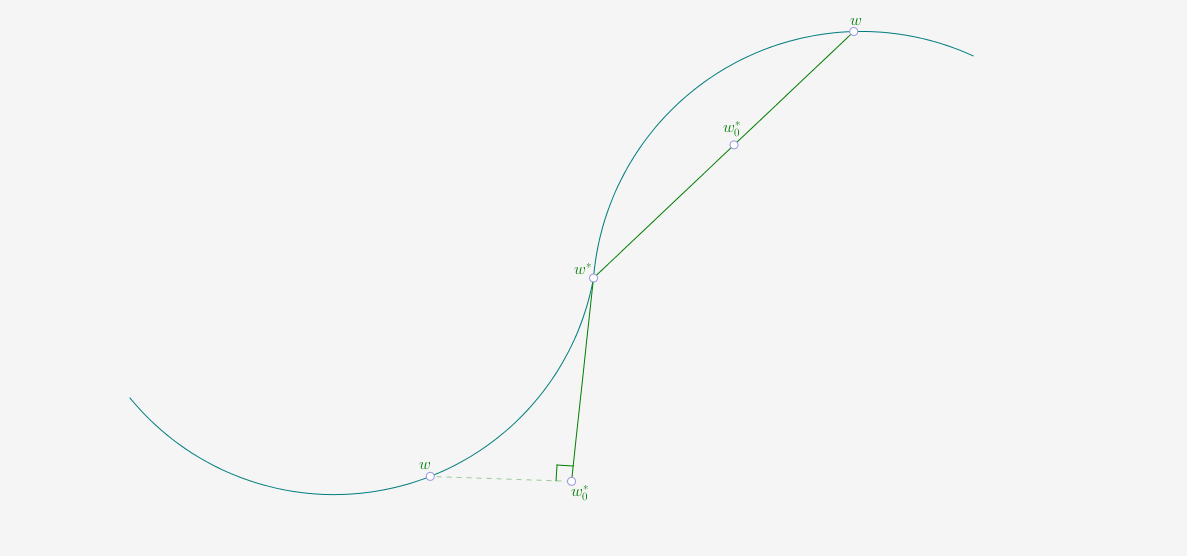}
		\caption{Smooth boundary property.}
		\label{fig: Lemma3_v}	
	\end{figure}
	
	\begin{remark}
		The lemma above can be generalized by exactly the same arguments to the case when $\boldsymbol D \in C^{m,\alpha}$ for some positive integer $m \ge 2$ and some $\alpha \in (0,1)$, in which case the higher derivatives of $F$ up to order $m$ can be defined for the boundary points of $\boldsymbol A_{1,e^{\lambda_1}}$ as well, and in addition they extend continuously to $\overline{\boldsymbol A_{1,e^{\lambda_1}}}$.
	\end{remark}
	
	\begin{theorem}\label{smooth exension proposition doubly-connected}
		Let $\boldsymbol D \in C^{m+1,\alpha}$ for some positive integer $m \ge 2$ and some $\alpha \in (0,1)$, and in addition assume $\Phi \in C^{m,\alpha}(\partial \boldsymbol D)$ satisfies the compatibility condition $\int\limits_{\partial \boldsymbol D} \Phi ~d\sigma = 0$. If $U$ is a solution of the Neumann problem \eqref{Neumann problem} with boundary data $\Phi$ then $U$ and all its partial derivatives up to order $m+1$ are uniformly H{\"o}lder continuous with exponent $\alpha$ on $\boldsymbol D$.
	\end{theorem} 
	
	Before proceeding with the proof of the theorem notice that due to the condition $\boldsymbol D \in C^{m+1,\alpha}$, $m \ge 2$, \textit{Kellogg}'s theorem guarantees that $\nabla \omega$ can be continuously extended to $\overline{\boldsymbol D}$, in which case its continuous extension is also denoted by $\nabla \omega$. Furthermore, as seen in the proof of Lemma \ref{Auxiliary lemma for Theorems 5}, $G' \neq 0$ on $\overline{\boldsymbol D}$ and since $G'=pG$, it follows that $\nabla \omega \neq 0$ on $\overline{\boldsymbol D}$. 
	
	\begin{proof}
		Set $r_1 = 1, ~r_2 = e^{\lambda_1}$ and let $F: \boldsymbol A_{r_{1};r_{2}}\rightarrow \boldsymbol D$ be the conformal map given in Lemma \ref{Auxiliary lemma for Theorems 5}. Without loss of generality assume $U(F(\sqrt{r_1r_2}))=0$. Since $\boldsymbol A_{r_1;r_2} \in C^{\infty}$ it can be also assumed, without restricting the generality, that $\Gamma_1, ~\Gamma_2 \in C^{2,\alpha}$. Defining $f: \partial \boldsymbol A_{r_1;r_2} \rightarrow \mathbb{R}, ~ f = \left(\Phi \circ F\right) |F'|$, we claim that $f \in C^{m,\alpha}(\partial \boldsymbol A_{r_1;r_2})$. In an attempt to keep the proof as clear as possible the author will prove the claim only for the case $m=2$; the case of a general $m \ge 2$ follows by induction in the same spirit as for this simpler case. To begin with notice that point $iv.$ of Lemma \ref{Auxiliary lemma for Theorems 5} shows that $F$ is locally Lipschitz continuous on $\boldsymbol A_{r_1,r_2}$ (use an integral representation of $F$ in terms of $F'$ in some convex neighborhood of any point of $\overline{\boldsymbol A_{r_1,r_2}}$). Secondly, point $v.$ of Lemma \ref{Auxiliary lemma for Theorems 5} reveals that $F'$ is locally Lipschitz continuous on $\overline{\boldsymbol A_{r_1,r_2}}$ (use again an integral representation of $F'$ in terms of $F''$ in some convex neighborhood of any point of $\overline{\boldsymbol A_{r_1,r_2}}$) and hence locally $\alpha$ H{\"o}lder continuous there. Using the compactness of $\overline{\boldsymbol A_{r_1,r_2}}$ we can argue that $F'$ is in fact uniformly H{\"o}lder continuous with exponent $\alpha$ on $\boldsymbol A_{r_1,r_2}$. For $F''$ choose any two points $z_1, ~z_2 \in \boldsymbol A_{r_1,r_2}$ and notice that $|F''(z_2)-F''(z_1)| = |z_2\left(p'(F(z_2))+ p^2(F(z_2))\right)\left(F'(z_2)\right)^3 - z_1\left(p'(F(z_1))+ p^2(F(z_1))\right)\left(F'(z_1)\right)^3| \le \left(\|p'\circ F\| + \|p^2 \circ F\| \right)\|(F')^3\||z_2 - z_1|+r_2|F'(z_2) - F'(z_1)|\left( 4\|(F')^2\| + \|p'\circ F\| + \|p^2 \circ F\|\right) + r_2(|p'(F(z_2))- p'(F(z_1))| + 2|p(F(z_2)) - p(F(z_1))|\|p \circ F\|)\|(F')^3\|$. Recall that $p = \omega_{\boldsymbol \xi} - i\omega_{\boldsymbol \eta}$ and by \textit{Kellogg}'s theorem combined with the compactness of $\overline{\boldsymbol D}$ the $m=2$ order partial derivatives of $\omega$ are uniformly H{\"o}lder continuous with exponent $\alpha$ on $\boldsymbol D$. Thus using the last inequality and the locally Lipschitz continuity of $F$ on $\overline{\boldsymbol D}$ it follows that $F''$ is locally $\alpha$ Holder continuous on $\overline{\boldsymbol A_{r_1,r_2}}$ and hence uniformly H{\"o}lder continuous with exponent $\alpha$ there. Finally $F'''(z) = -(F'(z))^2 [(p'(F(z))+p^2(F(z))F'(z) + z(p''(F(z))+2p(F(z)))(F'(z))^2 - 3z(p'(F(z))+p^2(F(z)))F''(z)]$, and using the locally Lipschitz property for $F$, the uniform H{\"o}lder continuity with exponent $\alpha$ of $F'$ and $F''$ as well as the uniform H{\"o}lder continuity with exponent $\alpha$ of the partial derivatives of $\omega$ up to order $m=2$, it can be deduced after several applications of the triangle inequality that $F'''$ is locally $\alpha$ H{\"o}lder continuous on $\overline{\boldsymbol A_{r_1,r_2}}$ and so uniform H{\"o}lder continuous with exponent $\alpha$ there. To complete the proof of the claim notice that since $0 \not\in \overline{\boldsymbol A_{r_1,r_2}}$ it follows that $|\cdot| \in C^{\infty}(\boldsymbol A_{r_1,r_2})$ and so we have $\frac{d^2}{dt^2}f(\Gamma_i(t)) = \frac{d^2}{dt^2}\Phi(F(\Gamma_i(t))) \cdot |F'(\Gamma_i(t))| + 2\frac{d}{dt}\Phi(F(\Gamma_i(t))) \cdot \frac{d}{dt}|F'(\Gamma_i(t))| + \Phi(F(\Gamma_i(t))) \cdot \frac{d^2}{dt^2}|F'(\Gamma_i(t))|$. But $\frac{d^2}{dt^2}F(\Gamma_i(t)) = F''(\Gamma_i(t))\left(\Gamma_i'(t)\right)^2 + F'(\Gamma_i(t))\Gamma_i''(t)$ which is readily seen to be locally $\alpha$ H{\"o}lder continuous. Also $\frac{d}{dt}|F'(\Gamma_i(t))| = |F'(\Gamma_i(t))| \Re\left(F''(\Gamma_i(t))\Gamma_i'(t)G(\Gamma_i(t))\right)$ and thus $\frac{d^2}{dt^2}|F'(\Gamma_i(t))| = \Re\left(\frac{(F'''(\Gamma_i(t))(\Gamma_i'(t))^2 + F''(\Gamma_i(t))\Gamma_i''(t))F'(\Gamma_i(t)) - (F''(\Gamma_i(t))\Gamma_i'(t))^2}{(F'(\Gamma_i(t)))^2} \right)\cdot$ \\ $|F'(\Gamma_i(t))|+ |F'(\Gamma_i(t))|\Re^2\left(G'(F(\Gamma_i(t)))F''(\Gamma_i(t))\Gamma_i'(t)\right)$. When putting everything together, the proof of the claim follows easily. Having proved that  $f$ belongs to $C^{m,\alpha}(\partial \boldsymbol A_{r_1,r_2})$ let $V$ be the solution of the Neumann problem \eqref{Neumann problem} on $\boldsymbol A_{r_1;r_2}$ with boundary data $f$, satisfying $V(\sqrt{r_1r_2})=0$ (using direct computations together with the assumption $\int\limits_{\partial \boldsymbol D} \Phi ~d\sigma = 0$ it is not difficult to see that $f$ satisfies the compatibility condition $\int\limits_{\partial \boldsymbol A_{r_1;r_2}} f ~ds = 0$). Then, according to Theorem \ref{main result cartesian}, the gradient of $V$ can be continuously extended to the closure of $\boldsymbol A_{r_1,r_2}$ (as before, its continuous extension will also be denoted $\nabla V$). Now set $W = V \circ G$ which shows that $W$ is harmonic in $\boldsymbol D$ and furthermore $W(F(\sqrt{r_1r_2})) = 0$. Also taking the partial derivatives of $W$ with respect to $\xi$ and $\eta$ it follows that for any $w \in \boldsymbol D$
		\begin{align*}
		\frac{\partial W}{\partial \xi}(w) &= \Re \left(\frac{\overline{\nabla V(G(w))}}{F'(G(w))}\right), \\
		\frac{\partial W}{\partial \eta}(w) &= -\Im \left(\frac{\overline{\nabla V(G(w))}}{F'(G(w))}\right).
		\end{align*}
		Hence $\nabla W = \frac{\nabla V \circ G}{\overline{F' \circ G}}$ on $\boldsymbol D$ which proves, together with point $iv.$ of Lemma \ref{Auxiliary lemma for Theorems 5}, that $\nabla W$ extends continuously to $\overline{\boldsymbol D}$. Consequently it follows by the \textit{Mean Value} theorem that $\frac{\partial W}{\partial \boldsymbol \nu}(w) = \langle \nabla W(w); \boldsymbol \nu(w) \rangle,~ w \in \partial \boldsymbol D$, where $\boldsymbol \nu$ stands for the (outward) normal derivative at $\partial \boldsymbol D$ and it is given by
		\begin{equation*}
		\boldsymbol \nu(F(z)) = \begin{cases}\frac{z F'(z)}{r_2|F'(z)|} = \frac{\nabla \omega(F(z))}{|\nabla \omega(F(z))|}~ &\text{if } |z|=r_2, \\ -\frac{z F'(z)}{r_1|F'(z)|}  = -\frac{\nabla \omega(F(z))}{|\nabla \omega(F(z))|}~ &\text{if } |z|=r_1. \end{cases}
		\end{equation*}
		Returning $\frac{\partial W}{\partial \boldsymbol \nu}(w) = \langle \nabla W(w); \boldsymbol \nu(w) \rangle = \frac{f \circ G}{|F'(G(w))|} = \Phi,~ \forall ~w \in \partial \boldsymbol D$. To sum up $W$ is harmonic in $\boldsymbol D$, has boundary data $\Phi$, satisfies $W(F(\sqrt{r_1r_2})) = 0$, and in addition $\nabla W \in C^1(\overline{\boldsymbol D})$. So $W-U = constant$. But then fixing any arbitrary points $w_1, ~w_2$ in $\boldsymbol D$ and letting $z_1 = G(w_1), ~z_2 = G(w_2)$ one obtains $|U(w_2) - U(w_1)| = |W(w_2) - W(w_1)| = |V(z_2) - V(z_1)| \le C_0|z_2-z_1|^{\alpha} = C_0|G(w_2) - G(w_1)|^{\alpha}$ for some positive constant $C_0$. Using the fact that $G'$ can be continuously extended to $\overline{\boldsymbol D}$, one concludes that $G$ is locally Lipschitz continuous on $\overline{\boldsymbol D}$ and hence $|U(w_2) - U(w_1)| \le C_0|G(w_2) - G(w_1)|^{\alpha} \lesssim |w_2 - w_1|^{\alpha}$, thus proving that $U$ is locally $\alpha$ H{\"o}lder continuous on $\overline{\boldsymbol D}$ and hence uniformly H{\"o}lder continuous with exponent $\alpha$ by means of a compactness argument. Also $U_{\boldsymbol \xi}(w) = W_{\boldsymbol \xi}(w) = \Re \left(\overline{\nabla V}(G(w))G'(w)\right)$ and so using Theorem \ref{main result cartesian} $|U_{\boldsymbol \xi}(w_2) - U_{\boldsymbol \xi}(w_1)| = |W_{\boldsymbol \xi}(w_2) - W_{\boldsymbol \xi}(w_1)| = |\Re \left(\overline{\nabla V}(G(w_2))G'(w_2)\right) - \Re\left(\overline{\nabla V}(G(w_1))G'(w_1)\right)| \le |\overline{\nabla V}(G(w_2))G'(w_2)- \overline{\nabla V}(G(w_1))G'(w_1)| \le |G(w_2)-G(w_1))|^{\alpha}C_1\|G\|\|p\| + H_1\|G\|\|\nabla V \circ G\| |w_2-w_1|^{\alpha} + |w_2-w_1|L_0\|\nabla V\circ G\|\|p\|$, where $C_1, ~H_1, ~L_0$ are positive constants. But if $w_1$ and $w_2$ are close enough then one can replace $|w_2-w_1|$ in the last term above by $|w_2-w_1|^{\alpha}$ and thus conclude that $U_{\boldsymbol \xi}$ is locally $\alpha$ H{\"o}lder continuous on $\overline{\boldsymbol D}$ and so uniformly H{\"o}lder continuous with exponent $\alpha$ there since $\boldsymbol D$ is bounded. Similarly $U_{\boldsymbol \eta}(w) = W_{\boldsymbol \eta}(w) = -\Im\left(\overline{\nabla V}(G(w))G'(w)\right)$ and thus $|U_{\boldsymbol \eta}(w_2) - U_{\boldsymbol \eta}(w_1)| = |W_{\boldsymbol \eta}(w_2) - W_{\boldsymbol \eta}(w_1)| = |\Im \left(\overline{\nabla V}(G(w_2))G'(w_2)\right) - \Im\left(\overline{\nabla V}(G(w_1))G'(w_1)\right)| \le |G'(w_2)\overline{\nabla V}(G(w_2)) - G'(w_1)$ \\
		$\overline{\nabla V}(G(w_1))| \le C_1\|p\|\|G\||G(w_2)-G(w_1))|^{\alpha} + H_1\|\nabla V \circ G\|\|G\| |w_2-w_1|^{\alpha} +L_0|w_2-w_1|\|\nabla V \circ G\|\|p\|$ and so $U_{\boldsymbol \eta}$ also turns out to be uniformly H{\"o}lder continuous with exponent $\alpha$ on $\boldsymbol D$. For the second order partial derivatives notice that $V$ harmonic implies that $\overline{\nabla V}$ is an analytic function, when considering $\nabla V$ as a complex number, and so $U_{\boldsymbol \xi \boldsymbol \xi}(w) = \Re\left(\overline{\nabla V}(G(w))G''(w) + \overline{\nabla V}'(G(w))(G'(w))^2\right)$. But then it is obtained that 
		$|U_{\boldsymbol \xi \boldsymbol \xi}(w_2) - U_{\boldsymbol \xi \boldsymbol \xi}(w_1)| = |W_{\boldsymbol \xi \boldsymbol \xi}(w_2) - W_{\boldsymbol \xi \boldsymbol \xi}(w_1)| \le |\Re \left(\overline{\nabla V}(G(w_2))G''(w_2) \right) - \Re (G''(w_1)\overline{\nabla V}(G(w_1)) )| + |\Re \left(\overline{\nabla V}'(G(w_2)) (G'(w_2))^2 \right) -$ \\ 
		$\Re \left(\overline{\nabla V}'(G(w_1))(G'(w_1))^2 \right)| \le|G''(w_2)\overline{\nabla V}(G(w_2))- G''(w_1)\overline{\nabla V}(G(w_1))|+$ \\
		$|\overline{\nabla V}'(G(w_2)(G'(w_2))^2 - \overline{\nabla V}'(G(w_1))\cdot(G'(w_1))^2| \le \|G\|\left(\|p'\|+\|p^2\|\right)\{|V_{\boldsymbol x}(G(w_2))$ \\
		$- V_{\boldsymbol x}(G(w_1))| + |V_{\boldsymbol y}(G(w_2)) - V_{\boldsymbol y}(G(w_1))|\}  + \|\nabla V \circ G\|\|G\|\{|p'(w_2) - p'(w_1)| + |p^2(w_2) - p^2(w_1)|\} + \|\nabla V \circ G\|\left(\|p'\|+\|p^2\|\right)|G(w_2)-G(w_1)| + \{|V_{\boldsymbol x \boldsymbol x}(G(w_2))-V_{\boldsymbol x \boldsymbol x}(G(w_1))| + |V_{\boldsymbol x \boldsymbol y}(G(w_2))-V_{\boldsymbol x \boldsymbol y}(G(w_1))|\}\|G'\|^2 + 2\|G'\|\|\overline{\nabla V}'\||G'(w_2) - G'(w_1)|$. Using the uniform Lipschitz continuity of $G$ and $G'$, the uniform H{\"o}lder continuity with exponent $\alpha$ of $V_{\boldsymbol x}$, $V_{\boldsymbol y}$, $V_{\boldsymbol x \boldsymbol x}$, $V_{\boldsymbol x \boldsymbol y}$, $p$, $p'$, as well as the compactness of $\overline{\boldsymbol D}$ it follows that $U_{\boldsymbol \xi \boldsymbol \xi}$ is uniformly H{\"o}lder continuous with exponent $\alpha$. Proceeding further we notice that $U_{\boldsymbol \xi \boldsymbol \eta}(w) = -\Im\left(\overline{\nabla V}(G(w))G''(w) + \overline{\nabla V}'(G(w))(G'(w))^2\right)$,  and also $U_{\boldsymbol \eta \boldsymbol \eta}(w) = -\Re \left(\overline{\nabla V}(G(w))G''(w) + \overline{\nabla V}'(G(w))(G'(w))^2 \right)$; so one gets the same upper-bounds for $|U_{\boldsymbol \eta \boldsymbol \eta}(w_2) - U_{\boldsymbol \eta \boldsymbol \eta}(w_1)|$ and $|U_{\boldsymbol \xi \boldsymbol \eta}(w_2) - U_{\boldsymbol \xi \boldsymbol \eta}(w_1)|$, respectively. In conclusion $U_{\boldsymbol \xi \boldsymbol \eta}$ and $U_{\boldsymbol \eta \boldsymbol \eta}$ are uniformly H{\"o}lder continuous with exponent $\alpha$ on $\boldsymbol D$ as well. For the third order partial derivatives of $U$ notice that applying the \textit{Cauchy-Riemann} equations for the expressions of $U_{\boldsymbol \xi \boldsymbol \xi}$ and $U_{\boldsymbol \eta \boldsymbol \eta}$  we obtain that $U_{\boldsymbol \xi \boldsymbol \xi \boldsymbol \xi}(w) = \Re(\mathcal{F}(w))$, $U_{\boldsymbol \xi \boldsymbol \xi \boldsymbol \eta}(w) = U_{\boldsymbol \xi \boldsymbol \eta \boldsymbol \xi}(w) = U_{\boldsymbol \eta \boldsymbol \xi \boldsymbol \xi}(w) = -\Im(\mathcal{F}(w))$, $U_{\boldsymbol \xi \boldsymbol \eta \boldsymbol \eta}(w) = U_{\boldsymbol \eta \boldsymbol \xi \boldsymbol \eta}(w) = U_{\boldsymbol \eta \boldsymbol \eta \boldsymbol \xi}(w) = -\Re(\mathcal{F}(w))$, and finally $U_{\boldsymbol \eta \boldsymbol \eta \boldsymbol \eta}(w) = \Im(\mathcal{F}(w))$, where $\mathcal{F}(w) = \overline{\nabla V}'(G(w))G'(w)G''(w) + \overline{\nabla V}(G(w))G'''(w) + \overline{\nabla V}''(G(w))(G'(w))^3$ \\
		$+ 2\overline{\nabla V}'(G(w))G'(w)G''(w)$ is an analytic function. On the other hand $G'(w)=p(w)G(w)$, $G''(w)=(p'(w)+p^2(w))G(w)$, $G'''(w)=(p''(w)+3p(w)p'(w)+p^3(w))G(w)$, and $p'(w)=\omega_{\boldsymbol \xi \boldsymbol \xi}(w) - i\omega_{\boldsymbol \xi \boldsymbol \eta}(w)$, $p''(w)=\omega_{\boldsymbol \xi \boldsymbol \xi \boldsymbol \xi}(w) - i\omega_{\boldsymbol \eta \boldsymbol \xi \boldsymbol \xi}(w)$, where the latter two are uniformly H{\"o}lder continuous with exponent $\alpha$ on $\boldsymbol D$ due to \textit{Kellogg's} theorem. Using these relations and proceeding similarly as we did for the second order partial derivatives, it follows that the third order partial derivatives of $U$ are locally $\alpha$ H{\"o}lder continuous on $\overline{D}$, and hence uniformly H{\"o}lder continuous with exponent $\alpha$ due to the compactness of $\overline{D}$. This concludes the proof for the case when $m=2$. For an arbitrary positive integer $m \ge 2$ one can proceed by induction using the same arguments as for the case $m=2$. The proof is now completed. 
	\end{proof}
	
	The next theorem shows the equivalence of the solutions of Dirichlet and Neumann problems for the Laplace operator in the case of bounded, planar, doubly-connected regions.
	
	\begin{theorem}\label{thm in the case of doubly connected regions}
		Let $F: \boldsymbol A_{1;r_2}\rightarrow \boldsymbol D$ be the conformal map given in Lemma \ref{Auxiliary lemma for Theorems 5}, where $r_2=e^{\lambda_1}$, and assume $\boldsymbol D \in C^{2,\alpha}$ for some $\alpha \in (0,1)$. If $\Phi \in C^0(\partial \boldsymbol D)$ satisfies the compatibility condition $\int\limits_{\partial \boldsymbol D} \Phi ~d\sigma = 0$ and if $U$ is the solution of the Neumann problem \eqref{Neumann problem} on $\boldsymbol D$ with boundary data $\Phi$, satisfying $U(F(\sqrt{r_2}))=0$, then for any point $w \in \overline{\boldsymbol D}$
		\begin{align}\label{U doubly connected}
		U&(w) = \int\limits_{\frac{\sqrt{r_2}} {|G(w)|} }^1 \frac{u(F(\rho G(w)))}{\rho}d\rho \\ &+ \sqrt{r_2}\int\limits_0^{arg(G(w))}\left[\tilde{\mathcal{C}} -\Re\left(\int\limits_0^t \overline{\nabla u}(F(\sqrt{r_2}e^{i\tau}))F'(\sqrt{r_2}e^{i\tau})e^{i\tau}d\tau \right)\right]dt, \nonumber
		\end{align}
		where $u$ is the solution of the Dirichlet problem \eqref{Dirichlet problem} on $\boldsymbol D$ with boundary values 
		\begin{equation}\label{varphi}
		\varphi(w) = \begin{cases}\frac{\Phi(w)}{|\nabla \omega(w)|} ~&\text{if } |G(w)|=r_2, \\ -\frac{\Phi(w)}{|\nabla \omega(w)|} ~&\text{if } |G(w)|=1, \end{cases}
		\end{equation}
		and where the constant $\tilde{\mathcal{C}}$ is given by
		\begin{equation*}
		\frac{\sqrt{r_2}}{2\pi}\int\limits_0^{2\pi}\Re\left(\int\limits_0^t \overline{\nabla u}(F(\sqrt{r_2}e^{i\tau}))F'(\sqrt{r_2}e^{i\tau})e^{i\tau}d\tau \right)dt.
		\end{equation*}
		Conversely if $\varphi \in C^0(\partial \boldsymbol D)$ satisfies  $\int\limits_0^{2\pi}\varphi(F(r_2e^{i\theta}))~ d\theta = \int\limits_0^{2\pi}\varphi(F(e^{i\theta}))~ d\theta$ and if $U$ is the solution of the Neumann problem \eqref{Neumann problem} on $\boldsymbol D$ with boundary data
		\begin{equation}\label{Phi}
		\Phi(w) = \begin{cases}|\nabla \omega(w)|\varphi(w), ~&\text{if } |G(w)|=r_2, \\ -|\nabla \omega(w)|\varphi(w), ~&\text{if } |G(w)|=1, \end{cases}
		\end{equation}
		then the solution $u$ of the Dirichlet problem \eqref{Neumann problem} on $\boldsymbol D$ with boundary values $\varphi$ is
		\begin{equation}\label{u doubly connected}
		u\left(w\right) = \frac{\langle \nabla U(w); \nabla\omega(w) \rangle}{|\nabla\omega(w)|^2},~ w \in \overline{\boldsymbol D}.
		\end{equation}
	\end{theorem}
	
	\begin{proof}
		For brevity the following notations will be adopted $V=U\circ F$, $\boldsymbol \nu: \partial \boldsymbol D \rightarrow \mathbb{R}$ is the outward normal derivative at $\boldsymbol D$, $\boldsymbol n: \partial \boldsymbol A_{1,r_2} \rightarrow \mathbb{R}$ is the outward normal derivative at $\boldsymbol A_{1,r_2}$. We have $\Phi(w) = \langle \nabla U(w),\boldsymbol \nu(w) \rangle$ on $\partial \boldsymbol D$ and notice that
		\begin{equation*}
		\boldsymbol \nu(F(z)) = \begin{cases}\frac{z F'(z)}{r_2|F'(z)|} = \frac{\nabla \omega(F(z))}{|\nabla \omega(F(z))|}~ &\text{if } |z|=r_2, \\ -\frac{z F'(z)}{r_1|F'(z)|}  = -\frac{\nabla \omega(F(z))}{|\nabla \omega(F(z))|}~ &\text{if } |z|=1, \end{cases}
		\end{equation*}
		from where it is obtained that
		\begin{equation*}
		\Phi(w) = \Re(\overline{\nabla U}(w)\boldsymbol \nu(w)), ~\forall w \in \partial \boldsymbol D.
		\end{equation*}
		Also $\nabla V(z) = \frac{\partial}{\partial x}(U(F(z))) + i\frac{\partial}{\partial y}(U(F(z)))$ and letting $w = F(z)$ compute successively
		\begin{align*}
		\frac{\partial}{\partial x}(U(F(z))) &= U_{\boldsymbol \xi}(w)\xi_{\boldsymbol x}(z) + U_{\boldsymbol \eta}(w)\eta_{\boldsymbol x}(z) = \Re\left( \overline{\nabla U}(F(z))F'(z) \right), \\
		\frac{\partial}{\partial y}(U(F(z))) &= U_{\boldsymbol \xi}(w)\xi_{\boldsymbol y}(z) + U_{\boldsymbol \eta}(w)\eta_{\boldsymbol y}(z) \\
		&\stackrel{\text{C-R equations}}{=}- U_{\boldsymbol \xi}(w)\eta_{\boldsymbol x}(z) + U_{\boldsymbol \eta}(w)\xi_{\boldsymbol x}(z)
		= -\Im\left( \overline{\nabla U}(F(z))F'(z) \right).
		\end{align*}
		In conclusion $\nabla V(z) = \Re\left(\overline{\nabla U}(F(z))F'(z)\right)-i~\Im\left(\overline{\nabla U}(F(z))F'(z)\right), ~z \in \boldsymbol A_{1;r_2}$, from where it follows by means of a continuity argument that
		\begin{equation*}
		\overline{\nabla V}(z) = \overline{\nabla U}(F(z))F'(z), ~z \in \overline{\boldsymbol A_{1;r_2}},
		\end{equation*}
		yielding $\frac{\partial V}{\partial \boldsymbol n}(z)\stackrel{|z|=r_2}{=} \Re\left(\overline{\nabla V}(z)\frac{z}{r_2}\right) = \Re( \overline{\nabla U}(F(z))F'(z)\frac{z}{r_2}) = \Re(\overline{\nabla U}(F(z))|F'(z)|$ \\
		$\cdot\boldsymbol \nu(F(z))) = |F'(z)|\Phi(F(z))$, and similarly one obtains $\frac{\partial V}{\partial \boldsymbol n}(z) \stackrel{|z|=1}{=} \Re\left(-z\overline{\nabla V}(z)\right)=$ \\
		$\Re\left(-z\overline{\nabla U}(F(z))F'(z)\right) = |F'(z)|\Re\left(\overline{\nabla U}(F(z))\boldsymbol \nu(F(z))\right)= |F'(z)|\Phi(F(z))$.
		To sum up
		\begin{equation}\label{Phi_V}
		\frac{\partial V}{\partial \boldsymbol n}(z) = \Phi(F(z))|F'(z)| = \Phi(F(z))|\nabla \omega(F(z))|^{-1}, ~\forall z \in \partial \boldsymbol A_{1;r_2},
		\end{equation}
		where the second equality is due to the relation $G' = pG$. But then defining $\Phi_V: \partial \boldsymbol A_{1;r_2} \rightarrow \mathbb{R}$, $\Phi_V = \frac{\partial V}{\partial \boldsymbol n}$, since $V$ is harmonic in $\boldsymbol A_{1;r_2}$ and $\nabla V$ can be extended by continuity to $\overline{\boldsymbol A_{1;r_2}}$, it follows that $V$ is a solution of the Neumann problem \eqref{Neumann problem} on $\boldsymbol A_{1;r_2}$ with boundary data $\Phi_V$. In addition $V(\sqrt{r_2})=U(F(\sqrt{r_2}))=0$.
		So applying Theorem \ref{main result cartesian} for $\Phi_V$ 
		\begin{equation}\label{V doubly connected}
		V(z) = \int\limits_{\frac{\sqrt{r_2}}{|z|}}^1\frac{v(\rho z)}{\rho}~ d\rho~ + \sqrt{r_2}\int\limits_0^{arg(z)}\left( \mathcal{C} - \int\limits_0^t\frac{\partial v}{\partial \boldsymbol a_{\tau}}(\sqrt{r_2}e^{i\tau}) d\tau \right) dt,
		\end{equation}
		where $\mathcal{C} = \frac{\sqrt{r_2}}{2\pi}\int\limits_0^{2\pi}\int\limits_0^t\frac{\partial v}{\partial \boldsymbol a_{\tau}}(\sqrt{r_2}e^{i\tau}) ~d\tau dt$, and where $v$ is the solution of the Dirichlet problem \eqref{Dirichlet problem} on $\boldsymbol A_{1;r_2}$ with boundary data $\varphi_V(z) = \begin{cases}r_2\Phi_V(z) &~\text{if } |z|=r_2, \\ -\Phi_V(z) &~\text{if } |z|=1. \end{cases}$ Consequently if $u = v \circ G$ then $u$ is harmonic in $\boldsymbol D$, extends continuously to $\overline{\boldsymbol D}$, and has continuous boundary data $\varphi_V \circ G$, which coincides with $\varphi$ given in the statement of the theorem. In addition denoting $\mathcal{G} = \{F(C_{\sqrt{r_2}})\}$ (the image through $F$ of $C_{\sqrt{r_2}}$) one obtains the following two relations
		\begin{align*}
		v_{\boldsymbol x}(G(w)) &= \langle \nabla u(w),F'(G(w)) \rangle = \Re\left(\overline{\nabla u}(w) F'(G(w))\right), \\
		v_{\boldsymbol y}(G(w)) &= \langle \nabla u(w),-\eta_{\boldsymbol x}(z) + i\xi_{\boldsymbol x}(z) \rangle = -\Im\left(\overline{\nabla u}(w) F'(G(w))\right), ~ ~ \forall w \in \mathcal{G}.
		\end{align*}
		
		To this end letting $z=\sqrt{r_2}e^{i\tau}$ it follows that $\frac{\partial v}{\partial \boldsymbol a_{\tau}}(\sqrt{r_2}e^{i\tau}) = \langle \nabla v(z), \frac{z}{\sqrt{r_2}} \rangle $ \\
		$=\frac{1}{\sqrt{r_2}}\Re\left(\overline{\nabla v}(z)z\right) = \Re\left(\overline{\nabla u}(w)F'(G(w))e^{i\tau}\right)$, $\tau \in \mathbb{R}$, where $w = F\left(\sqrt{r_2}e^{i\tau}\right)$. Combining this with the expression of $V$ given in \eqref{V doubly connected} it follows that $U=V\circ G$ has the desired expression \eqref{U doubly connected}. \\ 
		
		For the second part observe first that the assumption $\int\limits_0^{2\pi}\varphi(F(r_2e^{i\theta}))~ d\theta = \int\limits_0^{2\pi}\varphi(F(e^{i\theta}))~ d\theta$ implies $\int\limits_{\partial \boldsymbol D} \Phi~ d\sigma = 0$. Next putting $G(w) = z = re^{i\theta}$ and using the first part of the theorem one obtains 
		\begin{equation*}
		\frac{\partial}{\partial r}U(F(re^{i\theta})) = \frac{u(F(re^{i\theta}))}{r}.
		\end{equation*}
		Consequently compute $\frac{\partial}{\partial r}U(F(re^{i\theta})) = U_{\boldsymbol \xi}(w)\left(\frac{\partial}{\partial r}\xi(re^{i\theta})\right) + U_{\boldsymbol \eta}(w)\left(\frac{\partial}{\partial r}\eta(re^{i\theta})\right)$. But $\frac{\partial}{\partial r}\xi(re^{i\theta}) = \frac{1}{r} \Re \left(\frac{1}{p(w)}\right) = \frac{1}{r} \frac{\omega_{\boldsymbol \xi}(w)}{\omega_{\boldsymbol \xi}^2(w) + \omega_{\boldsymbol \eta}^2(w)}$, and similarly $\frac{\partial}{\partial r}\eta(re^{i\theta}) = \frac{1}{r} \Im \left(\frac{1}{p(w)}\right)$ \\ 
		$= \frac{1}{r} \frac{\omega_{\boldsymbol \eta}(w)}{\omega_{\boldsymbol \xi}^2(w) + \omega_{\boldsymbol \eta}^2(w)}$. To sum up $\frac{u(w)}{r} = \frac{\partial}{\partial r}U(F(re^{i\theta})) = \frac{1}{r}U_{\boldsymbol \xi}(w) \frac{\omega_{\boldsymbol \xi}(w)}{\omega_{\boldsymbol \xi}^2(w) + \omega_{\boldsymbol \eta}^2(w)} + \frac{1}{r}U_{\boldsymbol \eta}(w)\cdot$ \\
		$\frac{\omega_{\boldsymbol \eta}(w)}{\omega_{\boldsymbol \xi}^2(w) + \omega_{\boldsymbol \eta}^2(w)}$ which concludes the whole proof.
	\end{proof}
	
	Though this section is devoted to the case of doubly-connected regions, it will be ended with a result concerning the bounded simply-connected regions in the plane. More precisely Theorem 5 in \cite{Beznea1} will be completed with a result concerning the smooth extension of the higher order partial derivatives of a solution to the Neumann problem in the case of a smooth, bounded, simply-connected region $\boldsymbol D \subset \mathbb{C}, ~\boldsymbol D \neq \mathbb{C}$.
	
	\begin{theorem}\label{completion of Beznea1}
		Let $\boldsymbol D$ be a smooth, bounded, simply-connected region of the complex plane and let $f: \mathbb{U} \rightarrow \boldsymbol D$ be the conformal transformation of $\mathbb{U}$ onto $\boldsymbol D$ with $f(0)=w_0$ and $f'(0)>0$; define $g=f^{-1}$ and assume there is some positive integer $m \ge 2$ and some $\alpha \in (0,1)$ such that $\boldsymbol D \in C^{m+1,\alpha}$. If $\Phi \in C^{m,\alpha}(\partial \boldsymbol D)$ satisfies the compatibility condition $\int\limits_{\partial \boldsymbol D}\Phi ~d\sigma = 0$ and if $U$ is the solution of the Neumann problem \eqref{Neumann problem} on $\mathbb{U}$ with boundary data $f$, satisfying $U(w_0)=0$, then
		\begin{equation}
		U(w) = \int\limits_0^1\frac{u(f(\rho g(w)))}{\rho}d\rho, ~w \in \overline{\boldsymbol D},
		\end{equation} 
		where $u$ is the solution of the Dirichlet problem \eqref{Dirichlet problem} on $\mathbb{U}$ with boundary values $\varphi = \frac{\Phi}{|g'|}$. Moreover $U$ and all its partial derivatives up to order $m+1$ are uniformly H{\"o}lder continuous with exponent $\alpha$ on $\boldsymbol D$. Conversely if $\varphi \in C^0(\partial \boldsymbol D)$ satisfies $\int\limits_0^{2\pi} \varphi ~d\theta = 0$ and if $U$ is a solution of the Neumann problem \eqref{Neumann problem} on $\mathbb{U}$ with boundary data $\Phi = \varphi |g'|$, then the solution $u$ of the Dirichlet problem \eqref{Dirichlet problem} on $\mathbb{U}$ with boundary data $g$ is given by
		\begin{equation}
		u(w) = \Re\left(\overline{\nabla U}(w)\frac{g(w)}{g'(w)}\right), ~w \in \overline{\boldsymbol D}.
		\end{equation}
	\end{theorem}
	
	\begin{proof}
		For the first part, in light of \cite[Theorem 5]{Beznea1}, it only remains to prove that $U$ together with all its partial derivatives up to order $m+1$ are uniformly H{\"o}lder continuous with exponent $\alpha$ provided that $\Phi \in C^{m,\alpha}(\boldsymbol D), ~\boldsymbol D \in C^{m+1,\alpha}$. Again, to keep the derivations simple, the theorem will be proved only for the case $m=2$, as the general case $m \ge 2$ follows similarly using induction. To see that our claim for the case $m=2$ is indeed true notice that if $\boldsymbol D \in C^{3,\alpha}$ then in view of \cite[Chapter 3]{Pommerenke} $f^{(3)}$ is uniformly H{\"o}lder continuous with exponent $\alpha$ and moreover the continuous extension of $f'$ to $\overline{\mathbb{U}}$ does not vanish. Consequently $g'=\frac{1}{f'\circ g}$ extends continuously to $\overline{\boldsymbol D}$ and $g$ turns out to be uniformly Lipschitz continuous. Also it is easy to see that $f \in C^{3,\alpha}(\mathbb{U})$ and $\inf\limits_{z \in \mathbb{U}}|f'(z)| > 0$ imply that $g'$ as well as $g''$ and $g'''$ are uniformly H{\"o}lder continuous with exponent $\alpha$. On the other hand if $\Gamma$ ia any $C^{2,\alpha}$ parameterization of the unit circle then, assuming without loss of generality that $0 \not\in \partial \boldsymbol D$, $|f'\circ\Gamma|$ is readily seen to belong to $C^{2,\alpha}(\mathbb{R})$ and thus the function $\left(\Phi\circ f\right)|f'| \in C^{2,\alpha}(\partial \mathbb{U})$. Defining now $V = U\circ f$ it follows that $V$ is harmonic in $\mathbb{U}$ and $\overline{\nabla V}(z) = \overline{\nabla U}(f(z))f'(z)$. Thus $\nabla V$ extends continuously to $\overline{\mathbb{U}}$, which in turn shows that the normal derivative of $V$ is $\frac{\partial V}{\partial \boldsymbol n}(z) = \Re\left(z\overline{\nabla V}(z)\right) = \Re\left(z\overline{\nabla U}(f(z))f'(z)\right)=|f'(z)|\Re\left(\overline{\nabla U}(f(z))\nu(f(z))\right)=\Phi(f(z))|f'(z)| ~\forall z \in \partial \mathbb{U}$, where 
		\begin{equation*}
		\nu(w) = g(w)\frac{|g'(w)|}{g(w)}, ~w \in \partial \boldsymbol D
		\end{equation*}
		is the unitary outward pointing normal to $\partial \boldsymbol D$ at $w$ (see for instance \cite{Beznea1}). In conclusion $V$ is the solution of the Neumann problem \eqref{Neumann problem} on $\mathbb{U}$ with boundary data $\left(\Phi\circ f\right)|f'|$ and since the latter was proved to belong to $C^{2,\alpha}(\partial \mathbb{U})$ Corollary \ref{main corollary} guarantees that $V$ and all its partial derivatives up to order $m=3$ are uniformly H{\"o}lder continuous with exponent $\alpha$. Finally $\overline{\nabla U} = \left(\overline{\nabla V}\circ g\right)g'$ and since $\overline{\nabla U}$ and $\overline{\nabla V}$ are analytic functions in $\boldsymbol D$ and $\mathbb{U}$, respectively, taking the derivatives and using the \textit{Cauchy-Riemann} equations as well as the relations $\overline{\nabla U}' = U_{\boldsymbol \xi \boldsymbol \xi} - iU_{\boldsymbol \xi \boldsymbol \eta} = -U_{\boldsymbol \eta \boldsymbol \eta} - iU_{\boldsymbol \eta \boldsymbol \xi}$, $\overline{\nabla V}' = V_{\boldsymbol x \boldsymbol x} - iV_{\boldsymbol x \boldsymbol y}$, one obtains
		\begin{align*}
		\begin{cases}
		U_{\boldsymbol \xi}(w) = \Re\left(\overline{\nabla V}(g(w))g'(w)\right), \\
		U_{\boldsymbol \eta}(w) = -\Im\left(\overline{\nabla V}(g(w))g'(w)\right), \\
		U_{\boldsymbol \xi \boldsymbol \xi}(w) = \Re\left(\overline{\nabla V}'(g(w))(g'(w))^2 + \overline{\nabla V}(g(w))g''(w)\right), \\
		U_{\boldsymbol \xi \boldsymbol \eta}(w) = -\Im\left(\overline{\nabla V}'(g(w))(g'(w))^2 + \overline{\nabla V}(g(w))g''(w)\right), \\
		U_{\boldsymbol \eta \boldsymbol \eta}(w) = -U_{\boldsymbol \xi \boldsymbol \xi}(w) = -\Re\left(\overline{\nabla V}'(g(w))(g'(w))^2 + \overline{\nabla V}(g(w))g''(w)\right), \\	
		U_{\boldsymbol \xi \boldsymbol \xi \boldsymbol \xi}(w) = \Re\left(\mathcal{F}\right), ~U_{\boldsymbol \eta \boldsymbol \eta \boldsymbol \eta}(w) = \Im(\mathcal{F}),  \\
		U_{\boldsymbol \xi \boldsymbol \xi \boldsymbol \eta}(w) = U_{\boldsymbol \xi \boldsymbol \eta \boldsymbol \xi}(w) = U_{\boldsymbol \eta \boldsymbol \xi \boldsymbol \xi}(w) = -\Im(\mathcal{F}), \\
		U_{\boldsymbol \xi \boldsymbol \eta \boldsymbol \eta}(w) = U_{\boldsymbol \eta \boldsymbol \xi \boldsymbol \eta}(w) = U_{\boldsymbol \eta \boldsymbol \eta \boldsymbol \xi}(w) = -\Re(\mathcal{F}),
		\end{cases}
		\end{align*}
		with $\mathcal{F}(w)=\overline{\nabla V}''(g(w))(g'(w))^3+2\overline{\nabla V}'(g(w))g'(w)g''(w) + \overline{\nabla V}'(g(w))g'(w)g''(w)$ \\
		$+ \overline{\nabla V}(g(w))g'''(w)$ analytic. The theorem now follows using the properties of $\nabla V$ and $g$ discussed above, as well as using the triangle inequality in the above relations in a similar way it was done in the proof of Theorem \ref{smooth exension proposition doubly-connected}.
	\end{proof}
	
	\section{Conclusions \label{Conclusions}} The paper provides an equivalence between the solutions of the Neumann and the Dirichlet problems for planar, smooth, bounded, doubly-connected regions. This equivalence is expressed by the fact that solving any of these two problems leads by an analytic formula to an explicit solution of the other problem. As an application of this intimate connection the theory developed in this paper shows that under additional smoothness assumptions on the boundary of the region as well as on the boundary data, the higher order partial derivatives of the solutions of the Neumann problem are uniformly H{\"o}lder conditions. These assumptions are similar to those appearing in \textit{Kellogg's} theorem, where the problem of continuous extensions of the higher order partial derivatives for the solution of the Dirichlet problem was investigated. This fact comes to enhance the connection between the Dirichlet and the Neumann problems thus showing that for some types of regions of the complex plane these two problems should be studied simultaneously in a unified approach. In the end, a closer examination of the results reveals that the dependency between the solutions of the Dirichlet and the Neumann problems is more complex in the case of doubly-connected regions, and a natural question that arises is how does this connection look like for regions of connectivity greater than or equal to three and how could this be extended to regions of $\mathbb{R}^d$ for $d \ge 3$? A positive answer might help us dive deeper into this intimate connection     
	and obtain new interesting results regarding these fundamental problems.

\newpage

\section{APPENDIX}

This section is entirely devoted to the connection between the solution of the Dirichlet problem and a particular solution of the Neumann problem, in the case of elliptic regions. Although the elliptic regions are obviously planar, smooth, bounded, simply-connected regions for which the connection between the two problems has been explicitly provided (see \cite[Theorem 5]{Beznea1} or \cite[Theorem 5]{Beznea2}), the conformal mapping on which this connection relies on is somewhat cumbersome, thus making the representation of the Neumann problem in terms of the solution of the Dirichlet problem somehow redundant for a direct application. This issue is fixed in the paper at hand by considering another approach for obtaining the desired connection. This approach is based on the Joukowsky transform. \\

Let $\boldsymbol J: \mathbb{C}^* \rightarrow \mathbb{C}, ~\boldsymbol J(w) = \frac{1}{2}\left(w + \frac{1}{w}\right)$ be the Joukowsky transform, and define $\boldsymbol J_{+} = \boldsymbol J_{|_{\mathbb{U}^c\setminus \{e^{i\theta}|\theta\in (-\pi,0)\}}}$, $\boldsymbol T_{+} = \boldsymbol J_{+}^{-1}$, $\boldsymbol J_{-} = \boldsymbol J_{|_{\dot{\mathbb{U}} \cup \{e^{i\theta}|\theta\in [-\pi,0]\}}}$, $\boldsymbol T_{-} = \boldsymbol J_{-}^{-1}$. Throughout this section, the argument of a complex number will be defined as taking values in $(-\pi,\pi]$ and if $z$ is any complex number then its square root will be defined as 
\begin{equation*}
\sqrt{z} = \begin{cases}|z|e^{i\frac{arg(z)}{2}} &~\text{if } z \neq 0, \\ 0 &~\text{if } z=0. \end{cases}
\end{equation*}
If $\rho>1$ let $\boldsymbol E_{\rho}$ be the interior of the ellipse with equation $\frac{4x^2}{\left(\rho+\rho^{-1}\right)^2}+\frac{4y^2}{\left(\rho-\rho^{-1}\right)^2}=1$ and for $\theta \in (-\pi,\pi]$ the hyperbola described by $\frac{x^2}{\cos^2\theta}-\frac{y^2}{\sin^2\theta}=1$ will be denoted $H_{\theta}$. Finally we let $\mathcal{H} = \{H_{\theta}|\theta\in(-\pi,\pi]\}$. It is not difficult to notice that $\mathcal{H}$ is actually the set of all hyperbolas orthogonal to the family of confocal ellipses having foci $\{(-1,0),(1,0)\}$. This aspect will play an essential role in the proof of Theorem \ref{Main Result for Elliptical Regions} below.

\begin{lemma}\label{Lemma_az}
	$\boldsymbol T_+$ has  the following properties. 
	\begin{enumerate}
		\item It is well defined on the whole $\mathbb{C}$; \\
		\item It is analytic in $\mathbb{C} \setminus [-1,1]$ with nonzero derivative; \\
		\item For any point $\xi \in [-1,1]$ and any sequence $\{z_n\}_{n=1}^{\infty}$ satisfying 
		\begin{enumerate}
			\item[i.] there exist $H_{\theta} \in \mathcal{H}$ for which $z_n \in \mathcal{H_{\theta}} \cap \{\Im(z) > 0\} ~\forall ~n \in \mathbb{N}^*$,
			\item[ii.] $z_n \rightarrow \xi$,
		\end{enumerate}
		one has $\lim\limits_{n \rightarrow \infty} \boldsymbol T_+(z_n) = \boldsymbol T_+(\xi)$; \\
		\item $\boldsymbol T_+(z) = z + \sqrt{z^2 - 1}, ~z \in \mathbb{C}$.
	\end{enumerate}
\end{lemma} 

\begin{proof} 
	\begin{enumerate}
		\item It is not difficult to see that $\boldsymbol J_+\left(\mathbb{U}^c\setminus \{e^{i\theta}:\theta\in (-\pi,0)\}\right) = \mathbb{C}$ and that $\boldsymbol J_+$ is invertible there (see for instance \cite[Chapter 4.2]{Ahlfors}). This property was actually used in the definition of $\boldsymbol T_+$.
		\item Indeed choose any point $z \in \mathbb{C} \setminus [-1,1]$ and let $w = \boldsymbol T_+(z)$. As the set $\boldsymbol T_+\left(\mathbb{C} \setminus [-1,1]\right) = \{\zeta \in \mathbb{C}: |\zeta| > 1\}$ is open and  $\boldsymbol J_+'(w) = \boldsymbol J'(w) \neq 0$  it follows that there is an open neighborhood $U_w$ of $w$ in $\{\zeta \in \mathbb{C}: |\zeta| > 1\}$ such that $\boldsymbol J_+'(u) \neq 0 \ \forall u \in U_w$. Define $V_z = \boldsymbol J_+(U_w)$. Then $V_z$ is an open subset of $\mathbb{C}$ and we have $\boldsymbol T_+'(l)=\frac{1}{\boldsymbol J_+'(\boldsymbol T_+(l))} \in \mathbb{C}^* ~\forall ~l \in V_z$. 
		\item Since $z_n \in \mathcal{H_{\theta}} \cap \{\Im(z) > 0\}$ it follows that $z_n = \frac{1}{2}\left(\rho_n + \frac{1}{\rho_n}\right)\cos\theta + \frac{i}{2}\left(\rho_n - \frac{1}{\rho_n}\right)\sin\theta$ for some $\rho_n > 1$ and some $\theta \in [0,\pi]$. On the other hand since $z_n \rightarrow \xi \in [-1,1]$ it follows that $\rho_n \rightarrow 1$ and thus $\xi = \cos\theta$. One can also deduce that $\boldsymbol T_+(z_n) = \rho_ne^{i\theta}$ and thus $\lim\limits_{n \rightarrow \infty} \boldsymbol T_+(z_n) = e^{i\theta}$. But $\boldsymbol J_+\left(e^{i\theta}\right) = \frac{1}{2}\left(e^{i\theta} + e^{-i\theta}\right) = \cos\theta = \xi$ and so $\boldsymbol T_+(z_n) \rightarrow e^{i\theta} = \boldsymbol T_+(\xi)$.
		\item See \cite[Chapter 3]{Ahlfors}. 
	\end{enumerate}
\end{proof}

\begin{theorem}\label{Main Result for Elliptical Regions}
	Fix some $\rho>1$ and assume that $f \in C^0(\partial \boldsymbol E_{\rho})$ satisfies the compatibility condition $\int\limits_{\partial \boldsymbol E_{\rho}}f d\sigma = 0$. If $U$ is the solution of the Neumann problem \eqref{Neumann problem} on $\textbf{E}_{\rho}$ having boundary data $f$ and satisfying $U(1)=0$, then letting $R(z)e^{i\Theta(z)} := \boldsymbol T_+(z), ~\Theta(z) \in (-\pi,\pi]$, one has
	\begin{equation}\label{formula for ellipse}
	U(z) = \int\limits_{\frac{1}{R(z)}}^1 \frac{u(t\boldsymbol T_{+}(z))}{t}dt
	- \Theta(z) \int\limits_0^1\int\limits_0^{t\Theta(z)}\frac{\partial u}{\partial \boldsymbol a_{\tau}}(e^{i\tau})d\tau dt, ~ z \in \boldsymbol E_{\rho},
	\end{equation}
	where $u$ is the solution of the Dirichlet problem on $\{\frac{1}{\rho}<|w|<\rho\}$ with boundary values
	\begin{align*} 
	\varphi_{-} &= \frac{f\circ \boldsymbol J}{\rho |\boldsymbol T_{-}'\circ \boldsymbol J|} ~\text{on} ~C_{\frac{1}{\rho}}, \\
	\varphi_{+} &= \rho\frac{f\circ \boldsymbol J}{|\boldsymbol T_{+}'\circ \boldsymbol J|} ~\text{on} ~C_{\rho}.
	\end{align*}
\end{theorem}

\begin{proof}
	Let $U:\boldsymbol E_{\rho} \rightarrow \mathbb{R}$ be as in the statement of the theorem. Define $V = U  \circ \boldsymbol J_{|_{\overline{\boldsymbol A_{\rho^{-1};\rho}}}}$ and notice that $V$ thus obtained is harmonic on $\boldsymbol A_{\frac{1}{\rho};\rho}$ and also $\frac{\partial V}{\partial \boldsymbol \nu}(w) = V_{\boldsymbol \xi}(w)\frac{\Re(w)}{\rho} + V_{\boldsymbol \eta}(w)\frac{\Im(w)}{\rho},~ \forall~ w \in C_{\rho}$, where $w =: \xi + i\eta$ and $\boldsymbol \nu$ is the unitary outward pointing normal to $\partial \boldsymbol A_{\frac{1}{\rho};\rho}$. Defining
	\begin{equation*}
	\omega(w) = V_{\boldsymbol \xi}(w) - iV_{\boldsymbol \eta}(w) = \overline{\nabla V}(w), ~\{\frac{1}{\rho}<|w|<\rho \},
	\end{equation*} 
	one obtains alternatively $\frac{\partial V}{\partial \boldsymbol \nu}(w) = \Re\left( \omega(w)\frac{w}{\rho} \right), ~\forall~ w \in C_{\rho}$. Using the Cauchy-Riemann equations together with the harmonicity of $V$ it follows that $\omega$ is analytic on $\{\frac{1}{\rho}<|w|<\rho\}$. Furthermore $\frac{\partial V}{\partial \boldsymbol \nu}(w) = \Re \left(\frac{\omega(w)w}{\rho}\right), w \in C_{\rho}$. On the other hand let $G$ be an analytic function such that $U= \Re(G)$ on $\boldsymbol E_{\rho}$ (which is always possible since $U$ is harmonic and $\boldsymbol E_{\rho}$ is a simply connected region). But then setting $F = G \circ \boldsymbol J_{|_{\overline{\boldsymbol A_{\rho^{-1};\rho}}}}$ one obtains $V = \Re(F)$. Consequently it follows that $F'=\omega$ on $\{\frac{1}{\rho}<|w|<\rho\}$, which gives $F(w) = F(w_0) + \int\limits_{w_0}^w \omega(o)~ do$, or equivalently
	
	\begin{equation}\label{G}
	G(J(w)) = G(J(w_0)) + \int\limits_{w_0}^w \omega(o)~ do,
	\end{equation}
	where $w_0=w_0(w)$ is to be specified later on. Notice now that any $o \in \{1<|w|<\rho\}$ is of the form $\boldsymbol T_{+}(\lambda)$ for some $\lambda \in \boldsymbol E_{\rho} \setminus [-1,1]$, and hence according to point $(2)$ of Lemma \ref{Lemma_az}
	$o'(\lambda)= \boldsymbol T_{+}'(\lambda) = \frac{2o^2(\lambda)}{o^2(\lambda)-1}$ and consequently
	\begin{equation*}
	G(z) = G(z_0) + \int\limits_{z_0}^z \omega(\boldsymbol T_+(\lambda))\boldsymbol T_{+}'(\lambda)d\lambda = G(z_0) + 2\int\limits_{z_0}^z \omega(\boldsymbol T_{+}(\lambda))\frac{\boldsymbol T_{+}^2(\lambda)}{\boldsymbol T_{+}^2(\lambda)-1}d\lambda,
	\end{equation*}
	where $z_0 = \boldsymbol J(w_0)$. Most of the remaining part of the proof will be divided into several steps.
	\begin{enumerate}
		\item[Step 1.]
		Denote $w = \boldsymbol T_+(z)$ for any $w \in \boldsymbol A_{1;\rho}$ and consequently define the curve
		\begin{equation*}
		\gamma_{\epsilon}^w(t) := \begin{cases} 
		& wt,~ t \in \left[\frac{1+\epsilon}{|w|},1\right]; \\
		& (1+\epsilon)\exp\left( i\arg(w) \frac{t-\epsilon}{\frac{1+\epsilon}{|w|}-\epsilon} \right), ~t \in \left[\epsilon, \frac{1+\epsilon}{|w|}\right); \\
		& 1+t,~ t \in \left[\frac{\epsilon}{2},\epsilon \right), \end{cases}
		\end{equation*} 
		from where it follows that
		\begin{equation*}
		\dot{\gamma}_{\epsilon}^w(t) = \begin{cases} 
		& w,~ t \in \left(\frac{1+\epsilon}{|w|},1\right); \\
		& i\frac{\arg(w)}{\frac{1+\epsilon}{|w|}-\epsilon}(1+\epsilon)\exp\left( i\arg(w) \frac{t-\epsilon}{\frac{1+\epsilon}{|w|}-\epsilon} \right),~ t \in \left(\epsilon, \frac{1+\epsilon}{|w|}\right); \\
		& 1,~ t \in \left(\frac{\epsilon}{2},\epsilon \right)\end{cases}
		\end{equation*} 
		for any $w \in \{1<|w|\le \rho \}$ and any $\epsilon > 0$ small enough. Now define $\lambda_{\epsilon}^z(t) = \boldsymbol J_+(\gamma_{\epsilon}^w(t))$ which gives $\dot{\lambda}_{\epsilon}^z(t) = \frac{\gamma_{\epsilon}^w(t)}{\boldsymbol T_+'(\lambda_{\epsilon}^z(t))}$. Also $\lim\limits_{\epsilon \rightarrow 0}G\left(\boldsymbol J_+\left(1+\frac{\epsilon}{2}\right)\right) = \lim\limits_{\epsilon \rightarrow 0}G\left(\boldsymbol J\left(1+\frac{\epsilon}{2}\right)\right) = G(0)$. To this end setting $z_0=z_0(\epsilon) = \boldsymbol J_+\left(1+\frac{\epsilon}{2}\right)$ in \eqref{G} one obtains
		\begin{align*}
		G(z) = G\left(\boldsymbol J_+\left(1+\frac{\epsilon}{2}\right) \right) &+  \int\limits_{\frac{1+\epsilon}{|w|}}^1 \frac{tw\omega(tw)}{t} ~dt + \int\limits_{\frac{\epsilon}{2}}^{\epsilon}\omega(1+t) ~dt \\ 
		&+i\frac{\arg(w)}{\frac{1+\epsilon}{|w|}-\epsilon}\int\limits_{\epsilon}^{\frac{1+\epsilon}{|w|}}\omega(\gamma_{\epsilon}^w(t))\gamma_{\epsilon}^w(t)~ dt, ~\epsilon > 0. 
		\end{align*}
		Setting $\epsilon \rightarrow 0$ it follows by the use of the \textit{Dominant Convergence} theorem that
		$G(z) = G(1) + \int\limits_{\frac{1}{R(z)}}^1 \frac{t\boldsymbol T_+(z)\omega\left( t\boldsymbol T_+(z) \right)}{t}~ dt + iR(z)\Theta(z) \int\limits_0^{\frac{1}{R(z)}} \omega\left( \gamma^w(t) \right)\gamma^w(t)~ dt$, where $\gamma^w = \lim\limits_{\epsilon \rightarrow 0} \gamma_{\epsilon}^w$. Taking the real part in the equation above it follows that
		\begin{equation}\label{first half of the expression}
		U(z) = \int\limits_{\frac{1}{R(z)}}^1 \frac{\Re\left[ t\boldsymbol T_+(z)\omega\left( t\boldsymbol T_+(z) \right)\right]}{t}~ dt - R(z)\Theta(z) \int\limits_0^{\frac{1}{R(z)}}\Im\left[ \omega\left( \gamma^w(t) \right)\gamma^w(t) \right]~ dt,
		\end{equation}	
		where the normalization $U(1)=0$ was used.
		
		\item[Step 2.] Link the first integral term in \eqref{first half of the expression} to the solution of some Dirichlet problem on $\boldsymbol A_{\frac{1}{\rho};\rho}$. To this end evaluate first the corresponding boundary values on $C_{\rho}$ using the function $\boldsymbol T_+$ and some family of curves $\Gamma_{+,\epsilon}^w$, and second the corresponding boundary values on $C_{\frac{1}{\rho}}$ using the function $\boldsymbol T_-$ and some family of curves $\Gamma_{-,\epsilon}^w$, respectively. More precisely define $\Gamma_{+,\epsilon}^w: \left[\frac{1}{|w|}+\epsilon,1\right] \rightarrow \{1<|w|\le \rho \}, ~ \Gamma_{+,\epsilon}^w(t) = tw$ for any $w \in \{1<|w|\le \rho \}$ and any $\epsilon > 0$ small enough. Thus $\dot{\Gamma}_{+,\epsilon}^w(t) = w$ and defining $\Lambda_{+,\epsilon}^z(t) = \boldsymbol J_+(\Gamma_{+,\epsilon}^w(t))$ it follows that the image of $\Lambda_{+,\epsilon}^z$ is a branch of some hyperbola in $\mathcal{H}$ orthogonal to $\partial \boldsymbol E_{\rho} $ which approaches some $z_0^{*}=z_0^{*}(z) \in [-1,1]$ as $\epsilon \rightarrow 0$. Also notice that $\dot{\Lambda}_{+,\epsilon}^z(t) = \boldsymbol J_+'(\Gamma_{+,\epsilon}^w(t))w$ and thus $\frac{d}{dt}\left(\Lambda_{+,\epsilon}^z(t)\right) = \frac{T_{+}(z)}{T_{+}'(\Lambda_{+,\epsilon}^z(t))}$. Proceeding further observe that
		\begin{equation}\label{egalitate 1}
		\omega\left( \Lambda_{+,\epsilon}^z(t) \right) \boldsymbol T_{+}'(\Lambda_{+,\epsilon}^z(t))\dot{\Lambda}_{+,\epsilon}^z(t) = \omega \left(\boldsymbol T_{+}(\Lambda_{+,\epsilon}^t(t)) \right)w = \omega\left( tw \right)w,
		\end{equation}
		and also
		\begin{equation}\label{egalitate 2}
		\omega\left( \Lambda_{+,\epsilon}^z(t) \right) \boldsymbol T_{+}'(\Lambda_{+,\epsilon}^z(t))\dot{\Lambda}_{+,\epsilon}^z(t) = \overline{\nabla U}(\Lambda_{+,\epsilon}^z(t))\dot{\Lambda}_{+,\epsilon}^z(t),
		\end{equation} 
		both \eqref{egalitate 1} and \eqref{egalitate 2} being true for any $\epsilon > 0, ~w \in \{1<|w|<\rho\}, ~t \in \left[\frac{1}{|w|}+\epsilon,1\right]$, where in the derivation of \eqref{egalitate 2} the following important observation was made 
		\begin{equation}\label{nabla 1}
		\overline{\nabla U}(z) = \omega \left(\boldsymbol T_{+}(z)\right) \cdot \boldsymbol T_{+}'(z), ~\forall~ z \in \boldsymbol E_{\rho} \setminus [-1,1].
		\end{equation}
		Combining now \eqref{egalitate 1} and \eqref{egalitate 2} one obtains
		\begin{equation*}
		\Re\left[ \omega\left( tw \right) w \right] = \langle \nabla U(\Lambda_{+,\epsilon}^z(t));\dot{\Lambda}_{+,\epsilon}^z(t) \rangle,  
		\end{equation*}
		whenever $\epsilon > 0, ~w \in \{1 <|w|\le \rho\}, ~t \in \left[\frac{1}{|w|}+\epsilon,1\right]$ and choosing any $z^{*} \in \partial \boldsymbol E_{\rho}$
		\begin{equation}\label{boundary equality 1}
		\Re\left[ \omega\left(\boldsymbol T_{+}(z^{*}) \right)\boldsymbol T_{+}(z^{*} ) \right] = \langle \nabla U(z^{*}); \dot{\Lambda}_+^{z^{*}}(1) \rangle,
		\end{equation}
		where $\dot{\Lambda}_+^{z^{*}}(1)$ is thus the outward normal derivative in $z^{*}$ at $\partial \boldsymbol E_{\rho}$. Consequently compute $|\dot{\Lambda}_+^{z^{*}}(1)| = \frac{|\boldsymbol T_{+}(z^*)|}{| \boldsymbol T_{+}'(z^{*})|} = \frac{\rho}{| \boldsymbol T_{+}'(z^{*})|}$ which shows, using relation \eqref{boundary equality 1}, that
		\begin{equation*}
		\Re\left[ \omega \left(\boldsymbol T_{+}(z^{*}) \right) \boldsymbol T_{+}(z^{*} ) \right] = \langle \nabla U(z^{*}); \frac{\dot{\Lambda}_+^{z^{*}}(1)}{|\dot{\Lambda}_+^{z^{*}}(1)|} \rangle |\dot{\Lambda}_+^{z^{*}}(1)| = f(z^{*})\frac{\rho}{| \boldsymbol T_{+}'(z^{*})|},
		\end{equation*}
		or equivalently
		\begin{equation}\label{final form of the boundary condition 1}
		\Re \left[ \omega \left( w_{+}^{*} \right) w_{+}^{*} \right] = \frac{\rho f( \boldsymbol J_+(w_{+}^{*}))}{| \boldsymbol T_{+}'( \boldsymbol J_+(w_{+}^{*}))|} = \rho f( \boldsymbol J_+(w_{+}^{*})) | \boldsymbol J_+'(w_{+}^{*})|, ~ w_{+}^{*} \in C_{\rho},
		\end{equation}
		where $w_{+}^{*} = \boldsymbol T_{+}(z^{*})$. On the other hand $U(z) = V(\boldsymbol T_{-}(z)) ~ \forall z \in \overline{\boldsymbol E}_{\rho}$ which gives (exactly as it was done for the previous case) $\overline{\nabla U}(z) = \omega \left( \boldsymbol T_{-}(z) \right) \boldsymbol T_{-}'(z)$, $\forall z \in \boldsymbol E_{\rho} \setminus [-1,1]$. In the same way define $\Gamma_{-,\epsilon}^w: \left[1,\frac{1}{|w|+\epsilon}\right] \rightarrow \{\frac{1}{\rho} \le |w| < 1\}, ~ \Gamma_{-,\epsilon}^w(t) = tw$ for any $w \in \{\frac{1}{\rho} \le |w| < 1\}$ and any $\epsilon > 0$ small enough. Thus $\dot{\Gamma}_{-,\epsilon}^w(t) = w$ and by defining $\Lambda_{-,\epsilon}^z(t) := \boldsymbol J_-(\Gamma_{-,\epsilon}^w(t))$ it follows that the image of $\Lambda_{-,\epsilon}^z$ is also a branch of some hyperbola  in $\mathcal{H}$ orthogonal to $\partial \boldsymbol E_{\rho}$ which approaches some $z_0^{*}=z_0^{*}(z) \in [-1,1]$ as $\epsilon \rightarrow 0$. Also notice that $\dot{\Lambda}_{-,\epsilon}^z(t) = \boldsymbol J_-'(\Gamma_{-,\epsilon}^w(t))w$ and hence $\frac{d}{dt}\left(\Lambda_{-,\epsilon}^z(t)\right) = \frac{ \boldsymbol T_{-}(z)}{ \boldsymbol T_{-}'(\Lambda_{-,\epsilon}^z(t))}$. But then one obtains, similarly as for \eqref{egalitate 1} and \eqref{egalitate 2}
		\begin{equation}\label{egalitate 3}
		\omega \left( \boldsymbol T_{-}'(\Lambda_{-,\epsilon}^z(t)) \right) \boldsymbol T_{-}'(\Lambda_{-,\epsilon}^z(t))\dot{\Lambda}_{-,\epsilon}^z(t) = \omega\left( \boldsymbol T_{-}(\Lambda_{-,\epsilon}^z(t)) \right)w = \omega \left(tw \right)w,
		\end{equation}
		and also 
		\begin{equation}\label{egalitate 4}
		\omega \left( \boldsymbol T_{-}'(\Lambda_{-,\epsilon}^z(t)) \right) \boldsymbol T_{-}'(\Lambda_{-,\epsilon}^z(t))\dot{\Lambda}_{-,\epsilon}^z(t) = \overline{\nabla U}(\Lambda_{-,\epsilon}(t))\dot{\Lambda}_{-,\epsilon}^z(t),
		\end{equation} 
		both \eqref{egalitate 3} and \eqref{egalitate 4} being true for any $\epsilon > 0, ~ w \in \{\frac{1}{\rho} \le |w| < 1\}$ and any $ t \in \left[\frac{1}{|w|+\epsilon},1\right]$, respectively. Combining \eqref{egalitate 3} and \eqref{egalitate 4} it follows that
		\begin{equation*}
		\Re\left[ \omega \left( tw \right)w \right] = \langle \nabla U(\Lambda_{-,\epsilon}^z(t));\dot{\Lambda}_{-,\epsilon}^z(t) \rangle, 
		\end{equation*}
		$\epsilon > 0, ~ w \in \{\frac{1}{\rho} \le |w|< 1\}, ~ t \in \left[\frac{1}{|w|+\epsilon},1\right]$ and choosing any $z^{*} \in \partial \boldsymbol E_{\rho}$
		\begin{equation}\label{boundary equality 2}
		\Re \left[ \omega \left( \boldsymbol T_{-}(z^{*}) \right) \boldsymbol T_{-}(z^{*}) \right] = \langle \nabla U(z^{*}); \dot{\Lambda}_-^{z^{*}}(1) \rangle,
		\end{equation}
		where $\dot{\Lambda}_-^{z^{*}}(1)$ is thus the outward normal derivative in $z^{*}$ at $\partial \boldsymbol E_{\rho}$. Consequently compute $|\dot{\Lambda}_-^{z^{*}}(1)| = \frac{| \boldsymbol T_{-}(z^*)|}{| \boldsymbol T_{-}'(z^{*})|} = \frac{1}{\rho| \boldsymbol T_{-}'(z^{*})|}$ which shows, using relation \eqref{boundary equality 2}, that
		\begin{equation*}
		\Re \left[ \omega \left( \boldsymbol T_{-}(z^{*}) \right) \boldsymbol T_{-}(z^{*} ) \right] = \langle \nabla U(z^{*}); \frac{\dot{\Lambda}_-^{z^{*}}(1)}{|\dot{\Lambda}_-^{z^{*}}(1)|} \rangle |\dot{\Lambda}_-^{z^{*}}(1)| = \frac{f(z^{*})}{\rho| \boldsymbol T_{-}'(z^{*})|}, ~ z^{*} \in \partial \boldsymbol E_{\rho},
		\end{equation*}
		or equivalently
		\begin{equation}\label{final form of the boundary condition 2}
		\Re \left[ \omega \left( w_{-}^{*})w_{-}^{*} \right) \right] =  \frac{f \left( \boldsymbol J_-(w_{-}^{*}) \right)}{\rho| \boldsymbol T_{-}'( \boldsymbol J_-(w_{-}^{*}))|} = \frac{f( \boldsymbol J_-(w_{-}^{*})) | \boldsymbol J_-'(w_{-}^{*})|}{\rho}, ~ w_{-}^{*} \in C_{\frac{1}{\rho}},
		\end{equation}
		where $w_{-}^{*} = \boldsymbol T_{-}(z^{*})$. Finally, using equation \eqref{first half of the expression}, equations \eqref{final form of the boundary condition 1} and \eqref{final form of the boundary condition 2} together with the analyticity of $w\omega(w)$ on $ \boldsymbol A_{\frac{1}{\rho};\rho}$, the continuity (and hence boundedness) of $u(w) := \Re\left[ w\omega(w) \right]$ on $\overline{ \boldsymbol A_{\frac{1}{\rho};\rho}}$, and the uniqueness of the solution to the Dirichlet problem, it follows that $u$ is the solution of the Dirichlet problem \eqref{Dirichlet problem} on $ \boldsymbol A_{\frac{1}{\rho};\rho}$ with boundary data $\varphi_- \circ \boldsymbol J$ on $C_{\frac{1}{\rho}}$ and $\varphi_+ \circ \boldsymbol J$ on $C_{\rho}$, respectively. To sum up, it has been shown so far that 
		\begin{equation}\label{end of first half of the expression}
		U(z) = \int\limits_{\frac{1}{R(z)}}^1 \frac{u(t \boldsymbol T_+(z))}{t}~ dt - R(z)\Theta(z) \int\limits_0^{\frac{1}{R(z)}}\Im\left[ \omega \left( \gamma^w(t) \right)\gamma^w(t) \right]~ dt,
		\end{equation}
		where $u$ is the solution of the Dirichlet problem on $\{\frac{1}{\rho}<|w|<\rho\}$ with boundary values
		\begin{align*} 
		\varphi_{-} :&= \frac{f\circ \boldsymbol J}{\rho | \boldsymbol T_{-}'\circ J|} ~\text{on } C_{\frac{1}{\rho}}, \\
		\varphi_{+} :&= \rho\frac{f\circ \boldsymbol J}{| \boldsymbol T_{+}'\circ \boldsymbol J|} \ \text{on}  \ C_{\rho}.
		\end{align*}
		
		\item[Step 3.] Link the second integral in \eqref{end of first half of the expression} to $u$. To do so notice that since $w\omega(w)$ is an analytic function on $\boldsymbol A_{\frac{1}{\rho};\rho}$ which extends continuously to $\overline{ \boldsymbol A_{\frac{1}{\rho};\rho}}$ it follows that $\Im \left[ w\omega(w) \right]$ is a harmonic function on $\boldsymbol A_{\frac{1}{\rho};\rho}$ which extends continuously to $\overline{ \boldsymbol A_{\frac{1}{\rho};\rho}}$. Letting $w\omega(w) =: u(w) + iv(w)$ it follows that $v$ is a harmonic function in $ \boldsymbol A_{\frac{1}{\rho};\rho}$ which extends continuously to $\overline{ \boldsymbol A_{\frac{1}{\rho};\rho}}$ and the idea is to determine $v$ from $u$. Using \eqref{nabla 1}
		\begin{equation*}
		\omega \left( \boldsymbol T_+(z) \right) = \frac{\overline{\nabla U}(z)}{ \boldsymbol T_+'(z)}, ~ z \in \boldsymbol E_{\rho} \setminus [-1,1].
		\end{equation*} 
		Define the sequence $z_n = 1 + \frac{1}{n}$ for any $n$ large enough so that $z_n \in \boldsymbol E_{\rho}$. Then $ \boldsymbol T_+(z_n) \in \boldsymbol A_{1;\rho} \cap \mathbb{R}_+$ and by point (3.) of Lemma \ref{Lemma_az} it follows that $ \boldsymbol T_+(z_n) \rightarrow \boldsymbol T_+(1)=1$. In addition $ \boldsymbol T_+'(z_n)$ is well defined by point $(2.)$ of the same lemma and using point $(4.)$ of the same result it follows that $ \boldsymbol T_+'(z_n) = 1 + \frac{z}{\sqrt{z_n^2-1}}$ which shows that $| \boldsymbol T_+'(z_n)| \ge \frac{|z_n|}{|\sqrt{z_n^2-1}|}-1 > \frac{1}{|\sqrt{z_n^2-1}|}-1$ for $n$ large. Since $|\sqrt{z_n^2-1}| \rightarrow 0$ it follows that $| \boldsymbol T_+'(z_n)| \ge \frac{1}{2|\sqrt{z_n^2-1}|} \rightarrow \infty$. But $U \in C^1(\boldsymbol E_{\rho})$ and since $(1,0) \in \boldsymbol E_{\rho}$, $\boldsymbol E_{\rho}$ open, it follows that there is some neighborhood of $(1,0)$ contained in $\boldsymbol E_{\rho}$ on which $|\overline{\nabla U}| \le M$, for some $M > 0$. Consequently, it follows that $\exists N \in \mathbb{N}^*$ such that $\forall n \ge N$ one has $|\omega \left( \boldsymbol T_+(z_n) \right)| = \frac{|\overline{\nabla U}(z_n)|}{| \boldsymbol T_+'(z_n)|} \le \frac{M}{| \boldsymbol T_+'(z_n)|} \rightarrow 0$. To sum up $\omega(w_n) \rightarrow 0$ as $w_n \rightarrow 1, ~ w_n \in \boldsymbol A_{1;\rho} \cap \mathbb{R}_+$. Since $\omega$ is continuous on $ \boldsymbol A_{\frac{1}{\rho};\rho} \supset \boldsymbol A_{1;\rho} \Rightarrow \omega(1) = \lim\limits_{w_n \rightarrow 1, ~ w_n \in \boldsymbol A_{1;\rho} \cap \mathbb{R}_+}\omega(w_n) = 0$. This gives $v(1)=0$. Since $u$ and $v$ are conjugate-harmonic functions and $v(1)=0$, it follows that one can precisely determine $v$ solely from $u$. Indeed using the Cauchy-Riemann equations it follows that $v(a,b) = \int\limits_{\gamma}~ dv  = \int\limits_0^{\arg(a+ib)}\frac{\partial u}{\partial \boldsymbol a_{\theta}}(e^{i\theta})~ d\theta$ whenever $a+ib \in \gamma$, where $\gamma$ is considered to be the curve $e^{it}$ for $t \in [0,\arg(a+ib)]$. Hence
		\begin{equation}\label{second half of the expression}
		\int\limits_0^{\frac{1}{R(z)}} \Im \left[ \omega \left( \gamma^w(t) \right) \gamma^w(t) \right] ~dt = \int\limits_0^{\frac{1}{R(z)}} \int\limits_0^{tR(z)\Theta(z)}\frac{\partial u}{\partial \boldsymbol a_{\tau}}( e^{i\tau}) ~d\tau dt,
		\end{equation}		
		so combining equation \eqref{end of first half of the expression} with equation \eqref{second half of the expression} and using a change of variable, Theorem \ref{Main Result for Elliptical Regions} is proved for any $z \in \boldsymbol E_{\rho} \setminus [-1,1]$.
	\end{enumerate}

	If $z \in [-1,1]$ choose any sequence $\{z_n\}_{n=1}^{\infty}$ as in point $(3)$ of Lemma \ref{Lemma_az} such that $z_n \rightarrow z$. Then using the same point of Lemma \ref{Lemma_az}, a continuity argument for $U$, the \textit{Dominant Convergence} theorem, as well as the boundedness of $\frac{\partial u}{\partial \boldsymbol a_{\tau}}$ on $C_1$, the proof is completed.
\end{proof}

\begin{remark}
The proof of Theorem \ref{Main Result for Elliptical Regions} provides in addition an interesting interpretation of the second term in the right-hand side of equality \eqref{formula for ellipse}. Indeed 
\begin{equation}
\Theta(z) \int\limits_0^1\int\limits_0^{t\Theta(z)}\frac{\partial u}{\partial \boldsymbol a_{\tau}}(e^{i\tau})d\tau dt = -U(z_0^*(z)), ~z \in \boldsymbol E_{\rho} \setminus [-1,1],
\end{equation}
where $z_0^*(z)$ is the intersection of the (unique) hyperbola $H_{\theta}$ containing $z$ with the line segment $[-1,1]$.
\end{remark}

\end{document}